\documentclass[12pt]{amsart}

\usepackage{amsfonts,amsthm,amsmath,amssymb,amscd,mathrsfs}
\usepackage{graphics}
\usepackage{indentfirst}
\usepackage{cite}
\usepackage{latexsym}
\usepackage[dvips]{epsfig}

\newtheorem{theorem}{Theorem}[section]
\newtheorem{remark}{Remark}[section]
\newtheorem{definition}{Definition}[section]
\newtheorem{lemma}[theorem]{Lemma}
\newtheorem{pro}[theorem]{Proposition}

\newenvironment{pf}{{\noindent \it \bf Proof:}}{{\hfill$\Box$}\\}

\renewcommand{\div}{{\rm div }}

\newcommand{\bt}{\begin{theorem}}
\newcommand{\bl}{\begin{lemma}}
\newcommand{\el}{\end{lemma}}
\newcommand{\et}{\end{theorem}}

\newcommand{\la}{\label}

\newcommand{\bn}{\begin{eqnarray}}
\newcommand{\en}{\end{eqnarray}}
\newcommand{\bnn}{\begin{eqnarray*}}
\newcommand{\enn}{\end{eqnarray*}}

\newcommand{\ba}{\begin{aligned}}
\newcommand{\ea}{\end{aligned}}
\newcommand{\be}{\begin{equation}}
\newcommand{\ee}{\end{equation}}

\renewcommand{\la}{\label}

\newcommand{\Bv}{{\boldsymbol{v}}}

\newcommand{\Bn}{{\boldsymbol{n}}}

\newcommand{\Bu}{{\boldsymbol{u}}}
\newcommand{\Be}{{\boldsymbol{e}}}
\newcommand{\BF}{{\boldsymbol{F}}}

\newcommand{\supp}{\text{supp}}
\newcommand{\bBU}{\bar{{\boldsymbol{U}}}}

\newcommand{\Bo}{{\boldsymbol{\omega}}}
\newcommand{\Bp}{{\boldsymbol{\phi}}}
\newcommand{\msR}{\mathscr{R}}
\newcommand{\mcE}{\mathcal{E}}
\newcommand{\mcA}{\mathcal{A}}
\newcommand{\mcH}{\mathcal{H}}
\newcommand{\what}{\widehat}
\newcommand{\veps}{\varepsilon}

\begin{document}

\title[Stability of Hagen-Poiseuille Flows]
{Uniform Structural stability of Hagen-Poiseuille flows in a pipe}

\author{Yun Wang}
\address{School of Mathematical Sciences, Center for dynamical systems and differential equations, Soochow University, Suzhou, China}
\email{ywang3@suda.edu.cn}

\author{Chunjing Xie}
\address{School of mathematical Sciences, Institute of Natural Sciences,
Ministry of Education Key Laboratory of Scientific and Engineering Computing,
and SHL-MAC, Shanghai Jiao Tong University, 800 Dongchuan Road, Shanghai, China}
\email{cjxie@sjtu.edu.cn}

\begin{abstract}
In this paper, we prove the uniform nonlinear structural stability of Hagen-Poiseuille flows with arbitrary large fluxes in the axisymmetric case.  This uniform nonlinear structural stability is the first step to study Liouville type theorem for steady solution of Navier-Stokes system in a pipe, which may play an important role in proving the existence of solutions for the  Leray's problem, the existence of solutions of steady Navier-Stokes system with arbitrary flux in a general nozzle. A key step to establish nonlinear structural stability is the a priori estimate for the associated linearized problem for Navier-Stokes system around Hagen-Poiseuille flows.
The linear structural stability is established as a consequence of elaborate analysis for the governing equation for the partial Fourier transform of the stream function. The uniform estimates are obtained based on the analysis for the solutions with different fluxes and frequencies.  One of the most involved cases is to analyze the solutions with large flux and intermediate frequency, where the boundary layer analysis for the solutions plays a crucial role.
\end{abstract}



\maketitle

\section{Introduction and Main Results}
An important problem in fluid dynamics is to study the steady Navier-Stokes system in infinitely long nozzles (\cite{Galdi}). More precisely, given an infinitely long nozzle $\tilde{\Omega}$ tending to a flat cylinder $\Omega=\Sigma\times \mathbb{R}$, the famous Leray's problem is to establish the existence of the solutions for the system
\begin{equation}\label{SNS}
\left\{
\begin{aligned}
& \Bu\cdot \nabla \Bu +\nabla p=\Delta \Bu +\BF, \\
&\div~\Bu=0,
\end{aligned}
\right.
\end{equation}
 supplemented with the conditions
 \begin{equation}\label{fluxBC0}
\Bu= 0\ \ \ \ \mbox{on}\ \partial \tilde{\Omega},\quad \quad \int_{\tilde{\Sigma}} \Bu\cdot \Bn dS =\Phi,
\end{equation}
and the velocity field tending to the Poiseuille flow associated with the flux $\Phi$, where $\tilde{\Sigma}$ is any cross section of the nozzle $\tilde\Omega$ and the Poiseuille flow which is a solution steady Navier-Stokes equations of the form $\bar\Bu=(0, 0, \bar u^z(x,y)$.

 When $\Omega=B_1(0)\times \mathbb{R}$, the velocity field for the Poiseuille flow  with the flux $\Phi$ in this case has the explicit form
\be \label{1.8}
(0, 0, \bar{U}(r)) = (0, 0, \frac{2\Phi}{\pi}(1-r^2)),
\ee
where $r=\sqrt{x^2+y^2}$.
  These solutions are also called Hagen-Poiseuille flows and We denote them by $\bBU =\bar{U}(r) \Be_z$ later on.  The Leary's problem nowadays and it was first addressed by Leray (\cite{Leray}) in 1933. Without loss of generality,  $\Phi$ is assumed to be nonnegative.
  The first significant contribution to the solvability of Leray's problem is due to Amick \cite{Amick1, Amick2}, who reduced the proof of existence to the resolution of a well-known variational problem related to the stability of the Poiseuille flow in a flat cylinder. However, Amick left out the investigation of uniqueness and the existence of solutions with large flux. A rich and detailed analysis of the problem of the flow in domains having outlets to infinity of bounded cross section is due to Ladyzhenskaya and Solonnikov \cite{LS}. However, the asymptotic far field behavior of the solutions obtained in \cite{LS} is not very clear.
There are lots of studies on well-posedness for the Leray's problem and far field behavior for the associated solutions, one may refer to \cite{AmickF,MF, Rabier1, Rabier2, HW, LS, AP,Pileckas}, etc.
For more references on steady solutions of the Navier-Stokes equation in nozzles or other type of domains, please refer to the book by Galdi \cite{Galdi}. A significant open problem posed in \cite[p. 19]{Galdi} is global well-posedness for Leray's problem in a general nozzle when the flux $\Phi$ is large.

In fact, with the aid of the compactness obtained in \cite{LS},  global well-posedness for the Leary's problem in a general nozzle tending to a pipe could be established even when the flux $\Phi$ is large, provided that we can prove global uniqueness of Hagen-Poiseuille flow in a pipe. As a first step to prove global uniqueness of Hagen-Poiseuille flows, we investigate the local uniqueness of Hagen-Poiseuille flows. More precisely,  under the boundary condition
 \begin{equation}\label{fluxBC}
\Bu= 0\ \ \ \ \mbox{on}\ \partial \Omega,\quad \quad \int_\Sigma \Bu\cdot \Bn dS =\Phi,
\end{equation}
boundary conditions where $\Sigma$ is any cross section of the pipe, is Hagen-Poiseuille flow the only solution of steady Navier-Stokes system \eqref{SNS} in the neighborhood of such flow?

A key issue to prove the local uniqueness of Hagen-Poiseuille flows is to investigate the invertibility of associated linearized operator around Hagen-Poiseuille flows. More precisely, we need to study the well-posedness for the system
\be \label{2-0-1} \left\{ \ba
&\bBU \cdot \nabla \Bv + \Bv \cdot \nabla \bBU - \Delta \Bv + \nabla P = \BF, \ \ \ \mbox{in}\ \Omega, \\
& {\rm div}~\Bv = 0,\ \ \ \mbox{in}\ \Omega,
\ea
\right.
\ee
supplemented
with no-slip boundary conditions and the flux constraint,
\be \label{BC}
\Bv = 0\ \ \ \mbox{on}\ \partial \Omega,\quad \int_{B_1(0)} v^z(\cdot, \cdot, z)\, dS = 0\ \ \ \text{for any}\,\, z \in \mathbb{R}.
\ee


Our first main result is the following uniform linear structural stability of Hagen-Poiseuille flows, which plays a crucial role in proving local uniqueness.

\bt \label{thm1}
Assume that $\BF= \BF(r, z) \in L^2(\Omega)$  is axisymmetric, the linear problem \eqref{2-0-1} and \eqref{BC} has a unique axisymmetric solution $\Bv$ which satisfies
\be \label{estuniformlinear}
\|\Bv\|_{H^{\frac53}(\Omega)} \leq C \|\BF\|_{L^2 (\Omega)},
\ee
and
\be \label{estimatelinear}
\|\Bv\|_{ H^2 (\Omega)} \leq C (1 + \Phi^{\frac14} ) \|\BF\|_{L^2 (\Omega)},
\ee
where $C$ is a uniform constant independent of $\BF$ and $\Phi$.
\et

Several remarks with respect to Theorem \ref{thm1} are in order.

\begin{remark}
The key point of Theorem \ref{thm1} is that the constant $C$ in \eqref{estuniformlinear} does not depend on $\Phi$, which provides a uniform estimate for the axisymmetric solutions of \eqref{2-0-1} and \eqref{BC}.
\end{remark}

\begin{remark}
Here we would like to mention that our results also shed lights on hydrodynamical stability of Hagen-Poiseuille flows in an infinitely long pipe, which is a longstanding problem (\cite{Asen}).
One of the key issues  is to investigate the following eigenvalue problem
\be \label{eigpb} \left\{ \ba
&s\Bv+\bar\Bu \cdot \nabla \Bv + \Bv \cdot \nabla \bar\Bu - \frac{1}{Re}\Delta \Bv + \nabla P = \BF, \ \ \ \mbox{in}\ \Omega, \\
& {\rm div}~\Bv = 0,\ \ \ \mbox{in}\ \Omega,
\ea
\right.
\ee
supplemented
with the condition \eqref{BC}, where $\bar{\Bu}=(1 - r^2) \Be_z$.
The computations by Meseguer and Trefethen \cite{MT} indicate that the $L^2$-norm of the resolvent operator $\msR(s)$, the solution operator for \eqref{eigpb} and \eqref{BC}, is maximized at $s=0$ and depends on the Reynolds number $Re$ as $\|\msR(0)\|\sim Re^2$.
Given
$\tilde{\Bv} =\msR(0)\tilde{\BF}$,
then  $\Bv=\tilde{\Bv}$ satisfies \eqref{2-0-1}-\eqref{BC}  with ${\BF}= {Re}\tilde{\BF}$ and $\Phi=\frac{  \pi Re}{2} $. Therefore, it follows from Theorem \ref{thm1} that
\begin{equation}
\|\tilde{\Bv}\|_{L^2(\Omega)}= \|{\Bv}\|_{L^2(\Omega)}\leq C \|\BF\|_{L^2(\Omega)} = C Re \|\tilde{\BF}\|_{L^2(\Omega)}.
\end{equation}
This implies that
\[
\|\msR(0)\|_{L^2\to L^2}\leq C Re.
\]
Therefore, Theorem \ref{thm1} not only answers the problem left by Asen and Kreiss \cite[p. 461]{Asen} about the estimate for the resolvent operator $\msR(0)$ in the axisymmetric case, but also shows that the norm estimate for the resolvent computed in  \cite{MT} can be improved in the axisymmetric setting. The linear stability of Poiseuille flows in a pipe for the solutions periodic in $z$-direction was obtained in \cite{Guo, Zhang}.
\end{remark}

Making use of the uniform estimates for the linear problem, we have the following results on uniform nonlinear structural stability.
\bt \label{mainthm}
Assume that $\BF= \BF(r, z)\in L^2 (\Omega)$ is axisymmetric. There exists a  constant $\veps_0$, independent of $\BF$ and $\Phi$, such that if
\be \label{thmuniformnonlinear1}
\|\BF\|_{L^2  (\Omega)} \leq \veps_0,
\ee
then the steady Navier-Stokes system \eqref{SNS} supplemented with the boundary condition \eqref{fluxBC} has a unique axisymmetric solution
$\Bu$
satisfying the estimate
\be \label{thmuniformnonlinear2}
\|\Bu - \bBU \|_{H^{\frac53} (\Omega)} \leq C \|\BF\|_{L^2(\Omega)},
\ee
and
\be \label{thmuniformnonlinear3}
\|\Bu - \bBU \|_{H^2(\Omega)} \leq C (1 + \Phi^{\frac14}) \|\BF\|_{L^2(\Omega)}.
\ee
Here $C$ is a uniform constant independent of $\Phi$.
\et

There are a few remarks in order.

\begin{remark}
Theorem \ref{mainthm} gives the unique existence of steady solutions for the Navier-Stokes system near the Hagen-Poiseuille flow. The key point of Theorem \ref{mainthm} is that the flux of the flow can be arbitrarily large, and the constant $C$ in \eqref{thmuniformnonlinear2} is independent of the flux. So Theorem \ref{mainthm} provides uniform structural stability of Hagen-Poiseuille flows with respect to the flux.  This is the main difference between Theorem \ref{mainthm} and the results in previous work where the flux must be small or  satisfy certain relation with the viscosity, and the magnitude of the force depends on the viscosity,  \cite{Amick1, Amick2, LS}.
\end{remark}

\begin{remark}
In the forthcoming paper \cite{WX1}, we will show the results similar to Theorems \ref{thm1} and \ref{mainthm} are also true for axisymmetric flows satisfying Navier slip boundary conditions. Inspired by the method in this paper,  we proved the structural stability for Poiseuille flows in a two dimensional strip with Navier slip boundary conditions \cite{WX2}.
\end{remark}

\begin{remark}
In fact, using the ideas in this paper, we can even prove the existence of solutions for nonlinear Navier-Stokes equations when $\BF$ is big in certain spaces, see \cite{WX3} for these improved estimates and the asymptotic behavior for the solutions of Navier-Stokes system in a pipe.
\end{remark}

Here we outline some key observations and techniques in this paper. First, it is observed that  for the axisymmetric solutions to the linearized system around Hagen-Poiseuille flows, the equation for the swirl velocity is decoupled from the ones for other unknowns. And the major part of the analysis is for the axial and radial velocity components, which is  equivalent to a fourth order equation for the stream function. The main difficulty for the linear structural stability corresponds to the analysis for a non-self adjoint operator with large coefficients. After taking Fourier transform with respect to the axial variable, the linearized problem for the stream function is reduced to a fourth order complex ODE with frequency as a parameter. We deal with all frequencies simultaneously. The estimate derived from the imaginary part, together with an elementary inequality  (\eqref{ineqHLP} in Lemma \ref{lemmaHLP}) used to deal with the degeneracy of the flow near the boundary, gives the estimate for the real part, which consists the basic a priori estimate for the linearized problem. However,  these estimates do not give uniform stability. In order to deal with the flow with large flux, we analyze the problem by different methods when the frequencies are different.  The most involved situation comes from the case with large flux and intermediate frequency. Inspired by the study on two dimensional fast rotating flows in \cite{M}, we decompose the solutions of the linearized axisymmetric Navier-Stokes system in a pipe into four parts. The first part is an axisymmetric solution of the same linearized system supplemented with the slip boundary conditions, while the rest parts correspond to the boundary layer corrections of the solutions. More precisely, the boundary layer corrections contain three parts, one of them is an exact solution of the boundary layer equations, the other part is an exact irrotational solution of the linearized Navier-Stokes system which associates with the modified Bessel function of the first kind,  while the last one corresponds to remainders after we take care of all the above leading parts.

The organization of the paper is as follows. In Section \ref{Linear},  the stream function formulation for the  linear problem \eqref{2-0-1} and \eqref{BC} in the axisymmetric case is established and some basic a priori estimates for the stream function are given.  The uniform a priori estimate of the swirl velocity is established in Section \ref{sec-swirl}.  The existence and regularity of solutions are provided in Sections \ref{sec-ex} and \ref{sec-reg}, respectively.       The uniform a priori  estimate for the case with large flux is obtained in Section \ref{sec-res}. The estimate is established via different analysis for the problem with different frequencies. With the aid of the analysis on the associated linearized problem and a fixed point argument, the uniform nonlinear structural stability of Hagen-Poiseuille flows in axisymmetric case is proved in Section \ref{secnonlinear}.  Two appendices are included. The first one collects some important lemmas and their proofs which are used in the paper. The detailed analysis for the vorticity is given in the second appendix.


\section{Stream function formulation and a priori estimate}\label{Linear}
This section devotes to the study on the linearized problem \eqref{2-0-1} and \eqref{BC} for axisymmetric solutions. We will make full use of the fact that the equation for the swirl velocity decouples from the equations for  radial and axial velocity. After introducing the stream function, one can reduce the equations for the axial and radial velocity into a single fourth order equation.
 Taking Fourier transform with respect to the axial variable $z$  induces a fourth order complex ODE for the stream function, where the frequency is regarded as a parameter. These are presented in \S \ref{sec-stream}. The careful energy estimate for both the imaginary and the real parts of the complex ODE gives a good estimate for the solutions for all the frequencies simultaneously. Although these estimates are not uniform with respect to the fluxes, they are enough for getting  the existence of solutions for the associated linear problem for the stream function.

\subsection{Stream function formulation}\label{sec-stream}
In terms of the cylindrical coordinates, an axisymmetric solution $\Bv$ can be written as
\be \nonumber
\Bv = v^r(r,z) \Be_r + v^z(r,z) \Be_z +v^\theta(r,z) \Be_\theta.
\ee
Then the equations \eqref{2-0-1} become
\be \label{2-0-1-1}
\left\{
\ba
& \bar U(r)  \frac{\partial v^r}{\partial z} + \frac{\partial P}{\partial r} -\left[ \frac{1}{r} \frac{\partial}{\partial r}\left(
r \frac{\partial v^r}{\partial r} \right) + \frac{\partial^2 v^r}{\partial z^2} - \frac{v^r}{r^2} \right] = F^r  \ \ \ \mbox{in}\ D ,\\
& v^r \frac{\partial \bar U }{\partial r} + \bar U(r)  \frac{\partial v^z}{\partial z} + \frac{\partial P}{\partial z}
- \left[ \frac{1}{r} \frac{\partial }{\partial r} \left( r \frac{\partial v^z}{\partial r}\right) + \frac{\partial^2 v^z}{\partial z^2} \right] = F^z \ \ \mbox{in}\ D ,                     \\
& \partial_r v^r + \partial_z v^z + \frac{v^r}{r} =0\ \ \ \mbox{in}\ D ,
\ea \right.
\ee
and
\be \label{vswirl}
\bar U(r)  \partial_z  v^\theta - \left[ \frac{1}{r} \frac{\partial }{\partial r} \left( r \frac{\partial v^\theta}{\partial r}\right) + \frac{\partial^2 v^\theta}{\partial z^2} - \frac{v^\theta}{r^2} \right] =  F^\theta, \ \ \ \mbox{in}\ \ D.
\ee
Here $F^r$, $F^z$, and $F^\theta$ are the radial, axial, and azimuthal component of $\BF$, respectively, and $D=\{(r, z): r\in (0, 1), z\in \mathbb{R}\}$.
The Dirichlet boundary conditions and the flux constraint \eqref{BC} can be written as
\be \label{BC-1}
v^r(1, z)  = v^z(1, z) = 0,\quad  \int_0^1 r v^z(r, z)\, dr = 0,
\ee
and
\begin{equation}\label{BC-swirl}
v^\theta (1, z) = 0.
\end{equation}

It follows from  the third equation in \eqref{2-0-1-1} that there exists a stream function $\psi(r, z)$ satisfying
\be \label{2-0-4}
v^r =  \partial_z \psi \ \ \text{and} \ \ v^z = - \frac{\partial_r (r \psi) }{r}.
\ee
Then the azimuthal vorticities of $\Bv$ and $\BF$ are written as
\be \nonumber
\omega^\theta= \partial_z v^r - \partial_r v^z= \frac{\partial }{\partial r}  \left(  \frac1r \frac{\partial }{\partial r} (r \psi) \right) + \partial_z^2 \psi \ \ \ \ \mbox{and}\ \ \ \ f= \partial_z F^r - \partial_r F^z,
\ee
respectively.
It follows from the first two equations in \eqref{2-0-1-1} that
\be \label{2-0-2}
\bar U(r)  \partial_z \omega^\theta - \left(\partial_r^2 + \partial_z^2 + \frac{1}{r}  \partial_r \right)\omega^\theta
+ \frac{\omega^\theta}{r^2}  = f.
\ee
Denote
\be \nonumber
\mathcal{L} = \frac{\partial}{\partial r} \left(  \frac1r \frac{\partial}{\partial r}(r \cdot)      \right) = \frac{\partial^2}{\partial r^2} + \frac1r \frac{\partial}{\partial r} - \frac{1}{r^2}.
\ee
Then $\psi$ satisfies the following fourth order equation,
\be \label{2-0-4-1}
\bar U(r)  \partial_z( \mathcal{L} + \partial_z^2) \psi -
( \mathcal{L} + \partial_z^2)^2 \psi = f.
\ee

 Next, we derive the boundary conditions for $\psi$. As discussed in \cite{Liu-Wang}, in order to get classical solutions, some compatibility conditions at the axis should be imposed. Assume that $\Bv$ and the vorticity $\Bo$ are continuous, $v^r(0, z)$ and $\omega^\theta(0, z)$ should vanish, which implies that
\be \nonumber
\partial_z \psi(0, z) = (\mathcal{L} + \partial_z^2)\psi (0, z) = 0.
\ee
Without loss of generality, one can  assume that $\psi(0, z) = 0$. Hence, the following compatibility condition holds at the axis,
\be \label{2-0-4-2}
\psi(0, z) = \mathcal{L} \psi(0, z) = 0.
\ee
On the other hand, it follows from \eqref{BC-1} that
\be \nonumber
\int_0^1 \partial_r (r \psi ) (r, z)\, dr =-  \int_0^1 r v^z \, dr = 0.
\ee
This, together with \eqref{2-0-4-2}, gives
\be \label{2-0-4-3}
\psi(1, z) = \lim_{r\rightarrow 0+ } ( r \psi)  (r, z)  = 0.
\ee
Moreover, according to the Dirichlet boundary condition for $\Bv$,
\be \nonumber
\frac{\partial}{\partial r}(r \psi)  |_{r=1} = r v^z  |_{r=1} = 0,
\ee
which implies that
\be \label{2-0-4-4}
\frac{\partial \psi}{\partial r}(1, z) = 0.
\ee

In this paper, for a given function $g(r, z)$, define its Fourier transform with respect to $z$ variable by
\be  \nonumber
\hat{g}(r, \xi) = \int_{\mathbb{R}^1} g(r, z) e^{-i \xi z} dz.
\ee
We take the Fourier transform with respect to $z$ for the equation \eqref{2-0-4-1}. For each fixed $\xi$, $\hat{\psi}$ satisfies
\be \label{2-0-8}
i \xi \bar{U}(r) ( \mathcal{L} -\xi^2) \hat{\psi} - ( \mathcal{L} -\xi^2)^2 \hat{\psi} = \hat{f}.
\ee
Furthermore,
the boundary conditions \eqref{2-0-4-2}-\eqref{2-0-4-4} can be written as
\be \label{FBC}
 \hat{\psi}(0) = \hat{\psi} (1) = \hat{\psi}^{\prime} (1) = 0\quad \text{and}\quad  \mathcal{L} \hat{\psi}(0) = 0.
\ee

\subsection{ A priori estimates for the stream function}\label{sec-apri}
In this subsection, we derive some a priori estimates for the linear problem \eqref{2-0-8}-\eqref{FBC}, which guarantee the existence of solutions. The estimates consist in the following two lemmas.

\begin{lemma}\label{lemapri1}
Let $\hat{\psi}(r, \xi)$ be a smooth solution of the problem \eqref{2-0-8}--\eqref{FBC}. then one has
\be \la{lapr1}
\int_0^1 | \mathcal{L} \psi|^2 r \, dr + \xi^2 \int_0^1 \left| \frac{d}{dr} ( r\hat{\psi} )  \right|^2  \frac1r \, dr
+ \xi^4 \int_0^1 \left|  \hat{\psi} \right|^2  r \, dr
\leq C \int_0^1 |\hat{f}|^2 r \, dr.
\ee
\end{lemma}

\begin{pf}
 Multiplying \eqref{2-0-8} by $r \overline{\hat{\psi}}$ (here and later on $\overline{\hat{\psi}}$ denots the complex conjugate of $\hat{\psi}$) and integrating the resulting equation over $[0, 1]$ yield
\be \label{2-1-1}
\int_0^1 \left[ i \xi \bar U(r) (\mathcal{L} - \xi^2) \hat{\psi}- (\mathcal{L} -\xi^2)^2  \hat{\psi}        \right] \overline{\hat{\psi}} r  \,  dr
= \int_0^1 \hat{f}  \overline{\hat{\psi}}  r \, dr.
\ee

For the first  term on the left hand  of \eqref{2-1-1}, it follows from integration by parts and the homogeneous boundary conditions \eqref{FBC} for $\hat{\psi}$ that
\be \la{2-1-3} \ba
& \int_0^1 i \xi \bar U(r)  ( \mathcal{L}  -\xi^2) \hat{\psi}  \overline{\hat{\psi}}  r \, dr \\
= \,\,& i \xi \int_0^1 \bar{U}(r)  \frac{d}{dr} \left(  \frac1r \frac{d}{dr} (r\hat{\psi})     \right)  r \overline{\hat{\psi}} \, dr
-  i \xi^3 \int_0^1 \bar U(r) |\hat{\psi}|^2  r \, dr \\
= \,\,& i \frac{4\Phi}{\pi} \xi \int_0^1 \frac{d}{dr} (r \hat{\psi} )  r \overline{\hat{\psi}}\, dr
- i \xi \int_0^1 \frac{\bar{U}(r)}{r} \left| \frac{d}{dr} (r \hat{\psi})   \right|^2 \, dr - i \xi^3 \int_0^1 \bar{U}(r) |\hat{\psi}|^2  r \, dr.
\ea
\ee
While for the second term on the left hand  of \eqref{2-1-1}, one has
\be \la{2-1-2} \ba
& \int_0^1 ( \mathcal{L}  - \xi^2)^2 \hat{\psi}  \overline{\hat{\psi}}  r \, dr \\
 =\,\, & \int_0^1 \frac{d}{dr} \left(  \frac1r \frac{d}{dr} ( r \mathcal{L} \hat{\psi})      \right)  \overline{\hat{\psi}}  r \, dr
- 2\xi^2 \int_0^1 \frac{d}{dr} \left( \frac1r \frac{d}{dr} (r\hat{\psi})    \right)  \overline{\hat{\psi}}  r \, dr
+ \xi^4 \int_0^1 |\hat{\psi}|^2  r\, dr \\
 =\,\,& \int_0^1 | \mathcal{L} \hat{\psi} |^2  r \, dr + 2\xi^2 \int_0^1 \left|  \frac{d}{dr} (r \hat{\psi})  \right|^2  \frac1r \, dr
+ \xi^4 \int_0^1 |\hat{\psi}|^2  r \, dr.
\ea \ee

From now on, we denote $\Im g$ and $\Re g$ by the imaginary and real part of  $g$ (a function or a number), respectively.
It follows from  \eqref{2-1-1}-\eqref{2-1-2} that
\be \la{2-1-5} \ba
& \int_0^1 | \mathcal{L} \hat{\psi} |^2  r \, dr
+ 2\xi^2 \int_0^1 \left|  \frac{d}{dr} (r \hat{\psi})   \right|^2  \frac1r \, dr
+ \xi^4 \int_0^1 |\hat{\psi}|^2  r \, dr \\
 =\,\, &- \Re \int_0^1 \hat{f} \overline{\hat{\psi}}  r \, dr - \frac{4 \Phi}{\pi} \xi \Im \int_0^1  \left[ \frac{d}{dr} (r \hat{\psi} )  r \overline{\hat{\psi}}      \right] \, dr,
\ea \ee
and
\be \la{2-1-6} \ba
& - \frac{4 \Phi}{\pi} \xi  \Re \int_0^1   \left[ \frac{d}{dr} ( r\hat{\psi} )  r \overline{ \hat{\psi} } \right] \, dr
+ \xi \int_0^1 \frac{\bar{U}(r) }{ r } \left| \frac{d}{dr} ( r \hat{\psi} )  \right|^2 \, dr
+ \xi^3 \int_0^1  \bar{U}(r) |\hat{\psi}|^2  r\, dr \\
=\,\,& - \Im \int_0^1 \hat{f}  \overline{\hat{\psi} }  r \, dr.
\ea \ee
Note that the homogeneous boundary conditions for $\hat{\psi}$ implies
\be \nonumber
\Re \int_0^1 \frac{d}{dr} (r \hat{\psi} ) r \overline{\hat{\psi} } \, dr  = 0.
\ee
Hence the expression  \eqref{2-1-6} can be rewritten as
\be \la{2-1-7}
 \xi \int_0^1 \frac{\bar{U}(r) }{ r } \left| \frac{d}{dr} ( r \hat{\psi} )   \right|^2 \, dr
+ \xi^3 \int_0^1  \bar{U}(r) |\hat{\psi}|^2  r\, dr
= - \Im \int_0^1 \hat{f} \overline{\hat{\psi} }  r \, dr .
\ee
Substituting the explicit expression of $\bar{U}(r)$ into the above equation yields
\be \label{2-1-8}
\Phi |\xi| \int_0^1  \left| \frac{d}{dr} (r \hat{\psi} )   \right|^2  \frac{1-r^2}{r} \, dr +
\Phi |\xi|^3 \int_0^1  \left| \hat{\psi}   \right|^2  r(1-r^2) \, dr
\leq C \int_0^1 |\hat{f} \hat{\psi} |  r \, dr .
\ee
By Lemma \ref{lemmaHLP},  \eqref{2-1-8}, together with Cauchy-Schwarz inequality, implies that
\be \label{2-1-8-1}
\Phi^2 \xi^2 \int_0^1 |\hat{\psi}|^2 r \, dr \leq C \int_0^1 |\hat{f}|^2 r\, dr.
\ee

Now we estimate the second term on the right hand side of  \eqref{2-1-5}. By Lemma \ref{lemma1} and the inequality  \eqref{2-1-8-1}, one has
\be \label{2-1-9} \ba
  \left| \frac{4 \Phi}{\pi} \xi \int_0^1 \frac{d}{dr} ( r \hat{\psi} )  r \overline{\hat{\psi} }  \, dr          \right|
\leq  & \frac14 \int_0^1 \left|  \frac{d}{dr} (r \hat{\psi})   \right|^2 \frac1r \, dr + C \Phi^2 \xi^2 \int_0^1 |\hat{\psi}|^2 r \, dr
\\
\leq & \frac14 \int_0^1 |\mathcal{L} \hat{\psi} |^2 r\, dr + C \int_0^1 |\hat{f}|^2 r \, dr .
\ea \ee

Substituting  \eqref{2-1-9} into \eqref{2-1-5} and using Cauchy-Schwarz inequality yield  \eqref{lapr1}.
This finishes the proof of Lemma \ref{lemapri1}.
\end{pf}


Using the similar idea as in the proof of  Lemma \ref{lemapri1}, one has the following  higher order a priori estimates.
\begin{lemma}\label{lemhapr}
Let $\hat{\psi}$ be a smooth solution of the problem \eqref{2-0-8}--\eqref{FBC}. Then it holds that
\begin{equation}\label{hapr1}
\begin{aligned}
&\,\, \xi^4 \int_0^1 \left| \mathcal{L} \hat{\psi} \right|^2  r  \, dr
+ \xi^6 \int_0^1 \left| \frac{d}{dr} ( r\hat{\psi} ) \right|^2  \frac1r \,  dr +\xi^8 \int_0^1 \left|  \hat{\psi}   \right|^2  r  \, dr
\leq C (1+\Phi) \int_0^1 |\hat{f}|^2  r \, dr
\end{aligned}
\end{equation}
and
\begin{equation}\label{hapr2}
\int_0^1 |\mathcal{L}^2 \hat{\psi} |^2  r \, dr \leq C (1+ \Phi^2) \int_0^1 |\hat{f}|^2  r \,  dr .
\end{equation}
\end{lemma}

\begin{pf}
 Multiplying \eqref{2-1-5} by $\xi^2$ gives
\be \label{2-2-1}
\ba
&\xi^2 \int_0^1 \left| \mathcal{L} \hat{\psi}    \right|^2  r  \, dr + 2\xi^4 \int_0^1 \left|  \frac{d}{dr} (r \hat{\psi} )  \right|^2  \frac1r \, dr
+ \xi^6 \int_0^1 \left| \hat{\psi}  \right|^2  r  \, dr\\
 \leq \,\,&  \xi^2 \int_0^1 |\hat{f}| |\hat{\psi} | r \, dr + C \Phi |\xi|^3 \int_0^1 \left| \frac{d}{dr}(r \hat{\psi}) \right| |r \hat{\psi}| \, dr \\
\leq \, \, & \xi^2 \int_0^1 |\hat{f}| |\hat{\psi} | r \, dr + \xi^4 \int_0^1 \left| \frac{d}{dr}( r \hat{\psi})   \right|^2 \frac1r \, dr + C \Phi^2 \xi^2  \int_0^1 |  \hat{\psi} |^2 r   \, dr .
\ea
\ee
By Lemma \ref{lemma1},  Young's inequality and  the inequality \eqref{2-1-8-1}, one has
\be \label{2-2-2}
\xi^2 \int_0^1 \left| \mathcal{L} \hat{\psi}    \right|^2  r  \, dr + \xi^4 \int_0^1 \left|  \frac{d}{dr} (r \hat{\psi} )  \right|^2  \frac1r \, dr
+ \xi^6 \int_0^1 \left| \hat{\psi}  \right|^2  r  \, dr \leq C \int_0^1 |\hat{f}|^2 r \, dr .
\ee

Similarly, multiplying \eqref{2-1-5} by $\xi^4$ yields
\be \label{2-2-3} \ba
& \xi^4 \int_0^1 \left| \mathcal{L} \hat{\psi}    \right|^2  r  \, dr + 2\xi^6 \int_0^1 \left|  \frac{d}{dr} (r \hat{\psi} )  \right|^2  \frac1r \, dr
+ \xi^8 \int_0^1 \left| \hat{\psi}  \right|^2  r  \, dr \\
\leq  & \int_0^1 |\hat{f}| |\xi^4 \hat{\psi}| r \, dr + C \Phi |\xi|^5 \int_0^1 \left|\frac{d}{dr} ( r \hat{\psi})  \right| |r \hat{\psi}| \, dr \\
\leq & \int_0^1 |\hat{f} | |\xi^4 \hat{\psi} | r \, dr + \Phi \xi^4 \int_0^1 \left| \frac{d}{dr} ( r \hat{\psi})  \right|^2 \frac1r \, dr
+ C \Phi \xi^6 \int_0^1 |\hat{\psi}|^2 r \, dr.
\ea \ee
By Young's inequality and the inequality \eqref{2-2-2}, one has
\be \label{2-2-4}
 \xi^4 \int_0^1 \left| \mathcal{L} \hat{\psi}    \right|^2  r  \, dr + \xi^6 \int_0^1 \left|  \frac{d}{dr} (r \hat{\psi} )  \right|^2  \frac1r \, dr
+ \xi^8 \int_0^1 \left| \hat{\psi}  \right|^2  r  \, dr
\leq  C (1 + \Phi) \int_0^1 |\hat{f}|^2 r \, dr,
\ee
which gives exactly the estimate \eqref{hapr1}.

Next, we turn to the high order regularity of $\psi$ with respect to $r$.
Multiplying \eqref{2-0-8} by $r \mathcal{L} ^2 \overline{\hat{\psi}}  $ and integrating over $[0, 1]$ yield
\be \label{2-2-5} \ba
& i \xi \int_0^1 \bar{U}(r) \mathcal{L} \hat{\psi} \mathcal{L}^2\overline{\hat{\psi}} r  \, dr
- i \xi^3 \int_0^1 \bar{U}(r) \hat{\psi}  \mathcal{L}^2 \overline{\hat{\psi}}  r  \, dr \\
& \ \ - \int_0^1 | \mathcal{L}^2 \hat{\psi}|^2  r  \, dr +  2\xi^2 \int_0^1 \mathcal{L}\hat{\psi}  \mathcal{L}^2 \overline{\hat{\psi}} r  \, dr - \xi^4 \int_0^1 \hat{\psi}  \mathcal{L}^2 \overline{\hat{\psi}} r  \, dr
 = \int_0^1 \hat{f}  \mathcal{L}^2 \overline{\hat{\psi}}  r \, dr.
\ea \ee

It follows from the  estimates \eqref{2-2-2} and Cauchy-Schwarz inequality that
\be \label{2-2-6} \ba
& \,\,\left| \xi \int_0^1   \bar{U}(r) \mathcal{L} \hat{\psi}  \mathcal{L}^2\overline{\hat{\psi}}  r  \, dr  \right|
\leq C \Phi |\xi| \int_0^1 |\mathcal{L} \hat{\psi} | |\mathcal{L}^2 \hat{\psi} | r \, dr \\
\leq & \, \, C \Phi^2 \xi^2 \int_0^1 |\mathcal{L} \hat{\psi}|^2 r \, dr +  \frac18 \int_0^1 | \mathcal{L}^2 \hat{\psi} |^2  r  \, dr \\
 \leq &\,\, C  \Phi^2  \int_0^1 |\hat{f}|^2  r  \, dr +  \frac18 \int_0^1 | \mathcal{L}^2 \hat{\psi} |^2  r  \, dr
\ea \ee
and
\be \label{2-2-7} \ba
&\,\, \left|   \xi^3 \int_0^1 \bar{U}(r) \hat{\psi} \mathcal{L}^2 \overline{\hat{\psi}}  r  \, dr   \right|  \leq C \Phi^2 \xi^6 \int_0^1 |\hat{\psi}|^2  r  \, dr + \frac18 \int_0^1 | \mathcal{L}^2 \hat{\psi}|^2   r  \, dr \\
 \leq &\,\, C \Phi^2   \int_0^1 |\hat{f}|^2 r \, dr +  \frac18 \int_0^1 | \mathcal{L}^2 \hat{\psi} |^2  r  \, dr.
\ea \ee
Similarly, according to the estimate \eqref{hapr1}, one has
\be \la{2-2-8}
\ba
& \,\,\left|   2\xi^2 \int_0^1 \mathcal{L} \hat{\psi}  \mathcal{L}^2 \overline{\hat{\psi}} r  \, dr  \right|
 \leq C \xi^4 \int_0^1 | \mathcal{L} \hat{\psi}|^2  r  \, dr + \frac18 \int_0^1 | \mathcal{L}^2 \hat{\psi}|^2  r  \, dr \\
 \leq &\,\, C (1+ \Phi)  \int_0^1 |\hat{f}|^2  r \, dr +  \frac18 \int_0^1 | \mathcal{L}^2 \hat{\psi} |^2  r  \, dr
\ea\ee
and
\be \la{2-2-9} \ba
&\,\, \left|  \xi^4 \int_0^1 \hat{\psi}  \mathcal{L}^2 \overline{\hat{\psi}}  r  \, dr   \right|
 \leq  C \xi^8 \int_0^1 | \hat{\psi}|^2  r  \, dr + \frac18 \int_0^1 | \mathcal{L}^2 \hat{\psi}|^2  r  \, dr \\
 \leq &\,\, C (1+ \Phi)  \int_0^1 |\hat{f}|^2  r \, dr +  \frac18 \int_0^1 | \mathcal{L}^2 \hat{\psi} |^2  r \, dr.
\ea \ee
Hence, combining all the estimates \eqref{2-2-5}--\eqref{2-2-9} gives \eqref{hapr2}. This finishes the proof of Lemma \ref{lemhapr}.
\end{pf}

\begin{remark}
 The key point for the a priori estimates obtained in Lemmas \ref{lemapri1} and  \ref{lemhapr}  is that there is
no smallness assumption on the flux $\Phi$. The result reveals that one can still get the regularity estimate for the solutions of   \eqref{2-0-8}  in spite of the  existence of convection terms which could be very large comparing with the diffusion term.
\end{remark}

\begin{remark}
In fact, the a priori estimates established in  Lemmas \ref{lemapri1} and  \ref{lemhapr} are enough to prove the existence of solutions of the problem \eqref{2-0-8}--\eqref{FBC}. The details for the existence of solutions for the linear problem \eqref{2-0-8}--\eqref{FBC} are presented  in Section \ref{sec-ex}.
\end{remark}


\section{Analysis on the linearized problem for swirl velocity}\label{sec-swirl}

In this section, we give the uniform estimates for $\Bv^\theta$, which is defined by $\Bv^\theta = v^\theta \Be_\theta$.  Assume that $\Bv$ is continuous, then the compatibility conditions at the axis implies that $v^\theta (0, z) = 0$. Hence the problem for $v^\theta$ can be written as
\be \label{swirlsystem}
\left\{   \ba
& \bar U(r)  \partial_z  v^\theta - \left[ \frac{1}{r} \frac{\partial }{\partial r} \left( r \frac{\partial v^\theta}{\partial r}\right) + \frac{\partial^2 v^\theta}{\partial z^2} - \frac{v^\theta}{r^2} \right] =  F^\theta, \ \ \ \mbox{in}\ \ D, \\
& v^\theta(1, z)= v^\theta(0, z) = 0.
\ea
\right.
\ee

\begin{pro}\label{swirl}
Assume that $\BF^\theta = F^\theta \Be_\theta \in L^2 (\Omega)$. The smooth solution $v^\theta$ to the linear problem \eqref{swirlsystem} satisfies
\be \label{swirl-1}
\| \Bv^{\theta} \|_{H^2(\Omega)} \leq C \|\BF^\theta \|_{L^2 (\Omega)},
\ee
where the  constant $C$ is independent of $\BF^\theta$ and $\Phi$.
\end{pro}

\begin{proof}
Taking the Fourier transform with respect to $z$ in the equation \eqref{vswirl} yields
\be \label{swirl-2}
i \xi \bar{U}(r) \widehat{v^\theta} - ( \mathcal{L} - \xi^2 )  \widehat{v^\theta} = \widehat{F^\theta},\quad \xi\in \mathbb{R},
\ee
and the boundary conditions for $\widehat{v^\theta}$ become
\be \label{swirl-3}
\widehat{v^\theta}(1 ) = \widehat{v^\theta} (0) = 0.
\ee

 Multiplying \eqref{swirl-2} by $r\overline{\widehat{v^\theta}}$ and integrating over
$[0, 1]$ leads to
\be \label{swirl-4}
\int_0^1 \left| \frac{d}{dr} ( r \widehat{v^\theta} )  \right|^2 \frac1r \, dr + \xi^2 \int_0^1 |\widehat{v^\theta}|^2 r\,dr
= \Re \int_0^1 \widehat{F^\theta} \overline{\widehat{v^\theta}} r \, dr
\ee
and
\be \label{swirl-5}
\frac{2\Phi}{\pi} \xi \int_0^1 (1  - r^2) |\widehat{v^\theta} |^2 r \, dr
= \Im \int_0^1 \widehat{F^\theta} \overline{\widehat{v^\theta}} r\, dr.
\ee
According to Lemma \ref{lemma1} and the equality \eqref{swirl-4}, it holds that
\be \label{swirl-6}
 \int_0^1 \left| \frac{d}{dr}( r \widehat{v^\theta} )\right|^2 \frac1r \, dr + \xi^2 \int_0^1 |\what{v^\theta}|^2 r \, dr  \leq \int_0^1 |\widehat{F^\theta}|^2 r \, dr .
\ee

Moreover, it follows from Lemma \ref{lemma1}, Lemma \ref{weightinequality} and \eqref{swirl-4}-\eqref{swirl-5} that
\begin{equation} \label{swirl-15-1}
\begin{aligned}
\int_0^1 |\what{v^\theta}|^2 rdr \leq & C \int_0^1 (1-r^2) |\what{v^\theta}|^2 rdr \\
&\quad +  C \left( \int_0^1 (1-r^2) |\what{v^\theta}|^2 rdr \right)^{\frac23} \left(\int_0^1 \left| \frac{d}{dr} ( r \widehat{v^\theta} )  \right|^2 \frac1r \, dr\right)^{\frac13}\\
\leq & C \left( \int_0^1 (1-r^2) |\what{v^\theta}|^2 rdr \right)^{\frac23} \left(\int_0^1 \left| \frac{d}{dr} ( r \widehat{v^\theta} )  \right|^2 \frac1r \, dr\right)^{\frac13}\\
\leq & C (\Phi |\xi|)^{-\frac23} \int_0^1 |\widehat{F^\theta}| |\widehat{v^\theta}| r \, dr.
\end{aligned}
\end{equation}
Therefore,
\begin{equation} \label{swirl-15-2}
\begin{aligned}
\int_0^1 |\what{v^\theta}|^2 rdr \leq  C (\Phi |\xi|)^{-\frac43} \int_0^1 |\widehat{F^\theta}|^2 r \, dr.
\end{aligned}
\end{equation}
This, together with \eqref{swirl-4}, yields
\be \label{swirl-15-3}
\ba
\int_0^1 \left| \frac{d}{dr} ( r \widehat{v^\theta} )  \right|^2 \frac1r \, dr + \xi^2 \int_0^1 |\widehat{v^\theta}|^2 r\,dr
\leq  &  \left(\int_0^1 |\widehat{F^\theta}|^2 r  dr \right)^{\frac12}\left(\int_0^1| \widehat{v^\theta} |^2 r \, dr\right)^{\frac12}\\
\leq &  C (\Phi |\xi|)^{-\frac23} \int_0^1 |\widehat{F^\theta} |^2 r \, dr .
\ea
\ee
Multiplying the equation \eqref{swirl-2} by $r ( \mathcal{L} - \xi^2) \overline{\widehat{v^\theta}}$ and integrating over $[0, 1]$ yield
\be \label{swirl-7}
- \int_0^1 | ( \mathcal{L} - \xi^2) \widehat{v^\theta} |^2 r \, dr
= \Re \int_0^1 \widehat{F^\theta} ( \mathcal{L} - \xi^2) \overline{\widehat{v^\theta}} r \, dr
+ \frac{2 \Phi}{\pi} \xi \Im \int_0^1 2r \widehat{v^\theta} \frac{d}{dr} ( r \overline{\widehat{v^\theta}}) \, dr .
\ee
Combining  \eqref{swirl-15-2} and \eqref{swirl-15-3} gives
\be \label{swirl-16} \ba
 \left| \frac{2 \Phi}{\pi} \xi \Im \int_0^1 2 r \widehat{v^\theta} \frac{d}{dr} ( r \overline{\widehat{v^\theta}} ) \, dr  \right|
\leq & C \Phi |\xi|  \left( \int_0^1 |\widehat{v^\theta}|^2 r  \, dr \right)^{\frac12}
\left( \int_0^1 \left| \frac{d}{dr} ( r \widehat{v^\theta})  \right|^2 \frac1r \, dr   \right)^{\frac12}\\
\leq & C    \int_0^1 |\widehat{F^\theta} |^2 r \, dr.
\ea \ee
Hence, by Young's inequality, it holds that
\be \label{swirl-16-1} \ba
\int_0^1 \left(| \mathcal{L} \widehat{v^\theta} |^2  + 2\xi^2 \left| \frac1r  \frac{d}{dr} ( r {\widehat{v^\theta}} ) \right|^2 + \xi^4 |\widehat{v^\theta} |^2  \right) r\, dr & =\int_0^1 | ( \mathcal{L} - \xi^2) \widehat{v^\theta} |^2 r \, dr \\
&\leq  C    \int_0^1 |\widehat{F^\theta} |^2 r \, dr .
\ea
\ee

Note that $\Bv^\theta$ satisfies
\be \label{swirl-17}
\left\{
\begin{aligned}
&\Delta \Bv^\theta = ( \mathcal{L} + \partial_z^2) v^\theta \Be_\theta, \ \ \ &\mbox{in} \ \Omega,\\
& \Bv^\theta =0\quad  & \text{on}\ \partial \Omega.
\end{aligned}
\right.
\ee
Indeed, the proof of  for the validity of \eqref{swirl-17} is almost the same as that of \eqref{vorteq}, which can be found in Appendix \ref{secappendix}, so we omit the details here.
Applying the regularity theory for elliptic equations \cite{GT} yields
\be \label{swirl-18}
\|\Bv^\theta\|_{H^2 (\Omega)}
\leq C \| ( \mathcal{L} + \partial_z^2) v^\theta\|_{L^2(\Omega)}
+ C \|v^\theta \|_{L^2(\Omega)} \leq C \|\BF^\theta\|_{L^2(\Omega)}.
\ee
This finishes the proof of Proposition \ref{swirl}.
\end{proof}


\section{Existence of solutions for the linearized problem}\label{sec-ex}
In this section, we use  Galerkin method to  prove the existence of solutions to the problem \eqref{2-0-8}-\eqref{FBC} for each fixed $\xi$, give a sketch of the existence proof of solutions to the problem \eqref{swirlsystem} on the swirl velocity.

The following function spaces will be needed.
\begin{definition}\label{def0}
  Denote
$$X= \left\{ \varphi \in C^\infty([0, 1]): \  \varphi(1) = \varphi^{\prime}(1)=0, \ \ \lim_{r \rightarrow 0+} \mathcal{L}^k \varphi(r) = 0, \ \lim_{r \rightarrow 0+} \, \frac{d}{dr} ( r \mathcal{L}^k \varphi)(r) = 0,\  k\in \mathbb{N} \right\}.$$
Let $X_0$ be the completion of $C^\infty([0,1])$ under the norm
$$\|\varphi\|_{X_0} := \left( \int_0^1 |\varphi|^2  r \, dr \right)^{1/2}. $$
Let $X_4$ be the closure of $X$ with respect to the following $X_4$-norm,
$$\|\varphi\|_{X_4}^2 : = \displaystyle
\int_0^1  \left(  |\varphi|^2 + |\mathcal{L} \varphi |^2 + |\mathcal{L}^2 \varphi|^2\right) r \, dr .$$
\end{definition}

In order to apply Galerkin method, one needs to construct an orthonormal basis for $X_0$. Our strategy is  to seek
a basis in the function space $X_4$. To this end, we need some properties of functions in  $X_4$.

\begin{lemma}\la{lemma3-3-1}
Let $\varphi\in X_4$. Then $\varphi,$ $\frac{d}{dr}(r\varphi ) ,$  $\mathcal{L}\varphi$,
$\frac{d}{dr} ( r \mathcal{L} \varphi) $ $\in C([0, 1])$,
\begin{equation}\label{3-3-1-1}
\varphi (0) = \varphi (1) = \frac{d}{dr} ( r\varphi ) (1) = \mathcal{L} \varphi (0) = 0,
\end{equation}
and
\begin{equation} \label{3-3-1-2}
\lim_{r \rightarrow 0 +} \frac{1}{\sqrt{r}} \frac{d}{dr} ( r\varphi ) = \lim_{r \rightarrow 0 + }
\frac{1}{\sqrt{r}} \frac{d}{dr} ( r \mathcal{L} \varphi ) = 0.
\end{equation}
Moreover, there exists a positive constant $C$ independent of $\varphi$, such that
\be \la{3-3-1-3}
\int_0^1 \frac{|\varphi|^2}{r} \, dr + \int_0^1 \left|  \frac{d\varphi}{dr} \right|^2  r \, dr \leq C \int_0^1 | \mathcal{L} \varphi|^2  r \, dr,
\ee
and
\be \la{3-3-1-4}
 \int_0^1 \left| \frac{d}{dr}( r \mathcal{L} \varphi) \right|^2 \frac1r \, dr + \int_0^1 \frac{|\mathcal{L}\varphi |^2}{r} \, dr
+ \int_0^1 \left| \frac{d}{dr} \mathcal{L} \varphi \right|^2  r \, dr \leq C \|\varphi \|_{X_4}^2.
\ee
\end{lemma}

\begin{proof}
Suppose that $\{ \varphi_n  \} \subseteq X $ is a sequence which converges to $\varphi$ in $X_4$. It follows from the proof of
Lemma \ref{lemma1} that
\be  \nonumber
\int_0^1 \left| \frac{d}{dr} ( r \varphi_n)  \right|^2  \frac{1}{r} \, dr
\leq  \int_0^1 \left| \mathcal{L} \varphi_n  \right|^2  r \, dr
\leq  \|\varphi_n\|_{X_4}^2.
\ee

For every $r \in [0, 1]$,
\be \nonumber
| r \varphi_n (r) | = \left|  \int_0^r \frac{d}{ds} [s \varphi_n (s)]  ds       \right|
\leq \sqrt{2}  r\left(  \int_0^1  \left|  \frac{d}{ds}[s \varphi_n (s)] \right|^2 \frac{1}{s} \, ds \right)^{\frac12}  ,
\ee
which  gives that
\begin{equation}\label{3-3-1-7}
\sup_{r \in [0, 1] } | \varphi_n (r) | \leq \frac{\sqrt{2}}{2} \left(  \int_0^1  \left|  \frac{d}{ds}[s \varphi_n (s)] \right|^2 \frac{1}{s} \, ds \right)^{\frac12}  \leq \frac{\sqrt{2}}{2} \|\varphi_n \|_{X_4}.
\end{equation}
The inequality \eqref{3-3-1-7} also holds for $\varphi_n - \varphi_m$. Hence, $\varphi \in C([0, 1])$, and
\begin{equation}\nonumber
\varphi(0) = \lim_{n \rightarrow + \infty} \varphi_n (0) = 0,
\ \ \ \ \ \varphi(1) = \lim_{n \rightarrow + \infty}  \varphi_n (1) = 0.
\end{equation}

Similarly, one has
\begin{equation}\nonumber
\left|  \frac{d(r \varphi_n)}{dr}  \frac{1}{r}     \right|
= \left| - \int_r^1 \mathcal{L} \varphi_n \, ds             \right|
\leq \left[ \int_0^1 | \mathcal{L} \varphi_n |^2  s \, ds     \right]^{\frac12}  \left[ \int_r^1 \frac{1}{s} \, ds   \right]^{\frac12},
\end{equation}
which implies that
\begin{equation}\label{3-3-1-10}
\left| \frac{d (r \varphi_n )}{dr} \right|  \leq  r  |\ln r |^{\frac12}  \left[ \int_0^1 | \mathcal{L} \varphi_n|^2  s \, ds \right]^{\frac12}.
\end{equation}
Hence,  $ \frac{d (r \varphi) }{dr} \in C([0, 1])$, and
\begin{equation} \label{3-3-1-11}
\left[  \frac{1}{r} \frac{d(r \varphi)}{dr}  \right](1) = 0,\ \ \ \lim_{r \rightarrow 0 +} \frac{1}{\sqrt{r}}
\frac{d(r \varphi)}{dr} = 0.
\end{equation}

Moreover, one has
\be \la{3-3-1-10-1}
|r \varphi_n(r)|  \leq \int_0^r \left| \frac{d}{ds} \left[ s \varphi_n(s)  \right]      \right|\, ds
\leq   \int_0^r s |\ln s|^{\frac12} \, ds  \left[ \int_0^1 | \mathcal{L} \varphi_n|^2  s \, ds  \right]^{\frac12} ,
\ee
which implies that
\be \la{3-3-1-10-2}
|\varphi_n(r)| \leq C r^{\frac34}  \left[ \int_0^1 | \mathcal{L} \varphi_n|^2  s \, ds  \right]^{\frac12} ,
\ee
and consequently
\be \la{3-3-1-10-3}
\int_0^1 \frac{|\varphi_n|^2}{r} \, dr \leq C  \int_0^1 | \mathcal{L} \varphi_n|^2  s \, ds   \leq C \|\varphi_n \|_{X_4}^2.
\ee
Meanwhile, one has
\be \la{3-3-1-10-4}
\int_0^1 \left|  \frac{d\varphi_n}{dr} \right|^2  r \, dr
\leq 2 \int_0^1 \left| \frac{d}{dr} (r \varphi_n) \right|^2 \frac{1}{r}\, dr +  2 \int_0^1 \frac{|\varphi_n|^2}{r} \, dr
\leq C \int_0^1 | \mathcal{L} \varphi_n|^2  r \, dr  .
\ee
Finally, taking the limit for $\varphi_n$ shows that the inequalities \eqref{3-3-1-10-1}--\eqref{3-3-1-10-4} hold also for $\varphi$.

For every $0 \leq r \leq 1$, one has
\be\nonumber
 |  r \mathcal{L} \varphi_n (r) | = \left|  \int_0^r \frac{d}{ds} ( s \mathcal{L} \varphi_n) \, ds        \right|
\leq C \left[  \int_0^r  \left|   \frac{d}{ds} ( s \mathcal{L} \varphi_n) \right|^2  \frac{1}{s} \, ds   \right]^{\frac12} r ,
\ee
which together with  \eqref{3-3-1-20} in Lemma \ref{lemmaA2} implies that
\begin{equation}\label{3-3-1-25}
|\mathcal{L} \varphi_n (r) | \leq C \left[ \int_0^1 \left| \frac{d}{dr}(s \mathcal{L} \varphi_n)  \right|^2 \frac1s \, ds  \right]^{\frac12} \leq C \|\varphi_n\|_{X_4}.
\end{equation}
Hence, taking the limit for $\varphi_n$ gives that
\be \nonumber
\mathcal{L}\varphi \in C([0, 1])\ \ \ \ \ \mbox{and}\ \ \ \ \ \ \lim_{r\rightarrow 0+} \mathcal{L} \varphi(r) = 0.
\ee

Similarly,
\be \label{3-3-1-28} \ba
\left| \frac{1}{r} \frac{d}{dr} ( r \mathcal{L} \varphi_n)  \right|   = &\,\, \left|
\frac{d}{dr}( r \mathcal{L} \varphi_n) (1)  - \int_r^1 \frac{d}{ds} \left[
\frac{1}{s} \frac{d}{ds} ( s \mathcal{L} \varphi_n ) \right] \, ds   \right| \\
 \leq &\,\, \left| \frac{d}{dr} ( r \mathcal{L} \varphi_n )(1)  \right|  +  \left( \int_0^1 |\mathcal{L}^2 \varphi_n |^2  s \, ds \right)^{\frac12}
 | \ln r|^{\frac12} \\
 \leq &\,\, C \|\varphi_n\|_{X_4}  \left( 1 + |\ln r|^{\frac12} \right),
\ea
\ee
where we used the inequality \eqref{estlinfty} in Lemma \ref{lemmaA2} for the last inequality.
Therefore, taking the limit for $\varphi_n$ yields that
\begin{equation}\nonumber
\frac{d}{dr}( r \mathcal{L} \varphi ) \in C([0, 1])\ \ \ \ \ \mbox{and}\ \ \ \ \lim_{r\rightarrow 0+} \frac{1}{\sqrt{r}} \frac{d}{dr}
( r \mathcal{L} \varphi ) = 0.
\end{equation}

Similar to the estimate above, one can get that
\be \la{3-3-1-30}
\left| r \mathcal{L} \varphi(r) \right|
\leq \int_0^r \left|  \frac{d}{ds} \left[ s \mathcal{L} \varphi(s)   \right]    \right| \, ds
\leq C \int_0^r \left( s + s |\ln s|^{\frac12}  \right) \, ds  \|\varphi\|_{X_4}
\leq C r^{\frac74} \|\varphi\|_{X_4},
\ee
and
\be \la{3-3-1-31}
\left|  \frac{d}{dr} ( \mathcal{L} \varphi) (r) \right| = \left| \frac{1}{r} \frac{d}{dr}( r \mathcal{L} \varphi) - \frac{\mathcal{L} \varphi}{r}  \right|
\leq C \left(  r^{-\frac14} + |\ln r|^{\frac12} \right) \|\varphi \|_{X_4}.
\ee
Hence,
\be \la{3-3-1-32}
\int_0^1 \left|  \frac{d}{dr} (  r \mathcal{L} \varphi) \right|^2 \frac1r \, dr + \int_0^1 \frac{|\mathcal{L}\varphi|^2 }{r} \, dr +
\int_0^1 \left| \frac{d}{dr} (\mathcal{L} \varphi)  \right|^2  r \, dr \leq C \|\varphi \|_{X_4}.
\ee
This finishes the proof of Lemma \ref{lemma3-3-1}.
\end{proof}

Based on Lemma \ref{lemma3-3-1}, one can obtain the following compactness result.


\begin{lemma}\la{lemma3-3-3}
$X_4$ is compactly embedded into $X_0$.
\end{lemma}

\begin{proof} Assume that $\{\varphi_n\}$ is a bounded sequence in $X_4$. Owing to Lemma \ref{lemma3-3-1}, it holds that
\be \nonumber
 \int_0^1 \left|  \frac{d \varphi_n}{dr}    \right|^2  r \, dr + \int_0^1  |\varphi_n |^2  \frac1r \, dr
\leq C \|\varphi_n\|_{X_4}^2 .
\ee
Therefore, if  $\varphi_n$ is regarded as a radially symmetric function defined on $B_1(0) \subset \mathbb{R}^2$,
then $ \{ \varphi_n \}$ is a bounded sequence in $H^1(B_1(0))$.  It is well-known that $H^1(B_1(0))$ is compact in $L^2(B_1(0))$. Hence there is a subsequence of $\{\varphi_n \}$ (still labelled by $\{ \varphi_n \}$) and a radially symmetric function
$\varphi$, such that $\{ \varphi_n \}$ converges to $\varphi$ in $L^2(B_1(0))$.  Hence, $\{\varphi_n \} $ converges to $\varphi$ in $X_0$, which completes the proof of Lemma \ref{lemma3-3-3}.
\end{proof}

Before discussing the eigenfunctions of the differential operator $\mathcal{L}^2$, we study first the existence of solutions to the associated PDE.
\begin{lemma}\la{lemma3-3-2}
Given $g\in X_0$, the following problem
\begin{equation}\label{3-3-1}
\left\{  \ba
&\mathcal{L}^2 \varphi = g,\ \ \ \ \mbox{in}\ (0, 1),\\
&\varphi (0) = \mathcal{L} \varphi (0) =\varphi(1) = \varphi^{\prime}(1) = 0,
\ea
\right.
\end{equation}
has a unique solution $\varphi \in X_4$, and it holds that
\begin{equation}\nonumber
\|\varphi \|_{X_4}   \leq C \|g\|_{X_0},
\end{equation}
where the constant $C$ is independent of $g$.
\end{lemma}

\begin{proof}

We do not construct the solution to the problem \eqref{3-3-1} directly. Instead, we consider the following boundary value problem for a fourth order equation on $\overline{B_1^4}(0)\subset \mathbb{R}^4$,
\be \la{3-3-2}  \left\{ \ba
& \Delta^2_4 \phi = G, \ \ \ \ \ \mbox{in}\ B_1^4(0),\\
& \phi = \frac{\partial \phi}{\partial \Bn  } = 0,\ \ \ \ \ \mbox{on}\ \partial B_1^4(0),
\ea \right. \ee
where $B_1^4(0)$ is the unit ball centered at the origin in $\mathbb{R}^4$,
\be \nonumber
\Delta_4= \sum_{i=1}^4 \partial_{x_i}^2 \quad \text{and}\quad G(x_1, x_2, x_3, x_4) = g(r)/r\ \ \ \ \mbox{with}\ r = \left(\sum_{i=1}^4 x_i^2 \right)^{ \frac12 }.
\ee

According to the classical theory for elliptic equations\cite{ADN}, there exists a unique solution $\phi \in H^4(B_1^4(0) )$, satisfying
\be \label{3-3-3}
\|\phi\|_{H^4(B_1^4(0) )} \leq C \| G\|_{L^2(B_1^4(0))} = C \|g\|_{X_0}.
\ee
 Since $G$ is a radially symmetric function, due to the rotational invariance  of $\Delta_4$ and the uniqueness of solution to the problem \eqref{3-3-2},
$\phi$ is also a radially symmetric function, i.e.,
$\phi(x_1, x_2, x_3, x_4) = \phi(r)$.

Let $\varphi(r) = r \phi(r)$. It can be verified that
\be \nonumber \ba
\mathcal{L} \varphi & = \left(  \frac{d^2}{dr^2} + \frac{1}{r} \frac{d}{dr} - \frac{1}{r^2}       \right) ( r\phi)   = r \left(  \frac{d^2}{dr^2} + \frac{3}{r} \frac{d}{dr}             \right) \phi
 = r \Delta_4 \phi.
\ea\ee
Similarly,
\be \nonumber
\mathcal{L}^2 \varphi = \mathcal{L} ( \mathcal{L} \varphi) = \mathcal{L} (r \Delta_4 \phi) = r \Delta_4^2 \phi =g.
\ee
Hence,
\begin{equation}\label{3-3-9} \ba
\|\varphi\|_{X_4}^2  = &\,\, \int_0^1 \left(  |\mathcal{L}^2 \varphi |^2  + |\mathcal{L} \varphi |^2  +
|\varphi|^2 \right)  r \, dr
 =  \int_0^1 \left(     | \Delta_4^2 \phi |^2 + |\Delta_4 \phi |^2 + |\phi|^2     \right)  r^3 \, dr \\
 \leq &\,\, C \|\phi \|_{H^4(B_1^4(0))}^2
 \leq  C \|g\|_{X_0}^2.
\ea
\end{equation}

In fact, the function $G$ can be approximated by a sequence of smooth radially symmetric functions $\{ G_n \}$ in $L^2(B_1^4(0))$. The corresponding solutions $\{ \phi_n \} \subseteq C^\infty(\overline{B_1^4(0)})$, which implies
$\{ \varphi_n = r \phi_n \}\subseteq X$. Hence, the solution $\varphi$ can be approximated by $\{ \varphi_n \}\subseteq X$ under
$X_4$-norm, thus $\varphi \in X_4$. It follows from Lemma \ref{lemma3-3-1} that $\varphi$ satisfies the boundary conditions in \eqref{3-3-1}. Therefore, $\varphi$ is a solution to the problem \eqref{3-3-1}.

Finally, we prove the uniqueness. Suppose $\varphi_1$ and $\varphi_2$ are both solutions to \eqref{3-3-1} in $X_4$. Taking
$\varphi_1 - \varphi_2$ as a test function,  by virtue of Lemma \ref{lemma3-3-1},
\be\nonumber
\ba
0= &\,\,\int_0^1 \mathcal{L}^2 ( \varphi_1 - \varphi_2) \overline{(\varphi_1 - \varphi_2)}  r \, dr \\
=&\,\, \displaystyle -  \int_0^1 \frac{1}{r} \frac{d}{dr} ( r \mathcal{L} \varphi_1  - r \mathcal{L}\varphi_2)
 \frac{d}{dr} \left[   \overline{ r (\varphi_1 - \varphi_2) }      \right] \, dr
+ \frac{d}{dr}( r \mathcal{L} \varphi_1 - r\mathcal{L} \varphi_2 )(1)  (\overline{\varphi_1 - \varphi_2})(1)\\
&\ \ \ \ \ \
- \lim_{r \rightarrow 0+ } \frac{d}{dr} ( r \mathcal{L} \varphi_1 - r \mathcal{L} \varphi_2)(r)
\lim_{r \rightarrow 0+} (\overline{ \varphi_1 - \varphi_2} )(r)\\
 =&\,\, - \displaystyle  \int_0^1 \frac{1}{r} \frac{d}{dr} ( r \mathcal{L} \varphi_1  - r \mathcal{L}\varphi_2)
 \frac{d}{dr} \left[   \overline{ r (\varphi_1 - \varphi_2) }      \right] \, dr \\
=&\,\, \int_0^1 | \mathcal{L} (\varphi_1 - \varphi_2) |^2  r \, dr
- ( \mathcal{L} \varphi_1 - \mathcal{L} \varphi_2) (1)  \frac{d}{dr}[ r (\overline{\varphi_1 - \varphi_2} ) ](1) \\
&\ \ \ \ \
+  \lim_{r \rightarrow 0+ } ( \mathcal{L} \varphi_1 - \mathcal{L} \varphi_2)(r)  \lim_{r\rightarrow 0+}
\frac{d}{dr}[r( \overline{\varphi_1 - \varphi_2} )]\\
 =&\,\,  \int_0^1 | \mathcal{L} (\varphi_1 - \varphi_2) |^2  r \, dr.
\ea
\end{equation}
Hence, one has
\begin{equation}\nonumber
\int_0^1  |\varphi_1 - \varphi_2|^2  r \, dr  \leq C \int_0^1  |\mathcal{L} (\varphi_1 - \varphi_2)|^2  r \, dr = 0,
\end{equation}
which implies the uniqueness of solutions to \eqref{3-3-1}. This finishes the proof of Lemma \ref{lemma3-3-2}.
\end{proof}

Now we are ready to show the existence of orthonormal basis for $X_0$.
\begin{pro}\label{theorem3-3-4}
There exists an orthonormal basis $\{\varphi_n\}\subset X_4$ for $X_0$.
\end{pro}
\begin{proof}
Suppose that $\varphi$ is the unique solution to \eqref{3-3-1}. Define the solution operator as $\varphi = \mathcal{M} g$. Since
\be \nonumber
\int_0^1 \mathcal{M}g  \bar{g}  r  \, dr    =  \int_0^1 \varphi  \mathcal{L}^2 \overline{\varphi}  r \, dr
= \int_0^1 |\mathcal{L} \varphi|^2  r \, dr  \in \mathbb{R},
\ee
$\mathcal{M}$ is a symmetric operator on $X_0$.

It follows from Hilbert-Schmidt theory (\cite{Lax}) and Lemma \ref{lemma3-3-3} that there exists an orthonormal basis of $X_0$. In fact,
the basis consists of the eigenfunctions of the operator $\mathcal{M}$.
\end{proof}

Now for every $\xi \in \mathbb{R}$, the existence of a solution $\hat{\psi} $ can be obtained by the  standard Galerkin approximation method together with the a priori estimates. Since all the a priori estimates hold for the approximate solutions, they also hold for the solution $\hat{\psi}$. The uniqueness of the solution to \eqref{2-0-8}-\eqref{FBC} can also be established by a priori estimates.

\begin{definition}\label{def1}
Define a function space as
\be \nonumber
C_*^\infty(D )  = \left\{ \varphi(r, z) \left|  \ba & \ \varphi \in C_c^\infty([0, 1] \times \mathbb{R} ), \
\varphi(1, z) = \frac{\partial \varphi}{\partial r}(1, z) =0, \\
&\ \  \   \mbox{and}\ \lim_{r\rightarrow 0+ } \mathcal{L}^k \varphi(r, z)
= \lim_{r \rightarrow 0+ } \frac{\partial}{\partial r} ( r \mathcal{L}^k \varphi ) (r, z) = 0, \ k \in \mathbb{N}
\ea \right.
\right\}.
 \ee
The $H_r^4(D)$-norm is defined as follows,
\begin{equation}\nonumber
\|\varphi\|_{H^4_r(D)}^2  : = \int_{-\infty}^{+ \infty} \int_0^1 \left(  |\mathcal{L} \hat{\varphi}|^2 + \xi^4 |\hat{\varphi}|^2 \right)   r
 \, dr d\xi   + \int_{-\infty}^{+\infty} \int_0^1 \left( |\mathcal{L}^2 \hat{\varphi}|^2 + \xi^4 |\mathcal{L}\hat{\varphi}|^2
+ \xi^8 |\hat{\varphi}|^2 \right)  r \, dr d\xi .
\end{equation}
Let $H_*^4(D)$ denote the closure of $C_*^\infty(D)$ under the $H_r^4(D)$-norm.
Define
\be \nonumber
L_r^2(D) = \left\{ f(r, z):  \ \|f\|_{L_r^2(D )}^2 = \int_{-\infty}^{+\infty} \int_0^1  |f|^2  r \, dr dz < + \infty \right\}.
\ee
\end{definition}

The existence result for $\psi$ follows from the a priori estimates established in Section \ref{Linear} .
\begin{pro} \label{existence-stream} Assume that  $f(r, z) \in L_r^2(D)$. There exists a unique solution $\psi \in H_*^4(D)$ to the linear system \eqref{2-0-4-1}--\eqref{2-0-4-4},  and a positive constant $C$, which is independent of $f$ and $\Phi$, such that
\begin{equation}\nonumber
\|\psi \|_{H^4_r(D)} \leq C (1 + \Phi) \|f\|_{L^2_r(D)}.
\end{equation}
\end{pro}

 Let us discuss the existence of solutions for the problem \eqref{swirlsystem} on the swirl velocity. For every fixed $\xi \in \mathbb{R}$, when $\Phi=  0$, the existence of $v^\theta$ is obtained by standard Galerkin method. When $\Phi  \neq 0$, the existence of solutions for the problem \eqref{swirlsystem} is obtained by the continuity method (\cite{GT}) and the a priori estimates derived in Section \ref{sec-swirl}. The uniqueness of the solutions for \eqref{swirlsystem} also follows from the a priori estimate \eqref{swirl-1}. Therefore, we have the following existence result.
\begin{pro}\label{existence-swirl}
Assume that $\BF^\theta \in L^2 (\Omega)$. There exists a unique solution $v^\theta $ to the linear system \eqref{swirlsystem} and a positive constant $C$, which is independent of $F^\theta$ and $\Phi$, such that
\be \nonumber
\|\Bv^\theta \|_{H^2(\Omega)} \leq C \|\BF^\theta \|_{L^2(\Omega)}.
\ee
\end{pro}


\section{Regularity of the Velocity field} \label{sec-reg}
In this section, we first give more properties of functions in
$H_*^4(D)$,  and then give the regularity of $\Bv$.


\begin{lemma}\label{lemma3-4-1}  Let $\varphi$ be a function in $H^4_*(D)$  defined in Definition \ref{def1}. Then there exists a positive constant $C$, independent of $\varphi$, such that
\be \la{2-4-5} \ba
&\,\, \int_{-\infty}^{+\infty} \int_0^1\left( |\hat{\varphi} |^2 + \xi^2 |\hat{\varphi} |^2
+ \xi^2 | \mathcal{L} \hat{\varphi} |^2 + \left| \frac{\partial}{\partial r} (\mathcal{L} \hat{\varphi} ) \right|^2 + \xi^4 \left|  \frac{\partial}{\partial r}\hat{\varphi}    \right|^2 + \xi^6 | \hat{\varphi} |^2  \right)  r \, dr d\xi
\\[2mm]
& \ \ \ \ +\int_{-\infty}^{+ \infty} \int_0^1  \left(
\left| \frac{\partial}{\partial r}( r \hat{\varphi} )  \right|^2 + \xi^6 \left|  \frac{\partial}{\partial r} ( r\hat{\varphi} )   \right|^2 + \xi^2 \left| \frac{\partial}{\partial r} ( r \mathcal{L} \hat{\varphi} )     \right|^2  \right)  \frac1r  \, dr d\xi \\
  \leq &\,\, C \|\varphi\|_{H_r^4(D)}^2.
\ea \ee
\end{lemma}

\begin{proof}For simplicity, assume that $\varphi \in C_*^\infty(D )$, which is defined in Definition \ref{def1}. Following the proof
of Lemma \ref{lemma1}, due to the homogeneous boundary conditions for $ \varphi$, we have
\be \la{3-4-6}
\int_{-\infty}^{+ \infty} \int_0^1 \left| \frac{\partial}{\partial r} ( r \hat{\varphi})   \right|^2  \frac1r \, dr  d\xi \leq  \int_{-\infty}^{+ \infty} \int_0^1 |\mathcal{L} \hat{\varphi} |^2  r  \, dr d\xi \leq \|\varphi\|_{H^4_r(D)}^2,
\ee
and
\be \la{3-4-7}
\int_{-\infty}^{+\infty} \int_0^1 |\hat{\varphi}|^2  r  \, dr d\xi \leq
\int_{-\infty}^{+\infty} \int_0^1  \left| \frac{\partial}{\partial r} (r  \hat{\varphi} ) \right|^2  \frac1r \, dr d\xi  \leq  \|\varphi\|_{H^4_r(D)}^2.
\ee
Furthermore, it follows from  Cauchy-Schwarz inequality that one has
\be \la{3-4-8}
\int_{-\infty}^{+\infty} \int_0^1 \left( \xi^2  |\hat{\varphi}|^2 +  \xi^6 |\hat{\varphi}|^2 + \xi^2 |\mathcal{L} \hat{\varphi} |^2  \right)  r
 \, dr d\xi \leq C \|\varphi\|_{H_r^4(D)}^2.
\ee

Similarly, integration by parts, together with the homogeneous boundary conditions for $r \varphi$, yields
\be \nonumber \ba
\xi^6 \int_0^1 \left| \frac{\partial }{\partial r} (r  \hat{\varphi} ) \right|^2 \frac1r \, dr
 =  -  \xi^6 \int_0^1 \mathcal{L} \hat{\varphi}  \overline{\hat{\varphi}}  r  \, dr
 \leq \xi^8 \int_0^1 |\hat{\varphi}|^2  r  \, dr + \xi^4 \int_0^1 |\mathcal{L} \hat{\varphi}|^2  r  \, dr,
\ea \ee
which implies
\be \la{3-4-9}
\int_{-\infty}^{+\infty}  \int_0^1 \xi^6 \left| \frac{\partial}{\partial r } ( r\hat{\varphi}) \right|^2  \frac1r \, dr d\xi
\leq C \|\varphi\|_{H^4_r(D) }^2.
\ee

It follows from  Lemma \ref{lemma3-3-1} that
\be \la{3-4-10}
\int_{- \infty}^{+ \infty} \int_0^1 \xi^4 \left|  \frac{\partial}{\partial r} \hat{\varphi}  \right|^2  r \, dr d\xi
\leq C \int_{-\infty}^{+ \infty} \int_0^1 \xi^4 | \mathcal{L} \hat{\varphi} |^2  r \, dr d\xi
\leq C \|\varphi\|_{H_r^4(D)}.
\ee
Moreover, for every $\xi \in \mathbb{R}$, using \eqref{3-3-1-20} in Lemma \ref{lemmaA2} gives
\be \label{3-4-11}
 \int_0^1 \left| \frac{\partial}{\partial r} ( r \mathcal{L} \hat{\varphi})   \right|^2  \frac1r \, dr
\leq C \int_0^1 |\mathcal{L} \hat{\varphi} |^2 r \, dr + C \left( \int_0^1 |\mathcal{L} \hat{\varphi} |^2 r \, dr  \right)^{\frac12} \left(  \int_0^1
|\mathcal{L}^2 \hat{\varphi} |^2 r \, dr \right)^{\frac12}.
 \ee

For every fixed $\xi \in \mathbb{R}$, it holds that
\be \label{3-4-15-1}
\int_0^1 \left| \frac{\partial}{\partial r} ( r \mathcal{L} \hat{\varphi} ) \right|^2 \frac1r \, dr = - \int_0^1 \mathcal{L}^2 \hat{\varphi} \mathcal{L } \overline{\hat{\varphi}} r \, dr + \frac{\partial}{\partial r} (r \mathcal{L} \hat{\varphi} ) (1) \mathcal{L} \overline{\hat{\varphi}}(1).
\ee
Multiplying \eqref{3-4-15-1} by $\xi^2$ gives
\be \la{3-4-15}
\begin{aligned}
 &\xi^2 \int_0^1 \left|  \frac{\partial}{\partial r} (r \mathcal{L} \hat{\varphi} )     \right|^2 \frac1r \, dr
 =  - \xi^2 \int_0^1 \mathcal{L} \hat{\varphi}  \cdot \mathcal{L}^2 \overline{\hat{\varphi}}  r  \, dr + \xi^2 \frac{\partial}{\partial r}( r \mathcal{L} \hat{\varphi} ) (1)  \mathcal{L} \overline{\hat{\varphi}}(1) \\
\leq \,\,& \frac12 \xi^4 \int_0^1 | \mathcal{L} \hat{\varphi}|^2  r  \, dr
+ \frac12 \int_0^1  |\mathcal{L}^2 \hat{\varphi}|^2  r  \, dr
+ |\xi|^3 \left| \mathcal{L} \hat{\varphi}(1) \right|^2 + |\xi| \left|   \frac{\partial}{\partial r}(r  \mathcal{L} \hat{\varphi} ) (1)     \right|^2.
\end{aligned}
\ee
According to Lemma \ref{lemmaA2}, one has
\be \nonumber
\ba
&\,\, |\xi|^3 \left| \mathcal{L} \hat{\varphi}(1) \right|^2 \\
 \leq &\,\, 8 |\xi|^3 \int_{0}^1 | \mathcal{L} \hat{\varphi}|^2  r  \, dr +
 8 |\xi|^3 \left(  \int_{0}^1 | \mathcal{L} \hat{\varphi}|^2  r    \, dr \right)^{\frac12}
\left(  \int_{0}^1 \left|    \frac{\partial}{\partial r} (r \mathcal{L} \hat{\psi} )   \right|^2  \frac1r \, dr  \right)^{\frac12} \\
 \leq &\,\, 8    \left(  \int_0^1 | \mathcal{L} \hat{\varphi}|^2  r \, dr  \right)^{\frac14}  \left(  \xi^4
\int_0^1 | \mathcal{L} \hat{\varphi}|^2   r \,  dr          \right)^{\frac34} \\
&\ \ \  \ + 8 \left(  \xi^4 \int_0^1 | \mathcal{L} \hat{\varphi} |^2  r  \, dr    \right)^{\frac12} \left(            \xi^2 \int_0^1 \left|   \frac{\partial}{\partial r}( r  \mathcal{L} \hat{\varphi} )    \right|^2  \frac1r \, dr \right)^{\frac12}.
\ea \ee

Similarly, it follows from \eqref{3-3-1-20-1} in Lemma \ref{lemmaA2} that one has
\be \nonumber  \ba
& \,\,|\xi| \left|   \frac{\partial}{\partial r}(r  \mathcal{L} \hat{\varphi} ) (1)     \right|^2  \\
 \leq &\,\, 4  |\xi| \int_0^1 \left|   \frac{\partial}{\partial r} ( r \mathcal{L}\hat{\varphi}  )  \right|^2  \frac1r  \, dr +
4 |\xi| \left(  \int_{0}^1 \left| \frac{\partial}{\partial r} (r \mathcal{L}\hat{\varphi} ) \right|^2  \frac1r  \, dr \right)^{\frac12}  \left(  \int_{\frac12}^1 | \mathcal{L}^2 \hat{\varphi} |^2  r  \, dr\right)^{\frac12}\\
 \leq &\,\, 4 \left( \int_0^1 \left|   \frac{\partial}{\partial r} (r \mathcal{L}\hat{\varphi} )   \right|^2  \frac1r \, dr       \right)^{\frac12}  \left( \xi^2 \int_0^1 \left|   \frac{\partial}{\partial r} (r  \mathcal{L}\hat{\varphi} )  \right|^2  \frac1r \, dr      \right)^{\frac12}  \\
&\ \ \ \ \ \ + 8 \left(  \xi^2 \int_0^1 \left|  \frac{\partial}{\partial r} ( r \mathcal{L} \hat{\varphi} )  \right|^2  \frac1r \, dr \right)^{\frac12}
\left( \int_0^1 |\mathcal{L}^2 \hat{\varphi} |^2  r \, dr        \right)^{\frac12} .
\ea \ee
Using Young's inequality yields
\be \la{3-4-16}
 \int_{-\infty}^{+\infty} \int_0^1 \xi^2 \left|  \frac{\partial}{\partial r} ( r \mathcal{L} \hat{\varphi} )   \right|^2  \frac1r \, dr d\xi \leq C \|\varphi\|_{H_r^4(D)}^2.
\ee

Finally, according to the proof of Lemma \ref{lemma3-3-1}, it holds that
\be \la{3-4-17}
\int_{-\infty}^{+ \infty} \int_0^1 \left| \frac{\partial }{\partial r} \mathcal{L} \hat{\varphi}   \right|^2  r \, dr d\xi
\leq C \int_{-\infty}^{+ \infty} \| \hat{\varphi}(\cdot , \xi) \|_{X_4}^2 \, d \xi
\leq C \|\varphi \|_{H_r^4(D)}^2 .
\ee
This finishes the proof of Lemma \ref{lemma3-4-1}.
\end{proof}

Let
\be \nonumber
\Bv^* = v^r \Be_r + v^z \Be_z = \partial_z \psi \Be_r - \frac{\partial_r (r \psi)}{r } \Be_z ,
\ee
and
\begin{equation}\nonumber
\Bo^\theta(x, y, z) = \omega^\theta(r, z)\Be_\theta  = (\mathcal{L} + \partial_z^2) \psi \Be_\theta.
\end{equation}
Now we are ready to investigate the regularity of $\Bv^*$ and $\Bo^\theta$. First, we give the $L^2(\Omega)$-bound of $\Bv^*$ and $\Bo^\theta$.
\begin{lemma}\la{lemma3-4-2} Assume that $\psi \in H_*^4(D)$. There exists a constant $C$, independent of $\psi$, such that
\begin{equation}\nonumber
\| \Bv^* \|_{L^2(\Omega)} + \|\Bo^\theta \|_{L^2(\Omega)} \leq C \|\psi \|_{H_r^4(D)}.
\end{equation}
\end{lemma}

\begin{proof}
Note that $\widehat{v^r} =  i \xi  \hat{\psi} $ and $\widehat{v^z} = - \frac1r \partial_r (r \hat{\psi} ) $. By virtue of Lemma
\ref{lemma3-4-1}, one has
\be \la{3-4-21}
\|v^r\|_{L^2(\Omega)}^2 = \int_{-\infty}^{+\infty} \int_0^1 \xi^2 |\hat{\psi}|^2  r  \, dr d\xi
\leq C \|\psi\|_{H_r^4(D) }^2
\ee
and
\be \la{3-4-22}
\|v^z\|_{L^2(\Omega)}^2 = \int_{-\infty}^{+\infty} \int_0^1 \left| \frac{\partial}{\partial r} ( r\hat{\psi})       \right|^2  \frac1r  \, dr d\xi
\leq C \|\psi\|_{H_r^4(D)}^2.
\ee
It follows from \eqref{3-4-21} and \eqref{3-4-22} that
\be \la{3-4-23}
\|\Bv^*\|_{L^2(\Omega)} \leq C \|\psi \|_{H_r^4(D)}.
\ee

Since $\omega^\theta = ( \mathcal{L} + \partial_z^2) \psi$,  using  Lemma \ref{lemma3-4-1} again gives
\begin{equation}\nonumber
\int_{-\infty}^{+ \infty}  \int_0^1 | \widehat{\omega^\theta} |^2 r \, dr d\xi
= \int_{-\infty}^{+\infty} \int_0^1  | ( \mathcal{L} - \xi^2) \hat{\psi} |^2 r \, drd\xi
\leq C \|\psi \|_{H_r^4(D)}^2.
\end{equation}
Therefore,
\begin{equation}\la{3-4-25}
\|\Bo^\theta \|_{L^2(\Omega)} \leq C \|\psi \|_{H_r^4(D)}.
\end{equation}
This, together with \eqref{3-4-23}, proves Lemma \ref{lemma3-4-2}.
\end{proof}

Next, we derive the $H^2(\Omega)$-estimate for $\Bo^\theta$. In order to get the global regularity of $\Bo^\theta$ in $\Omega$,  we have to analyze the behavior of $\Bo^\theta$ near the axis very carefully since it may contain singularities there. The key point is to use the following equation for $\Bo^\theta$,
\begin{equation}\label{vorteq}
\Delta \Bo^\theta = ( \mathcal{L} + \partial_z^2)^2 \psi \Be_\theta ,\ \ \ \ \mbox{in}\ \Omega.
\end{equation}
However, there is no proper boundary conditions
for $\Bo^\theta $.  Hence our  strategy is to establish the $H^2$-estimate for $\Bo^\theta$ away from the axis first.

Given $0 < r_0 < 1$, denote $\Omega_{r_0} = ( B_1(0) \setminus
B_{r_0} (0) ) \times \mathbb{R}$.

\begin{lemma}\label{lemma3-4-3} Assume that $\psi \in H_*^4(D)$.
There exists a positive constant $C$ independent of $\psi$ such that
\begin{equation}\nonumber
\|\Bo^\theta \|_{H^2(\Omega_{r_0} )} \leq C \|\psi \|_{H_r^4(D)},
\end{equation}
where the constant $C$ may depend on $r_0$.
\end{lemma}


Combining the interior $H^2$-estimate for $\Bo^\theta$ and Lemma \ref{lemma3-4-3} together gives the following global $H^2$-estimate of $\Bo^\theta$.
\begin{lemma}\label{lemma3-4-4} Assume that $\psi \in H_*^4(D)$.
There exists a positive constant $C$ independent of $\psi$ such that
\begin{equation}\label{globalomega}
\|\Bo^\theta \|_{H^2(\Omega)} \leq C \|\psi \|_{H_r^4(D)}.
\end{equation}
\end{lemma}

The detailed proof of Lemmas \ref{lemma3-4-3} and \ref{lemma3-4-4} will be given in Appendix \ref{secappendix}.

Now we can improve  the regularity of $\Bv^*$ to $H^3(\Omega)$.

\begin{lemma}\label{theorem3-4-5}Assume that $\psi \in H_*^4(D)$. There exists a positive constant $C$ independent of $\psi$ such that
\be \label{H3est}
\|\Bv^* \|_{H^3(\Omega) } \leq C \|\psi \|_{H_r^4(D)} .
\ee
\end{lemma}

\begin{proof}
For every $0< r<1$, straightforward computations give that
\be \label{vorteq2}
\Delta \Bv^* = -{\rm curl}~\Bo^\theta,\ \ \ \ \mbox{in}\ \Omega_r.
\ee
 In fact,  the equation \eqref{vorteq2} holds on the whole domain $\Omega$. Suppose that $\Bp \in C_0^\infty(\Omega)$ is a  vector-valued function defined on $\Omega$ and $\supp(\Bp) \subseteq B_1(0) \times [-Z, Z]$,

\be \la{3-4-61} \ba
&\,\,\int_{-\infty}^{+\infty} \int_{B_1(0)\setminus B_r(0)} \Bv^* \cdot \Delta \Bp \, dxdy dz \\
  = &\,\, \int_{-\infty}^{+ \infty} \int_{B_1(0) \setminus B_r(0)} \Delta \Bv^* \cdot \Bp \, dxdydz
+  \int_{-\infty}^{+\infty} \int_{\partial B_r (0) } \frac{\partial \Bv^* }{\partial \Bn } \cdot \Bp \, dSdz \\
&\ \ \ \ \ \ \ \ - \int_{-\infty}^{+\infty} \int_{\partial B_r (0) } \Bv^* \cdot \frac{\partial \Bp }{\partial \Bn} \, dSdz.
\ea \ee
On $\partial B_r(0) \times \mathbb{R}$,  it holds that
\be \nonumber
\frac{\partial \Bv^*}{\partial \Bn } = \frac{\partial v^r}{\partial r} \Be_r
+ \frac{\partial v^z}{\partial r} \Be_z =  \frac{\partial^2 \psi}{\partial r \partial z} \Be_r - \mathcal{L}\psi
\Be_z.
\ee
Therefore,
\be \nonumber \ba
&\,\, \left|   \int_{-\infty}^{+ \infty} \int_{\partial B_r(0)}  \frac{\partial \Bv^* }{\partial \Bn} \cdot \Bp \, dSdz       \right| \\
 \leq &\,\,   \int_{-\infty}^{+ \infty} \int_{\partial B_r(0)}  \left|  \frac{\partial \psi}{\partial r} \right|   \left| \frac{\partial \Bp }{\partial z}  \right| \, dSdz +
 \int_{-\infty}^{+ \infty} \int_{\partial B_r(0)} | \mathcal{L}\psi|  |\Bp | \, dSdz\\
 \leq &\,\, C \sup_{\Omega} | \partial_z \Bp |  \int_{-Z}^{+ Z} \left| \frac{\partial \psi}{\partial r}(r, z) \right|  r  \, dz   + C \sup_{\Omega} | \Bp | \int_{-Z}^{+Z} | \mathcal{L} \psi (r, z) | r \, dz.
\ea
\ee
It follows from the proof of Lemma \ref{lemma3-3-1} that for every fixed $z \in \mathbb{R}$, one has
\be \nonumber
 \left|  r \frac{\partial \psi}{\partial r} \right|
= \left|  \frac{ \partial}{\partial r}(r \psi) - \psi \right|
\leq C (r |\ln r|^{\frac12} + r^{\frac34} )   \|\psi(\cdot , z ) \|_{X_4}
\ee
and
\be \nonumber
| r \mathcal{L} \psi(r, z) | \leq C r \|\psi(\cdot ,z) \|_{X_4}.
\ee
Hence,
\be \la{3-4-65}
\lim_{r \rightarrow 0+ } \left|   \int_{-\infty}^{+ \infty} \int_{\partial B_r(0)}  \frac{\partial \Bv^* }{\partial \Bn} \cdot \Bp \, dSdz       \right| = 0.
\ee

Similarly,
\be \la{3-4-66}
\lim_{r \rightarrow 0+ } \left| \int_{-\infty}^{+\infty} \int_{\partial B_r (0) } \Bv^* \cdot \frac{\partial \Bp }{\partial \Bn} \, dSdz \right| = 0.
\ee

It follows from \eqref{3-4-61}--\eqref{3-4-66} that the equation \eqref{vorteq2} holds in $\Omega$.
Applying the regularity theory for the elliptic equation \eqref{vorteq2} with homogeneous Dirichlet boundary condition for $\Bv^*$ (\cite{GT}) yields \eqref{H3est}.
This completes the proof of Lemma \ref{theorem3-4-5}.
\end{proof}

\begin{pro}\label{back} Let $\psi$ and $v^\theta$ be the solutions obtained in Propositions \ref{existence-stream}-\ref{existence-swirl}, the corresponding velocity field $\Bv$ is a strong solution to the problem \eqref{2-0-1} and \eqref{BC}, satisfying the estimates
\be \label{back1}
\|\Bv^*\|_{H^3(\Omega)} + \|\Bv^\theta\|_{H^2(\Omega) } \leq C (1 + \Phi )\| \BF\|_{H^1(\Omega)}.
\ee
\end{pro}

\begin{proof} Since $\psi$ is a solution to \eqref{2-0-4-1}, following similar argument as in the proof of Lemma \ref{theorem3-4-5}, one has
\be \label{vorticityeq}
{\rm curl}~\left( (\bBU \cdot \nabla ) \Bv^* + (\Bv^* \cdot \nabla) \bBU   \right) - {\rm curl}~(\Delta \Bv^*) = {\rm curl}~\BF^*.
\ee
where $\BF^* = F^r \Be_r + F^z \Be_z$.

Hence, there exists some function $P \in H^1(\Omega)$, such that
\be \label{ns-new}
 (\bBU \cdot \nabla ) \Bv^* + (\Bv^* \cdot \nabla) \bBU  - \Delta \Bv^* + \nabla P = \BF^*,
\ee
and $\Bv^*$ is a strong solution to the problem \eqref{2-0-1} and \eqref{BC}.  According to Lemma \ref{theorem3-4-5} and Proposition \ref{existence-stream}, one has
\be \nonumber
\|\Bv^* \|_{H^3(\Omega)} \leq C \|\psi\|_{H_r^4(D)}\leq C (1 + \Phi) \|f\|_{L_r^2(D)} \leq C (1 + \Phi) \|\BF^*\|_{H^1(\Omega)}.
\ee
This, together with the result in Proposition \ref{existence-swirl} for swirl velocity, finishes the proof of Proposition \ref{back}.
\end{proof}

\begin{remark}
The coefficient in the estimate \eqref{back1} depends on $\Phi$, and $\BF$ is required to be a function in $H^1(\Omega)$.  Our next task is to establish some uniform estimates independent of the flux $\Phi$ of the Hagen-Poiseuille flow, when $\BF$ belongs to $L^2(\Omega)$.
\end{remark}

\section{Uniform estimate independent of the flux}\label{sec-res}

In this section, we  give the  proof for Theorem \ref{thm1}. The existence of strong solutions to the problem  \eqref{2-0-1} and \eqref{BC} has been established. It suffices to give the uniform estimate.    First, we give the proof for the uniform estimate when the flux is not big. Then we deal with the case with large flux in terms of three subcases.  More precisely,  choosing some small constant $\epsilon_1 \in (0, 1)$, we divide the proof into three subcases: (1)\ $ |\xi| \leq \frac{1}{\epsilon_1 \Phi}$ \ (2)\ $|\xi | \geq \epsilon_1 \sqrt{\Phi}$ \ (3) \ $\frac{1}{\epsilon_1 \Phi} <  |\xi| < \epsilon_1 \sqrt{\Phi}$. Basic energy estimates are enough for the first two subcases, while the last subcase requires much more elaborate analysis for the boundary layers.
\subsection{Estimate for the case with small flux}
In this subsection, we give the estimate for the solutions of the problem \eqref{2-0-8}--\eqref{FBC} in terms of $\BF^*$.
\begin{pro}\label{smallflux}
Let $\psi$ be the solution obtained in Proposition \ref{existence-stream}, the corresponding velocity field $\Bv^*$ satisfies
\be
\|\Bv^*\|_{H^2(\Omega)} \leq C (1 + \Phi^2) \|\BF^*\|_{L^2(\Omega)}.
\ee
\end{pro}

\begin{proof}
It follows from \eqref{2-1-7}, Lemma \ref{lemmaHLP} and integration by parts that one has
\be \label{app1-1}
\Phi |\xi| \int_0^1 |\hat{\psi}|^2 r \, dr
\leq C \left( \int_0^1 | \widehat{F^r} | | \xi \hat{\psi}| r \, dr + \int_0^1 |\widehat{F^z} | \left| \frac{d}{dr} ( r \hat{\psi} ) \right| \, dr \right).
\ee
Similarly, the estimate \eqref{2-1-5} gives
\be \la{app1} \ba
&  \int_0^1 |\mathcal{L} \hat{\psi}|^2 r \, dr + 2\xi^2 \int_0^1 \left| \frac{d}{dr}( r \hat{\psi})    \right|^2 \frac1r \, dr + \xi^4 \int_0^1 |\hat{\psi} |^2 r \, dr  \\
\leq & \int_0^1 | \widehat{F^r} | | \xi \hat{\psi}| r \, dr + \int_0^1 |\widehat{F^z} | \left| \frac{d}{dr} ( r \hat{\psi} ) \right| \, dr +
\frac{4\Phi}{\pi} |\xi| \int_0^1 \left| \frac{d}{dr} (r \hat{\psi})  \right| |r \hat{\psi}| \, dr .
\ea
\ee
According to \eqref{app1-1}, one has
\be \la{app1-2}
\ba
 & \frac{4\Phi}{\pi} |\xi | \int_0^1 \left| \frac{d}{dr} (r \hat{\psi})  \right| |r \hat{\psi}| \, dr \\
\leq & C \Phi^{\frac12} \left[ |\xi|  \int_0^1 \left| \frac{d}{dr}(r \hat{\psi} )  \right|^2 \frac1r \, dr   \right]^{\frac12} \left( \Phi |\xi| \int_0^1 |\hat{\psi}|^2 r \, dr  \right)^{\frac12} \\
\leq & \frac14 ( 1  + \xi^2) \int_0^1 \left| \frac{d}{dr}(r \hat{\psi} )  \right|^2 \frac1r \, dr + C \Phi \left( \int_0^1 | \widehat{F^r} | | \xi \hat{\psi}| r \, dr + \int_0^1 |\widehat{F^z} | \left| \frac{d}{dr} ( r \hat{\psi} ) \right| \, dr \right)
\ea
\ee
By virtue of Lemma \ref{lemma1} and Cauchy-Schwarz inequality, one has
\be \la{app2}
  \int_0^1\left( |\mathcal{L} \hat{\psi}|^2 r \, + \xi^2  \left| \frac{d}{dr}(r \hat{\psi} )  \right|^2 \frac1r \, + \xi^4  |\hat{\psi}|^2 r \right)\, dr
\leq   C (1+ \Phi^2) \int_0^1  (|\widehat{F^r} |^2 + |\widehat{F^z}|^2 ) r \, dr .
\ee
Integrating \eqref{app2} with respect to $\xi$ yields
\be \la{app3}
\int_{-\infty}^{+\infty} \int_0^1 \left\{\left( |\mathcal{L} \hat{\psi}|^2 + \xi^4 |\hat{\psi}|^2 \right)r  \,
+  \xi^2 \left|  \frac{d}{dr}( r \hat{\psi})   \right|^2 \frac{1}{r} \right\}\, dr d\xi
\leq C(1 + \Phi^2) \| \BF^* \|_{L^2(\Omega)}^2 .
\ee

Since $\Bo^\theta = ( \mathcal{L} + \partial_z^2) \psi \Be_\theta$, it follows from  \eqref{app3}  that
\be \la{app3-1}
\|\nabla \Bv^* \|_{L^2(\Omega)} = \|\Bo^\theta \|_{L^2(\Omega)} \leq C (1 + \Phi) \|\BF^* \|_{L^2(\Omega)}.
\ee
And by Poincar\'e's inequality,
\be \la{app3-2}
\|\Bv^*\|_{H^1(\Omega)} \leq C (1 + \Phi) \| \BF^* \|_{L^2(\Omega)}.
\ee


Next, let us consider the $H^2$-norm of $\Bv^*$. As shown in Proposition \ref{back}, $\Bv^*$ satisfies the equation
\be \label{app3-3} \left\{  \ba
& \bar{U} \partial_z  \Bv^* + v^r  \partial_r \bBU -\Delta \Bv^*  + \nabla P = \BF^*, \ \ \ \mbox{in}\ \ \Omega, \\
& {\rm div}~\Bv^* = 0,\ \ \ \ \mbox{in} \ \ \Omega, \\
&\Bv^* = 0, \ \ \ \ \mbox{on} \ \ \partial\Omega.
\ea  \right.
\ee
According to the regularity theory for Stokes equations\cite[Lemma VI.1.2]{Galdi}, one has
\be \label{app3-4} \ba
\|\Bv^*\|_{H^2 (\Omega)} & \leq C\|\BF^*\|_{L^2(\Omega)} + \Phi \|\partial_z \Bv^*\|_{L^2(\Omega)} + \Phi \|v^r\|_{L^2(\Omega)} + C \|\Bv^*\|_{H^1(\Omega)} \\
& \leq C (1 + \Phi^2 ) \|\BF^*\|_{L^2(\Omega)}.
\ea
\ee
This finishes the proof of the proposition.
\end{proof}

\subsection{Uniform estimate for the case with large flux and low frequency}
In this subsection, we give the uniform estimate for the solutions of \eqref{2-0-8}--\eqref{FBC} with respect to the flux $\Phi$ when the flux is large and the frequency is low.
\begin{pro}\label{Bpropcase1}
Assume that $|\xi| \leq \frac{1}{\epsilon_1 \Phi}\leq 1 $. Let $\hat{\psi}(r, \xi)$ be a smooth solution of the problem
\eqref{2-0-8}--\eqref{FBC}, then one has
\be \label{B-1}
\int_0^1 |\mathcal{L} \hat{\psi} |^2 r \, dr + 2 \xi^2 \int_0^1 \left| \frac{d}{dr}(r \hat{\psi})   \right|^2 \frac{1}{r}\, dr
+ \xi^4 \int_0^1 |\hat{\psi}|^2 r\, dr \leq C (\epsilon_1) \int_0^1 |\what{\BF^*}|^2 r \, dr.
\ee
\end{pro}

\begin{proof}
It follows from \eqref{2-1-5}, \eqref{2-1-7}, Lemma \ref{lemmaHLP}, and integration by parts that one has
\be \label{B3} \ba
& \Phi |\xi| \int_0^1 |\hat{\psi}|^2 r\, dr + \Phi |\xi| \int_0^1 \frac{1-r^2}{r} \left|  \frac{d}{dr} ( r \hat{\psi}) \right|^2 \frac1r \, dr + \Phi |\xi|^3 \int_0^1 (1 - r^2) |\hat{\psi}|^2 r \, dr \\
\leq & C \left( \int_0^1 |\what{\BF^*}|^2 r \, dr \right)^{\frac12} \left( \int_0^1 \xi^2 |\hat{\psi}|^2 r \, dr + \int_0^1 \left| \frac{d}{dr}(r \hat{\psi} ) \right|^2 \frac1r \, dr   \right)^{\frac12},
 \ea \ee
and
\be \label{B4} \ba
& \int_0^1 |\mathcal{L} \hat{\psi}|^2 r \, dr + 2\xi^2 \int_0^1 \left| \frac{d}{dr}( r \hat{\psi} )  \right|^2 \frac1r \, dr + \xi^4 \int_0^1 |\hat{\psi}|^2 r \, dr  \\
\leq &  C \left( \int_0^1 |\what{\BF^*}|^2 r \, dr \right)^{\frac12} \left( \int_0^1 \xi^2 |\hat{\psi}|^2 r \, dr + \int_0^1 \left| \frac{d}{dr}(r \hat{\psi} ) \right|^2 \frac1r \, dr   \right)^{\frac12}+ \frac{4\Phi}{\pi} |\xi| \int_0^1 \left| \frac{d}{dr}( r \hat{\psi} )  \right| | r \hat{\psi}| \, dr .
\ea
\ee
According to Lemma \ref{lemma1} and \eqref{B3},
\be \label{B5} \ba
& \frac{4 \Phi}{\pi} |\xi| \int_0^1 \left| \frac{d}{dr} (  r \hat{\psi} ) \right| |r \hat{\psi}| \, dr \\
\leq & \frac14 \int_0^1 \left| \frac{d}{dr} ( r \hat{\psi} ) \right|^2 \frac1r \, dr + C \Phi^2 \xi^2 \int_0^1 |\hat{\psi}|^2 r \, dr \\
\leq & \frac14 \int_0^1 |\mathcal{L} \hat{\psi} |^2 r\, dr + \frac{C}{\epsilon_1} \left( \int_0^1 |\what{\BF^*}|^2 r \, dr  \right)^{\frac12}
\left[ \int_0^1 \left( \xi^2 |\hat{\psi}|^2 r + \left|   \frac{d}{dr}(r \hat{\psi}) \right|^2 \frac1r \right) \, dr   \right]^{\frac12} \\
\leq & \frac14 \int_0^1 |\mathcal{L} \hat{\psi} |^2 r\, dr + \frac14 \xi^2 \int_0^1 |\hat{\psi}|^2 r \, dr + \frac14 \int_0^1 \left| \frac{d}{dr} (r \hat{\psi}) \right|^2 \frac1r \, dr + C (\epsilon_1) \int_0^1 |\what{\BF^*}|^2 r \, dr \\
\leq & \frac12 \int_0^1 |\mathcal{L} \hat{\psi}|^2 r \, dr + \frac14 \xi^2 \int_0^1 \left| \frac{d}{dr} ( r \hat{\psi})  \right|^2 \frac1r \, dr +
C (\epsilon_1) \int_0^1 |\what{\BF^*}|^2 r \, dr .
\ea
\ee
Taking \eqref{B5} into \eqref{B4}, one completes  the proof of Proposition \ref{Bpropcase1}.
\end{proof}

Let
\be \nonumber
\chi_1 (\xi ) = \left\{ \ba & 1 , \ \ \ |\xi| \leq \frac{1}{\epsilon_1 \Phi} , \\
& 0,\ \ \ \ otherwise,  \ea  \right.
\ee
and $\psi_{low}$ be the function  such that
$
\what{\psi_{low}}  =  \chi_1 (\xi) \hat{\psi}.
$
Define
\be \nonumber
v^r_{low} = \partial_z \psi_{low},\ \ \ \ v^z_{low} = - \frac{\partial_r ( r \psi_{low} )}{r}, \ \ \ \ \Bv^*_{low} = v^r_{low}\Be_r + v^z_{low}\Be_z.
\ee
And similarly, we define $F^r_{low}, F^z_{low}, \BF^*_{low}$, $\Bo^\theta_{low}$.

\begin{pro}\label{Bpropcase1-1} The solution $\Bv^*$  satisfies
\be \label{B6}
\|\Bv^*_{low} \|_{H^2(\Omega)} \leq C \|\BF^*_{low}\|_{L^2(\Omega)},
\ee
where $C$ is a uniform constant independent of $\Phi$ and $\BF$.
\end{pro}

\begin{proof}
In fact, $\Bv^*_{low}$ is a strong solution to the following Stokes equations
\be \label{stokes-low} \left\{ \ba
& - \Delta \Bv^*_{low} + \nabla P = \BF^*_{low} - \bar{U} \partial_z \Bv^*_{low} - v^r_{low} \partial_r \bBU, \ \ \ \mbox{in}\ \Omega, \\
& {\rm div}~\Bv^*_{low} = 0,\ \ \ \ \mbox{in}\ \Omega, \\
& \Bv^*_{low}= 0,\ \ \ \ \mbox{on}\ \partial\Omega.
\ea \right. \ee
According to the regularity theory for Stokes equations\cite[Lemma VI.1.2]{Galdi}, one has
\be \label{B7}
\|\Bv^*_{low} \|_{H^2(\Omega)}
\leq C \|\BF^*_{low}\|_{L^2(\Omega)} + C \Phi \|\partial_z \Bv^*_{low} \|_{L^2(\Omega)} +
C \Phi\| v^r_{low} \|_{L^2(\Omega)} + C \|\Bv^*_{low}\|_{H^1(\Omega)}.
\ee
Herein, by virtue of the estimate \eqref{B-1} and Lemma \ref{lemma1},
\be \label{B8} \ba
\| v^r_{low} \|_{L^2(\Omega)} & \leq C \left( \int_{|\xi| \leq \frac{1}{\epsilon_1 \Phi} }\xi^2 \int_0^1 |\hat{\psi}|^2 r\, dr d\xi       \right)^{\frac12} \\
& \leq C \left( \int_{|\xi| \leq \frac{1}{\epsilon_1 \Phi} }\frac{1}{(\epsilon_1\Phi)^2} C(\epsilon_1)  \int_0^1 \left| \what{\BF^*}  \right|^2 \frac1r \, dr d\xi       \right)^{\frac12} \\
& \leq C \Phi^{-1} \|\BF^*_{low}\|_{L^2(\Omega)}
\ea
\ee
and
\be \label{B9} \ba
& \| \partial_z \Bv^*_{low}\|_{L^2(\Omega)}  \leq \|\partial_z v^r_{low}\|_{L^2(\Omega)} + \|\partial_z v^z_{low} \|_{L^2(\Omega)} \\
\leq & C \left(  \int_{|\xi| \leq \frac{1}{\epsilon_1 \Phi} } \xi^4 \int_0^1  |\hat{\psi}|^2 r \, dr d\xi \right)^{\frac12} + C  \left(
\int_{|\xi| \leq \frac{1}{\epsilon_1 \Phi} } \xi^2 \int_0^1 \left| \frac{\partial}{\partial r} ( r \hat{\psi} ) \right| \frac1r    \, dr d\xi         \right)^{\frac12} \\
\leq & C \Phi^{-1} \|\BF^*_{low}\|_{L^2(\Omega)} .
\ea \ee
Moreover, by Poincar\'e's inequality, one has
\be \label{B10} \ba
& \|\Bv^*_{low} \|_{H^1(\Omega)}\leq C \|\nabla \Bv^*_{low} \|_{L^2(\Omega)}  = C \|\Bo^\theta_{low}\|_{L^2(\Omega)} \\
\leq & C \left( \int_{|\xi| \leq \frac{1}{\epsilon_1 \Phi} }  \int_0^1  | ( \mathcal{L} - \xi^2) \hat{\psi}|^2 r \, dr d\xi \right)^{\frac12}  \\
\leq &   C \| \BF^*_{low} \|_{L^2(\Omega)}.
\ea \ee
Summing \eqref{B7}-\eqref{B10} together gives \eqref{B6}. This completes the proof of Proposition \ref{Bpropcase1-1}.
\end{proof}

\subsection{Uniform estimate for the case with large flux and high frequency}
In this subsection, we give the uniform estimate for the solutions of \eqref{2-0-8}--\eqref{FBC} with respect to the flux $\Phi$ when the flux is large and the frequency is high.
\begin{pro}\label{Bpropcase2}
Assume that $|\xi| \geq \epsilon_1 \sqrt{\Phi} \geq 1 $. Let $\hat{\psi}$ be a  smooth solution to the problem \eqref{2-0-8}--\eqref{FBC}, then one has
\be \label{highf1}
 \int_0^1 |\mathcal{L} \hat{\psi}|^2 r \, dr + \xi^2 \int_0^1 \left| \frac{d}{dr} ( r \hat{\psi})    \right|^2 \frac{1}{r} \, dr
+ \xi^4 \int_0^1 |\hat{\psi}|^2 r \, dr
\leq C |\xi|^{-2}  \int_0^1 |\what{\BF^*}|^2 r \, dr ,
\ee
and
\be \label{highf2}
\Phi |\xi| \int_0^1 \frac{1 - r^2}{r} \left| \frac{d}{dr} ( r \hat{\psi}) \right|^2 \, dr +  \Phi |\xi|^3 \int_0^1 (1 - r^2) |\hat{\psi}|^2 r \, dr
\leq C |\xi|^{-2}  \int_0^1 |\what{\BF^*}|^2 r \, dr .
\ee
\end{pro}

\begin{proof} According to \eqref{B3}, one has
\be \label{B11} \ba
& \frac{4\Phi}{\pi} |\xi| \int_0^{\frac12} \left| \frac{d}{dr}(r \hat{\psi} )   \right| |r \hat{\psi}| \, dr \\
\leq & C \Phi |\xi| \left( \int_0^1 \frac{1- r^2}{r} \left| \frac{d}{dr} ( r \hat{\psi} ) \right|^2 \, dr   \right)^{\frac12}
 \left( \int_0^{\frac12} |\hat{\psi}|^2 r \, dr  \right)^{\frac12} \\
 \leq & C (\Phi |\xi|)^{\frac12} \left(  \int_0^1 |\what{\BF^*} |^2 r \, dr \right)^{\frac14} \left[ \int_0^1 \left( \xi^2 |\hat{\psi}|^2 r + \left|
 \frac{d}{dr} ( r\hat{\psi} ) \right|^2 \frac1r   \right) \, dr   \right]^{\frac14}   \left( \int_0^{\frac12} |\hat{\psi}|^2 r \, dr  \right)^{\frac12}  \\
 \leq & C \Phi^{\frac12} |\xi|^{-2} \left(  \int_0^1 |\what{\BF^*} |^2 r \, dr \right)^{\frac14} \left[ \int_0^1 \left( \xi^4 |\hat{\psi}|^2 r + \xi^2 \left|
 \frac{d}{dr} ( r  \hat{\psi} ) \right|^2 \frac1r   \right) \, dr   \right]^{\frac34}.
\ea
\ee
Using Hardy's inequality yields
\be \label{B12} \ba
& \frac{4\Phi}{\pi} |\xi| \int_{\frac12}^1  \left| \frac{d}{dr}(r \hat{\psi} )   \right| |r \hat{\psi}| \, dr \\
\leq & C  \Phi |\xi| \left( \int_{\frac12}^1  \frac{1- r^2}{r} \left| \frac{d}{dr} ( r \hat{\psi} ) \right|^2 \, dr   \right)^{\frac12}
 \left( \int_{\frac12}^1  |\hat{\psi}|^2 \frac{r}{1 -r^2}  \, dr  \right)^{\frac12} \\
 \leq & C \Phi |\xi| \left( \int_{\frac12}^1  \frac{1- r^2}{r} \left| \frac{d}{dr} ( r \hat{\psi} ) \right|^2 \, dr   \right)^{\frac12}
 \left( \int_{\frac12}^1  |r \hat{\psi}|^2 \, dr  \right)^{\frac14 }
 \left( \int_{\frac12}^1 \left|  \frac{ r\hat{\psi} }{1- r}  \right|^2  \, dr  \right)^{\frac14} \\
 \leq & C  \Phi |\xi|^{-\frac12} \left( \int_{\frac12}^1  \frac{1- r^2}{r} \left| \frac{d}{dr} ( r \hat{\psi} ) \right|^2 \, dr   \right)^{\frac12}
 \left( \xi^4  \int_{\frac12}^1  |r \hat{\psi}|^2 \, dr  \right)^{\frac14 }
 \left( \xi^2  \int_{\frac12}^1 \left|  \frac{d}{dr} ( r \hat{\psi} )  \right|^2  \, dr  \right)^{\frac14} \\
 \leq & C \Phi^{\frac12} |\xi|^{-\frac32} \left(  \int_0^1 |\what{\BF^*} |^2 r \, dr \right)^{\frac14}   \left[ \int_0^1 \left( \xi^4 |\hat{\psi}|^2 r + \xi^2 \left|
 \frac{d}{dr} ( r \hat{\psi} ) \right|^2 \frac1r   \right) \, dr   \right]^{\frac34}.
\ea
\ee
Taking \eqref{B11}-\eqref{B12} into \eqref{B4} and using Young's inequality, one completes the proof of \eqref{highf1}.

According to \eqref{B3} and \eqref{highf1}, one has
\be \label{B12-1} \ba
& \Phi |\xi| \int_0^1 \frac{1 - r^2}{r} \left| \frac{d}{dr} ( r \hat{\psi}) \right|^2 \, dr +  \Phi |\xi|^3 \int_0^1 (1 - r^2) |\hat{\psi}|^2 r \, dr \\
\leq & C \left( \int_0^1 |\what{\BF^*}|^2 r \, dr  \right)^{\frac12} \left( \int_0^1 \left| \frac{d}{dr}( r \hat{\psi})  \right|^2 \frac1r + \xi^2 |\hat{\psi}|^2 r \, dr             \right)^{\frac12} \\
\leq & C |\xi|^{-2} \int_0^1 |\what{\BF^*}|^2 r \, dr .
\ea
\ee
This completes the proof of Proposition \ref{Bpropcase2}.
\end{proof}

\begin{remark}
It follows from the proof of Proposition \ref{Bpropcase2} that the estimates \eqref{highf1} and \eqref{highf2} can be improved when $|\xi|$ is much bigger than $\epsilon_1\sqrt{\Phi}$. The reason that we choose the high frequency region as $\xi \in [\epsilon_1\sqrt{\Phi}, +\infty)$ is that we can only deal with the intermediate frequency region as $[\frac{1}{\epsilon_1\Phi}, \epsilon_1\sqrt{\Phi}]$ by our current analysis in Subsection \ref{sec-intermediate}.
\end{remark}

Let
\be \nonumber
\chi_2 (\xi ) = \left\{ \ba & 1 , \ \ \ |\xi| \geq \epsilon_1 \sqrt{\Phi} , \\
& 0,\ \ \ \ \text{otherwise},  \ea  \right.
\ee
and $\psi_{high}$ be the function  such that
$
\what{\psi_{high}}  =  \chi_2 (\xi) \hat{\psi}.
$
Define
\be \nonumber
v^r_{high} = \partial_z \psi_{high},\ \ \ \ v^z_{high} = - \frac{\partial_r ( r \psi_{high} )}{r}, \ \ \ \ \Bv^*_{high} = v^r_{high}\Be_r + v^z_{high}\Be_z.
\ee
And similarly, we define $F^r_{high}, F^z_{high}, \BF^*_{high}$, $\Bo_{high}^\theta $.

\begin{pro}\label{Bpropcase2-1} The solution $\Bv^*$  satisfies
\be \label{B16}
\|\Bv^*_{high} \|_{H^2(\Omega)} \leq C (1 + \Phi^{\frac14} ) \|\BF^*_{high} \|_{L^2(\Omega)},
\ee
and
\be \label{B16-1}
\|\Bv^*_{high} \|_{H^{\frac53} (\Omega)} \leq C  \|\BF^*_{high} \|_{L^2(\Omega)},
\ee
where $C$ is a uniform constant independent of $\Phi$ and $\BF$.
\end{pro}

\begin{proof}
In fact, $\Bv^*_{high}$ is a strong solution to the following Stokes equations
\be \label{stokes-high}
- \Delta \Bv^*_{high} + \nabla P = \BF^*_{high} - \bar{U} \partial_z \Bv^*_{high} - v^r_{high} \partial_r \bBU, \ \ \ \mbox{in}\ \Omega.
\ee
According to the regularity theory for Stokes equations\cite[Lemma VI.1.2]{Galdi}, one has
\be \label{B17} \ba
\|\Bv^*_{high}\|_{H^2(\Omega)}
\leq &  C \|\BF^*_{high}\|_{L^2(\Omega)}
+ C \Phi \| (1 - r^2) \partial_z \Bv^*_{high} \|_{L^2(\Omega)} \\
&\ \ \ \ + C \Phi \|v^r_{high}\|_{L^2(\Omega)} + C \|\Bv^*_{high}\|_{H^1(\Omega)} .
\ea
\ee
It follows from \eqref{highf1} and \eqref{highf2} that
\be \label{B18} \ba
& \Phi \|(1 - r^2) \partial_z \Bv^*_{high}\|_{L^2(\Omega)}  \\
\leq & C \left\{ \int_{|\xi| \geq \epsilon_1 \sqrt{\Phi}} \Phi^2  \left[ \xi^4 \int_0^1 (1 - r^2)^2  |\hat{\psi}|^2 r \, dr + \xi^2 \int_0^1 (1 -r^2)^2 \left|  \frac{d}{dr} ( r \hat{\psi})  \right|^2 \, dr   \right] \frac1r \, d\xi   \right\}^{\frac12} \\
\leq & C \left( \int_{|\xi| \geq \epsilon_1 \sqrt{\Phi} } \Phi |\xi|^{-1} \int_0^1 |\hat{\BF}|^2 r \, dr d\xi  \right)^{\frac12}\\
\leq & C \Phi^{\frac14} \|\BF^*_{high}\|_{L^2(\Omega)},
\ea
\ee
and
\be \label{B19} \ba
\Phi \|v^r_{high}\|_{L^2(\Omega)}&  \leq C \Phi \left( \int_{|\xi| \geq \epsilon_1 \sqrt{\Phi}} \xi^2 \int_0^1 |\hat{\psi} |^2 r \, dr d\xi   \right)^{\frac12}\\
& \leq C  \|\BF^*_{high}\|_{L^2(\Omega)}.
\ea
\ee
Following the same argument as in the proof of \eqref{B10} yields
\be \label{B20}
\|\Bv^*_{high} \|_{H^1(\Omega)} \leq C \|\Bo^\theta_{high}\|_{L^2(\Omega)} \leq C \Phi^{-\frac12}  \| \BF^*_{high}\|_{L^2(\Omega)}.
\ee
Hence, substituting  \eqref{B18}-\eqref{B20} into \eqref{B17} gives \eqref{B16}.
The estimate  \eqref{B16-1} can be obtained via the interpolation between $H^2(\Omega)$ and $H^1(\Omega)$. This finishes the proof of the proposition.
\end{proof}

\subsection{Uniform estimate for the case with large flux  and  intermediate frequency}\label{sec-intermediate}
In this subsection, we give the uniform estimate for the solutions of \eqref{2-0-8}--\eqref{FBC} with respect to the flux $\Phi$ when the flux is large and the frequency is intermediate. The analysis in this case is much more involved and is inspired by \cite{M}.

\begin{pro}\label{Bpropcase3} Assume that $\Phi \gg 1$.
There exists a small constant $\epsilon_1 \in (0, 1)$, such that as long as $\frac{1}{\epsilon_1 \Phi}
\leq |\xi| \leq \epsilon_1 \sqrt{\Phi}$,
the solution $\hat{\psi}(r)$ to the problem \eqref{2-0-8}-\eqref{FBC} can be decomposed into four parts,
\be \label{case3-1}
\hat{\psi}(r) = \what{\psi_s}(r) + a I_1(|\xi|r) + b\left[ \chi \what{\psi_{BL}}(r) + \what{\psi_e}(r) \right] .
\ee
Here $(1)$\ $\widehat{\psi_s}$ is a solution to the following equation
\be \label{slip}
\left\{ \ba   & i \xi \bar{U}(r) ( \mathcal{L} - \xi^2) \what{\psi_s} - (\mathcal{L} - \xi^2)^2 \what{\psi_s} = \hat{f} = \frac{d}{dr} \what{F^z}- i\xi \what{F^r}, \\
& \what{\psi_s}(0) = \what{\psi_s}(1) = \mathcal{L} \what{\psi_s} (0) = \mathcal{L} \what{\psi_s} (1) = 0,
\ea      \right.
\ee
with the estimates
\be \label{case3-3}
\int_0^1 |\what{\psi_s}|^2 r \, dr \leq C (\Phi |\xi|)^{-\frac53} \int_0^1 |\what{\BF^*}|^2 r \, dr ,
\ee
\be \label{case3-3-1}
\int_0^1 \left|\frac{d}{dr}(r \what{\psi}_s) \right|^2  \frac{1}{r}\, dr
+ \xi^2 \int_0^1 \left| \what{\psi}_s \right|^2 r \, dr
\leq C (\Phi |\xi|)^{- \frac43 } \int_0^1 |\what{\BF^*}|^2 r \, dr .
\ee
\be \label{case3-4}
\int_0^1 | \mathcal{L} \what{\psi_s}|^2 r \, dr + \xi^2 \int_0^1 \left|  \frac{d}{dr} ( r \what{\psi_s} )\right|^2 \frac1r \, dr + \xi^4
\int_0^1 |\what{\psi_s}|^2  r \, dr \leq C ( \Phi |\xi|)^{-\frac23} \int_0^1 |\what{\BF^*}|^2 r \, dr .
\ee
\be \label{case3-5} \ba
& \int_0^1 \left| \frac{d}{dr}( r \mathcal{L} \what{\psi_s} )\right|^2 \frac1r \, dr +
\xi^2 \int_0^1 |\mathcal{L} \what{\psi_s} |^2 r \, dr + \xi^4 \int_0^1 \left|  \frac{d}{dr}( r \what{\psi_s})       \right|^2 \frac1r \, dr +\xi^6 \int_0^1 |\what{\psi_s} |^2 \, dr \\
\leq  & C \int_0^1 |\what{\BF^*}|^2 r\, dr .
\ea
\ee

$(2)$\ $I_1(\rho)$ is the modified Bessel function of the first kind, i.e.,
\be \label{eqBessel1}
\left\{
\ba & z^2 \frac{d^2}{dz^2} I_1 (z) + z \frac{d}{dz} I_1 (z)  - (z^2 + 1) I_1 (z) = 0, \\
& I_1 (0) = 0,\quad I_1(z)>0 \,\,\text{if}\,\, z>0.
\ea
\right.
\ee
Furthermore, $a$ is a constant satisfying
\be \label{case3-6}
|a| \leq C (\Phi|\xi|)^{-\frac56} I_1(\xi)^{-1} \left( \int_0^1 |\what{\BF^*}|^2 r \, dr      \right)^{\frac12}.
\ee

$(3)$ \ $\what{\psi_{BL}}$ is the boundary layer profile of the form
\be \label{case3-7}
\what{\psi_{BL}}(r)  = G_{\xi, \Phi} ( |\beta| (1 - r)) \quad \text{with}\quad |\beta|=\left(\frac{4|\xi|\Phi}{\pi}\right)^{1/3},
\ee
where $G_{\xi, \Phi}(\rho) $ is a smooth function, decaying exponentially at infinity, uniformly bounded in the set
\be \nonumber
\mcE = \left\{ ( \xi, \Phi, \rho):\ \Phi \geq 1, \ \frac{1}{\Phi} \leq |\xi| \leq \sqrt{\Phi},\  0\leq \rho < + \infty   \right\}.
\ee
Moreover,  $\chi$ is a  smooth increasing cut-off function, which satisfies
\be \label{defchi}
\chi (r) = \left\{ \ba  &  1, \ \ \ \ r\geq \frac12,  \\ & 0, \ \ \  \ r \leq \frac14,   \ea  \right.
\ee
and the constant $b$ satisfies
\be \label{case3-8}
|b|\leq C (\Phi |\xi|)^{-\frac{5}{6}} \left( \int_0^1 |\what{\BF^*}|^2 r \, dr  \right)^{\frac12}.
\ee

$(4)$\ $\widehat{\psi_e}$ is a remainder term, which satisfies
\be \label{case3-9}
\left\{  \ba  &  i\xi \bar{U}(r) ( \mathcal{L} - \xi^2) (\chi \what{\psi_{BL}} + \what{\psi_e}) - (\mathcal{L} - \xi^2)^2 (\chi \what{\psi_{BL}}+ \what{\psi_e} ) = 0,   \\
& \what{\psi_e}(0) = \what{\psi_e}(1) = \mathcal{L}\what{\psi_{e}}(0) = \mathcal{L} \what{\psi_e}(1) = 0.
\ea  \right.
\ee
Here $\chi$ is a cut-off function defined in \eqref{defchi}. And $\what{\psi_e}$ satisfies the following estimates,
\be \label{case3-10}
\int_0^1 |\what{\psi_e}|^2 r \, dr \leq C (\Phi |\xi|)^{-\frac{1}{3}},
\ee
\be \label{case3-11}
\int_0^1 |\mathcal{L} \what{\psi_e}|^2 r \, dr + \xi^2 \int_0^1 \left|\frac{d}{dr}(r \what{\psi_e}) \right|^2 \frac1r \, dr + \xi^4 \int_0^1 |\what{\psi_e}|^2 r\, dr
\leq C (\Phi |\xi|)^{\frac{17}{21}}
\ee
\be \label{case3-12}
\int_0^1 \left| \frac{d}{dr}( r \mathcal{L} \what{\psi_e}) \right|^2 \frac1r \, dr + \xi^2  \int_0^1 |\mathcal{L}\what{\psi_e}|^2 r\, dr +
\xi^4\int_0^1 \left| \frac{d}{dr}( r \what{\psi_e})  \right|^2 \frac1r \, dr + \xi^6 \int_0^1 |\what{\psi_e}|^2 r\, dr
\leq C (\Phi|\xi|)^{\frac53}.
\ee

In conclusion, $\hat{\psi}$ satisfies the following estimate,
\be \label{case3-15}
\ba
& \int_0^1 |\mathcal{L} \hat{\psi}|^2 r\, dr + \int_0^1 \left| \frac{d}{dr}( r\hat{\psi}) \right|^2 \frac1r \, dr + \int_0^1 |\hat{\psi}|^2 r \, dr \\
&\ + \int_0^1 \left| \frac{d}{dr}( r \mathcal{L} \hat{\psi}) \right|^2 \frac1r \, dr + \xi^2  \int_0^1 |\mathcal{L}\hat{\psi}|^2 r\, dr +
\xi^4\int_0^1 \left| \frac{d}{dr}( r \hat{\psi})  \right|^2 \frac1r \, dr + \xi^6 \int_0^1 |\hat{\psi}|^2 r\, dr \\
\leq & C \int_0^1 |\what{\BF^*}|^2 \, dr .
\ea
\ee

\end{pro}

Before the proof of Proposition \ref{Bpropcase3}, we  consider the linear problem \eqref{slip}, which is similar to \eqref{2-0-8}--\eqref{FBC}, but with a different boundary condition. The vorticity vanishes on the boundary for this problem.
\begin{lemma}\label{propslip}
Given $\hat{f} \in L_r^2(0, 1)$, the system \eqref{slip} admits a unique solution $\what{\psi_s}$  satisfying  the estimates
\be \label{propslip1-1}
\int_0^1 |\what{\psi_s}|^2 r \, dr \leq C (\Phi |\xi|)^{-2} \int_0^1 |\hat{f}|^2 r \, dr ,
\ee
\be \label{propslip1-2}
\int_0^1 | \mathcal{L} \what{\psi_s}|^2 r \, dr + \xi^2 \int_0^1 \left|  \frac{d}{dr} ( r \what{\psi_s} )\right|^2 \frac1r \, dr + \xi^4
\int_0^1 |\what{\psi_s}|^2  r \, dr \leq C ( \Phi |\xi|)^{-\frac67} \int_0^1 |\hat{f}|^2 r \, dr .
\ee
\be \label{propslip1-3} \ba
& \int_0^1 \left| \frac{d}{dr}( r \mathcal{L} \what{\psi_s} )\right|^2 \frac1r \, dr +
\xi^2 \int_0^1 |\mathcal{L} \what{\psi_s} |^2 r \, dr + \xi^4 \int_0^1 \left|  \frac{d}{dr}( r \what{\psi_s})       \right|^2 \frac1r \, dr +\xi^6 \int_0^1 |\what{\psi_s} |^2 \, dr \\
\leq  & C \int_0^1 |\hat{f}|^2 r\, dr .
\ea
\ee

Moreover, the solution $\what{\psi_s}$ also satisfies the estimates \eqref{case3-3}--\eqref{case3-5} and
\begin{equation}\label{propslip2}
\left| \frac{d}{dr} ( r \widehat{\psi_s} ) \right| (1) \leq C ( \Phi |\xi|)^{-\frac12} \left( \int_0^1 |\what{\BF^*}|^2 r \, dr\right)^{\frac12} .
\end{equation}
\end{lemma}

\begin{proof}
We first show the a priori estimates \eqref{propslip1-1}-\eqref{propslip1-3}.  Similar to the proof of Lemma \ref{lemapri1}, multiplying the equation in \eqref{slip} by $r\overline{\what{\psi_s}} $ and integrating over $[0, 1]$ yield
\be \label{5-51}
\xi \int_0^1 \frac{\bar{U}(r)}{r} \left|\frac{d}{dr} ( r \what{\psi_s})   \right|^2 \, dr + \xi^3
\int_0^1 \bar{U}(r) |\what{\psi_s}|^2 r \, dr = - \Im \int_0^1 \hat{f} \overline{\what{\psi_s}} r \, dr, \ee
and
\be \label{5-52} \ba
& \int_0^1 |\mathcal{L} \what{\psi_s} |^2 r \, dr + 2\xi^2 \int_0^1 \left| \frac{d}{dr} ( r \what{\psi_s} )\right|^2 \frac1r \, dr + \xi^4 \int_0^1 |\what{\psi_s}|^2 r \, dr \\
= & - \Re \int_0^1 \hat{f} \overline{\what{\psi_s}} r \, dr - \frac{4 \Phi}{\pi} \xi \Im \int_0^1
\left[ \frac{d}{dr}( r \what{\psi_s}) r \overline{\what{\psi_s}}  \right] \, dr .
\ea
\ee
The inequality \eqref{5-51} together with Lemma \ref{lemmaHLP} implies
\be \label{5-53}
\Phi |\xi| \int_0^1 \left|  \frac{d}{dr}( r \what{\psi_s} ) \right|^2 \frac{1 - r^2}{r} \, dr + \Phi |\xi| \int_0^1 |\what{\psi_s}|^2 r \, dr
\leq C \int_0^1 |\hat{f}| |\what{\psi_s}| r\, dr.
\ee
By  H\"older inequality, one has
\be \label{5-55}
\Phi^2 \xi^2 \int_0^1 \left|  \frac{d}{dr}( r \what{\psi_s} ) \right|^2 \frac{1 - r^2}{r} \, dr + \Phi^2 \xi^2  \int_0^1 |\what{\psi_s}|^2 r \, dr
\leq C \int_0^1 |\hat{f}|^2  r\, dr.
\ee
Consequently, one has
\be \label{5-56} \ba
 & \frac{4\Phi}{\pi} |\xi| \int_0^{\frac12} \left| \frac{d}{dr}(r \what{\psi_s} ) \right| | r\what{\psi_s}| \, dr \\
\leq & C \Phi |\xi| \left(  \int_0^{\frac12} \left| \frac{d}{dr} ( r \what{\psi_s})  \right|^2 \frac{1 - r^2}{r} \, dr       \right)^{\frac12 } \left( \int_0^{\frac12} |r \what{\psi_s}|^2 \frac{r}{1 - r^2} \, dr  \right)^{\frac12}\\
\leq & C \Phi |\xi| \left(  \int_0^1 \left| \frac{d}{dr} ( r \what{\psi_s} )   \right|^2 \frac{1-r^2}{r} \, dr     \right)^{\frac12} \left( \int_0^1  |\what{\psi_s} |^2 r \, dr  \right)^{\frac12} \\
\leq & C \Phi |\xi|  \left(  \int_0^1 \left| \frac{d}{dr} ( r \what{\psi_s} )   \right|^2 \frac{1-r^2}{r} \, dr     \right)^{\frac12} \left( \int_0^1  |\mathcal{L} \what{\psi_s} |^2 r \, dr  \right)^{\frac18} \left( \int_0^1  |\what{\psi_s} |^2 r \, dr  \right)^{\frac38} \\\\
\leq & \frac14 \int_0^1 |\mathcal{L} \what{\psi_s} |^2 r \, dr +  C (\Phi |\xi|)^{\frac17}  \int_0^1 |\hat{f}| |\hat{\psi_s}| r \, dr.
\ea \ee
Furthermore, it follows from  Hardy's inequality, Lemma \ref{lemma1} and \eqref{5-53} that one has
\be \label{5-57}
\ba
 & \frac{4\Phi}{\pi} |\xi| \int_{\frac12}^{1} \left| \frac{d}{dr}(r \what{\psi_s} ) \right| | r\what{\psi_s}| \, dr \\
\leq & C \Phi |\xi| \left(  \int_{\frac12}^{1} \left| \frac{d}{dr} ( r \what{\psi_s})  \right|^2 \frac{1 - r^2}{r} \, dr       \right)^{\frac12 } \left( \int_{\frac12}^1 \frac{|r\what{\psi_s}|^2}{(1 - r)^2 } \, dr  \right)^{\frac14} \left(  \int_{\frac12}^1 |\what{\psi_s}|^2 r \, dr \right)^{\frac14}\\
\leq & C \Phi |\xi| \left(  \int_{0}^{1} \left| \frac{d}{dr} ( r \what{\psi_s})  \right|^2 \frac{1 - r^2}{r} \, dr       \right)^{\frac12 } \left( \int_0^1 \left| \frac{d}{dr} ( r \what{\psi_s} ) \right|^2 \frac1r \, dr    \right)^{\frac14} \left(  \int_{0}^1 |\what{\psi_s}|^2 r \, dr \right)^{\frac14}\\
\leq & C \Phi |\xi| \left(  \int_{0}^{1} \left| \frac{d}{dr} ( r \what{\psi_s})  \right|^2 \frac{1 - r^2}{r} \, dr       \right)^{\frac12 } \left( \int_0^1 |\mathcal{L} \what{\psi_s}|^2 r \, dr   \right)^{\frac18} \left(   \int_0^1 |\what{\psi_s}|^2 r\, dr   \right)^{\frac38}\\
\leq & \frac14 \int_0^1 |\mathcal{L} \what{\psi_s}|^2 r \,dr + C ( \Phi |\xi|)^{\frac17} \int_0^1 |\hat{f}| |\what{\psi_s}| r \, dr .
\ea \ee
Hence, taking the estimates \eqref{5-56}--\eqref{5-57} into \eqref{5-52} gives
\be \label{5-58}
\int_0^1 |\mathcal{L} \what{\psi_s} |^2 r \, dr + \xi^2 \int_0^1 \left|  \frac{d}{dr} ( r \what{\psi_s})  \right|^2 \frac1r \, dr
+ \xi^4 \int_0^1 |\what{\psi_s}|^2 r \, dr
\leq C ( \Phi |\xi|)^{\frac17} \int_0^1 |\hat{f}| |\what{\psi_s}| r \, dr ,
\ee
which together with \eqref{5-53} implies that
\be \label{5-59}
\int_0^1 |\mathcal{L} \what{\psi_s} |^2 r \, dr + \xi^2 \int_0^1 \left|  \frac{d}{dr} ( r \what{\psi_s})  \right|^2 \frac1r \, dr
+ \xi^4 \int_0^1 |\what{\psi_s}|^2 r \, dr
\leq C (\Phi |\xi|)^{-\frac67}  \int_0^1 |\hat{f}|^2 r \, dr .
\ee

Moreover, multiplying \eqref{5-58} by $\xi^2$ and applying Lemma \ref{lemma1} yield
\be \label{5-61} \ba
& \xi^2 \int_0^1 |\mathcal{L} \what{\psi_s}|^2 r \, dr + \xi^4  \int_0^1 \left|  \frac{d}{dr} ( r \what{\psi_s})  \right|^2 \frac1r \, dr
+ \xi^6 \int_0^1 |\what{\psi_s}|^2 r \, dr \\
\leq & C (\Phi |\xi|)^{\frac17} \left( \int_0^1 |\hat{f}|^2 r \, dr \right)^{\frac12}
\left(  \int_0^1 |\what{\psi_s}|^2 r \, dr \right)^{\frac16} \left( \xi^6 \int_0^1 |\what{\psi_s}|^2 r\, dr  \right)^{\frac13} \\
\leq & \frac12 \xi^6 \int_0^1 |\what{\psi_s}|^2 r \, dr + C  ( \Phi |\xi|)^{\frac{3}{14}} \left( \int_0^1 |\hat{f}|^2 r \, dr\right)^{\frac34} \left( \int_0^1 | \mathcal{L} \what{\psi_s}|^2 r \, dr   \right)^{\frac14} \\
\leq & \frac12 \xi^6 \int_0^1 |\what{\psi_s}|^2 r \, dr + C   \int_0^1 |\hat{f}|^2 r \, dr ,
\ea
\ee
where the last inequality is due to \eqref{propslip1-2}. Therefore, one has
\be \label{5-62}
\xi^2 \int_0^1 |\mathcal{L}\what{\psi_s}|^2 r \, dr + \xi^4 \int_0^1 \left| \frac{d}{dr} ( r \what{\psi_s}) \right|^2 \frac1r \, dr
+ \xi^6 \int_0^1 |\what{\psi_s}|^2 r \, dr
\leq C \int_0^1 |\hat{f}|^2 r \, dr .
\ee

 Multiplying the equation in \eqref{slip} by $r\mathcal{L} \overline{\what{\psi_s}} $ and integrating over $[0, 1]$ give
\be \label{5-65} \ba
&\int_0^1 \left| \frac{d}{dr}( r \mathcal{L} \what{\psi_s}) \right|^2 \frac1r \, dr + 2\xi^2 \int_0^1 |\mathcal{L} \what{\psi_s}|^2 r \, dr
+ \xi^4 \int_0^1 \left| \frac{d}{dr}(r \what{\psi_s} )  \right|^2  \frac1r \, dr \\
= & - \Re \int_0^1 \hat{f} \mathcal{L} \overline{\what{\psi_s}} r \, dr + \frac{4 \Phi}{\pi }\xi^3 \Im \int_0^1 \left[\frac{d}{dr}( r \what{\psi_s} ) r \overline{\what{\psi_s} }  \right] \, dr.
\ea \ee
Note that
\be \label{5-66}
\left| \frac{4 \Phi}{\pi }\xi^3 \int_0^1 \left[\frac{d}{dr}( r \what{\psi_s} ) r \overline{\what{\psi_s} }  \right] \, dr \right|
\leq \frac12 \xi^4 \int_0^1 \left|  \frac{d}{dr} ( r \what{\psi_s})   \right|^2 \frac1r \, dr + \Phi^2 \xi^2 \int_0^1 |\what{\psi_s}|^2 r \, dr.
\ee
Hence, one has
\be \label{5-67}
\int_0^1 \left| \frac{d}{dr}( r \mathcal{L} \what{\psi_s}) \right|^2 \frac1r \, dr + \xi^2 \int_0^1 |\mathcal{L} \what{\psi_s}|^2 r \, dr
+ \xi^4 \int_0^1 \left| \frac{d}{dr}(r \what{\psi_s} )  \right|^2  \frac1r \, dr
\leq C \int_0^1 |\hat{f}|^2 r \, dr .
\ee

Next, we establish some a priori  estimates in terms of $\BF^*$.
Multiplying the equation in \eqref{slip} by $\xi^2r\overline{\what{\psi_s}} $ and integrating over $[0, 1]$ yield
\be \label{B-51}
\xi^3 \int_0^1 \frac{\bar{U}(r)}{r} \left|\frac{d}{dr} ( r \what{\psi_s})   \right|^2 \, dr + \xi^5
\int_0^1 \bar{U}(r) |\what{\psi_s}|^2 r \, dr = - \xi^2\Im \int_0^1 \hat{f} \overline{\what{\psi_s}} r \, dr, \ee
and
\be \label{B-52}
\begin{aligned}
& \xi^2\int_0^1 |\mathcal{L} \what{\psi_s} |^2 r \, dr + 2\xi^4 \int_0^1 \left| \frac{d}{dr} ( r \what{\psi_s} )\right|^2 \frac1r \, dr + \xi^6 \int_0^1 |\what{\psi_s}|^2 r \, dr \\
= & - \xi^2\Re \int_0^1 \hat{f} \overline{\what{\psi_s}} r \, dr - \frac{4 \Phi}{\pi} \xi^3 \Im \int_0^1
\left[ \frac{d}{dr}( r \what{\psi_s}) r \overline{\what{\psi_s}}  \right] \, dr .
\end{aligned}
\ee
If one multiplies the both sides by $\mathcal{L} \overline{ \what{\psi_s} }r$ and integrates over $[0, 1]$, then one has
\be \label{B-51.5}
\begin{aligned}
&i\xi^3 \int_0^1 \frac{\bar{U}(r)}{r} \left|\frac{d}{dr} ( r \what{\psi_s})   \right|^2 \, dr + i\xi
\int_0^1 \bar{U}(r) |\mathcal{L}\what{\psi_s}|^2 r \, dr  -i \frac{4 \Phi}{\pi} \xi^3 \int_0^1
\left[ \frac{d}{dr}( r \overline{\what{\psi_s}}) r {\what{\psi_s}}  \right] \, dr\\
&\quad + \int_0^1 \frac{1}{r}\left|\frac{d}{dr}(r\mathcal{L} \what{\psi_s}) \right|^2 r \, dr
+2\xi^2 \int_0^1 \frac{1}{r}\left|\mathcal{L} \what{\psi_s} \right|^2 r \, dr
+ \xi^4 \int_0^1 \left| \frac{d}{dr} ( r \what{\psi_s} )\right|^2 \frac1r \, dr  \\
= &  \int_0^1 \hat{f} \mathcal{L}\overline{\what{\psi_s}} r \, dr.
\end{aligned}
\ee
Summing \eqref{B-52}-\eqref{B-51.5} together yields
\begin{equation}
\begin{aligned}
&\xi^2\int_0^1 |\mathcal{L} \what{\psi_s} |^2 r \, dr + 2\xi^4 \int_0^1 \left| \frac{d}{dr} ( r \what{\psi_s} )\right|^2 \frac1r \, dr + \xi^6 \int_0^1 |\what{\psi_s}|^2 r \, dr \\
&\quad + \int_0^1 \frac{1}{r}\left|\frac{d}{dr}(r\mathcal{L} \what{\psi_s}) \right|^2 r \, dr
+2\xi^2 \int_0^1 \left|\mathcal{L} \what{\psi_s} \right|^2 r \, dr
+ \xi^4 \int_0^1 \left| \frac{d}{dr} ( r \what{\psi_s} )\right|^2 \frac1r \, dr  \\
= & \Re \int_0^1 \hat{f} \mathcal{L}\overline{\what{\psi_s}} r \, dr
 - \xi^2\Re \int_0^1 \hat{f} \overline{\what{\psi_s}} r \, dr.
 \end{aligned}
\end{equation}

Note that
\begin{equation}
\begin{aligned}
\left|\int_0^1 \hat{f} \mathcal{L}\overline{\what{\psi_s}} r \, dr\right|\leq & \left|\int_0^1 \what{F^r} i\xi \mathcal{L}\bar{\hat{\psi}}r +\what{F^z}
\frac{d}{dr}(r\mathcal{L}\overline{\what{\psi_s}})\, dr\right|\\
\leq & 2 \int_0^1 |\what{\BF^*}|^2 r \, dr +\frac14   \int_0^1  \left|\frac{d}{dr}(r\mathcal{L} \what{\psi_s}) \right|^2 \frac1r  \, dr
+\frac14 \xi^2 \int_0^1 \left|\mathcal{L} \what{\psi_s} \right|^2 r \, dr ,
\end{aligned}
\end{equation}
and
\begin{equation}
\begin{aligned}
\left|\int_0^1 \hat{f} \xi^2 \overline{\what{\psi_s}} r \, dr\right|\leq & \left|\int_0^1 \what{F^r} i\xi^3 \overline{\what{\psi_s}}r +\what{F^z} \xi^2\frac{d}{dr} (r\overline{\what{\psi_s}})dr\right|\\
\leq & 2 \int_0^1 |\what{\BF^*}|^2 rdr + \frac14 \xi^4 \int_0^1 \left|\frac{d}{dr}(r \what{\psi}_s) \right|^2 \frac1r  \, dr
+\frac14 \xi^6 \int_0^1 \left| \what{\psi}_s \right|^2 r \, dr
\end{aligned}
\end{equation}
Therefore, we have
\begin{equation}  \label{B-58}
\ba
&\xi^2\int_0^1 |\mathcal{L} \what{\psi_s} |^2 r \, dr + \xi^4 \int_0^1 \left| \frac{d}{dr} ( r \what{\psi_s} )\right|^2 \frac1r \, dr + \xi^6 \int_0^1 |\what{\psi_s}|^2 r \, dr  +  \int_0^1 \left|\frac{d}{dr}(r\mathcal{L} \what{\psi_s}) \right|^2  \frac1r \, dr  \\
\leq  &  C\int_0^1| \what{\BF^*}|^2  r \, dr.
 \ea
\end{equation}
This is exactly the estimate \eqref{case3-5}.

Furthermore, it follows from \eqref{B-51.5} that
\begin{equation} \label{B-59}
\begin{aligned}
&\frac{2\Phi |\xi|^3}{\pi}\int_0^1 \frac{1-r^2}{r} \left|\frac{d}{dr}(r\what{\psi_s})\right|^2 \, dr + \frac{2\Phi|\xi|}{\pi}\int_0^1 (1-r^2)|\mathcal{L}\what{\psi_s}|^2 r \, dr\\
=&\Im \int_0^1 \hat{f}\mathcal{L}\overline{ \what{\psi_s} } r\, dr\\
\leq & 2 \left(  \int_0^1 |\what{\BF^*}|^2 r \, dr \right)^{\frac12} \left[ \xi^2 \int_0^1 |\mathcal{L} \what{\psi_s}|^2 r \, dr + \int_0^1 \left|\frac{d}{dr} ( r \mathcal{L} \what{\psi_s} ) \right|^2 \frac1r \, dr  \right]^{\frac12}\\
\leq & C \int_0^1 |\what{\BF^*}|^2 r\, dr,
\end{aligned}
\end{equation}
where we used \eqref{B-58} to get the last inequality.
Hence, it follows from Lemma \ref{weightinequality}, Lemma \ref{lemma1} and \eqref{B-58}-\eqref{B-59} that one has
\begin{equation} \label{B-60}
\begin{aligned}
\int_0^1 |\mathcal{L}\what{\psi_s}|^2 r \, dr \leq& C \int_0^1 (1-r^2)|\mathcal{L}\what{\psi_s}|^2 r\, dr\\
&\quad + C\left(\int_0^1 (1-r^2)|\mathcal{L}\what{\psi_s}|^2 r \, dr\right)^{\frac23}\left(\int_0^1 \left|\frac{d}{dr}(r\mathcal{L}\what{\psi_s})\right|^2 \frac1r \, dr\right)^{\frac13}\\
\leq & C\left(\int_0^1 (1-r^2)|\mathcal{L}\what{\psi_s}|^2 r \, dr\right)^{\frac23}\left(\int_0^1 \left|\frac{d}{dr}(r\mathcal{L}\what{\psi_s})\right|^2 \frac1r \, dr\right)^{\frac13} \\
\leq & C (\Phi |\xi|)^{-\frac23} \int_0^1 |\hat{\BF}|^2 r \, dr.
\end{aligned}
\end{equation}
Using Lemmas \ref{weightinequality} and \ref{lemma1} again gives
\begin{equation} \label{B-61}
\begin{aligned}
\xi^2\int_0^1 |\what{\psi_s}|^2 r \, dr \leq & C \xi^2\int_0^1 (1-r^2)|\what{\psi_s}|^2 r \, dr\\
&\quad + C\xi^2\left(\int_0^1 (1-r^2)|\what{\psi_s}|^2 r \,  dr\right)^{\frac23}\left(\int_0^1 \left|\frac{d}{dr}(r\what{\psi_s})\right|^2  \frac1r \, dr\right)^{\frac13} \\
\leq & C\xi^2\left(\int_0^1 (1-r^2)|\what{\psi_s}|^2 r \,  dr\right)^{\frac23}\left(\int_0^1 \left|\frac{d}{dr}(r\what{\psi_s})\right|^2  \frac1r \, dr\right)^{\frac13}.
\end{aligned}
\end{equation}
According to \eqref{B-51} and \eqref{B-58}, one has
\begin{equation} \label{B-62}
\begin{aligned}
&\xi^2 \left(\Phi|\xi|\int_0^1 \frac{1-r^2}{r} \left|\frac{d}{dr}(r\what{\psi_s})\right|^2 \, dr +\Phi|\xi|^3 \int_0^1 (1-r^2) |\what{\psi_s}|^2 r \, dr\right)\\
\leq & \left(\int_0^1 |\what{F^r}|^2 r \, dr\right)^{\frac12}\left(\xi^6\int_0^1 |\what{\psi_s}|^2 r \, dr\right)^{\frac12} \\
&\quad + \left(\int_0^1 |\what{F^z}|^2 r\, dr\right)^{\frac12}\left(\xi^4 \int_0^1  \left|\frac{d}{dr}(r\what{\psi_s})\right|^2 \frac{1}{r} \, dr\right)^{\frac12}\\
\leq & \int_0^1 |\what{\BF^*}|^2 r \, dr.
\end{aligned}
\end{equation}
Therefore, combining \eqref{B-61}-\eqref{B-62} and \eqref{B-58} yields
\begin{equation} \label{B-63}
\xi^2\int_0^1 |\what{\psi_s}|^2 r \, dr \leq  C \left(\frac{\int_0^1 |\what{\BF^*}|^2 r \, dr}{\Phi |\xi|^3}\right)^{\frac23} \left(\frac{\int_0^1 |\what{\BF^*}|^2 r \, dr}{ |\xi|^2}\right)^{\frac13}.
\end{equation}
Hence one has
\begin{equation} \label{B-64}
\xi^4 \int_0^1 |\what{\psi_s}|^2 r \, dr
\leq  C (\Phi |\xi|)^{-\frac23}  \int_0^1 |\what{\BF^*}|^2 r \,  dr .
\end{equation}
Therefore, we have
\begin{equation} \label{B-65}
\xi^4 \int_0^1 |\what{\psi_s}|^2 rdr +
\int_0^1 |\mathcal{L}\what{\psi_s}|^2 r \, dr
\leq C (\Phi |\xi|)^{-\frac23} \int_0^1 |\what{\BF^*}|^2 r \, dr ,
\end{equation}
and by Cauchy-Schwarz inequality,
\begin{equation} \label{B-66} \ba
\xi^2\int_0^1  \left|\frac{d}{dr}(r\what{\psi_s})\right|^2 \frac1r  \, dr &  \leq \left(\int_0^1 |\mathcal{L}\what{\psi_s}|^2 r\, dr\right)^{\frac12}\left(\xi^4\int_0^1 |\what{\psi_s}|^2 r\, dr\right)^{\frac12} \\
&  \leq C (\Phi |\xi|)^{-\frac23} \int_0^1 |\what{\BF^*}|^2 r\, dr.
\ea \end{equation}
The two inequalities \eqref{B-65}--\eqref{B-66} give the estimate \eqref{case3-4}.

Moreover, according to \eqref{B-51},
\be \label{B66-1} \ba
& \Phi |\xi| \int_0^1 \frac{1-r^2}{r} \left| \frac{d}{dr} (r \what{\psi_s} )  \right|^2  \, dr + \Phi |\xi|^3 \int_0^1 (1 - r^2) |\what{\psi_s}|^2 r \, dr \\
\leq &  C  \left(\int_0^1 |\what{\BF^*}|^2 r \, dr\right)^{\frac12} \left(  \int_0^1 \left|\frac{d}{dr}(r \what{\psi}_s) \right|^2 \frac{1}{r} \, dr
+ \xi^2 \int_0^1 \left| \what{\psi}_s \right|^2 r \, dr\right)^{\frac12}.
\ea \ee
Herein, it follows from Lemma \ref{weightinequality} and Lemma \ref{lemma1} that one has
\be \label{B66-2}
\ba
& \int_0^1 \left|\frac{d}{dr}(r \what{\psi_s}) \right|^2  \frac{1}{r}\, dr
+ \xi^2 \int_0^1 \left| \what{\psi_s} \right|^2 r \, dr \\
\leq &  C \left[ \int_0^1 \frac{1-r^2}{r} \left| \frac{d}{dr} ( r \what{\psi_s} )  \right|^2 \, dr \right]^{\frac23} \left( \int_0^1 |\mathcal{L} \what{\psi_s} |^2 r \, dr \right)^{\frac13}\\
&\ \ \ \ + C \xi^2 \left( \int_0^1 (1-r^2) |\what{\psi_s}|^2 r \, dr \right)^{\frac23} \left[  \int_0^1  \left| \frac{d}{dr} ( r \what{\psi_s} )  \right|^2 \frac1r \, dr\right]^{\frac13}.
\ea
\ee
Hence, substituting \eqref{B66-2} into \eqref{B66-1} and  using Young's inequality and the estimates \eqref{B-65}--\eqref{B-66}  yield
\be \label{B66-3}
\Phi |\xi| \int_0^1 \frac{1-r^2}{r} \left| \frac{d}{dr} (r \what{\psi_s} )  \right|^2  \, dr + \Phi |\xi|^3 \int_0^1 (1 - r^2) |\what{\psi_s}|^2 r \, dr
\leq C (\Phi |\xi| )^{-\frac23} \int_0^1 |\what{\BF^*}|^2 r \, dr .
\ee
Consequently, combining \eqref{B66-2}--\eqref{B66-3} gives
\be \label{B66-4}
\int_0^1 \left|\frac{d}{dr}(r \what{\psi}_s) \right|^2  \frac{1}{r}\, dr
+ \xi^2 \int_0^1 \left| \what{\psi}_s \right|^2 r \, dr
\leq C (\Phi |\xi|)^{- \frac43 } \int_0^1 |\hat{\BF}|^2 r \, dr .
\ee
Therefore, it follows from  Lemma \ref{lemmaA2} that one has
\begin{equation} \ba
\left|\frac{d}{dr}(r\what{\psi_s})\right|(1) & \leq C \left( \int_0^1 |\mathcal{L} \what{\psi_s} |^2 r \, dr    \right)^{\frac{1}{4}} \left( \int_0^1 \left|  \frac{d}{dr}(r \what{\psi_s} ) \right|^2 \frac1r \, dr  \right)^{\frac14} \\
&\leq  C (\Phi |\xi|)^{-\frac12}  \left( \int_0^1 |\what{\BF^*}|^2 r \, dr \right)^{\frac12}.
\ea
\end{equation}
Finally, the inequality \eqref{B66-1} together with Lemma \ref{lemmaHLP} gives
\begin{equation} \label{B-67} \ba
 \int_0^1 |\what{\psi_s}|^2 r\, dr & \leq  C (\Phi |\xi|)^{-1}  \left(\int_0^1 |\what{\BF^*}|^2 r \, dr\right)^{\frac12} \left(  \int_0^1 \left|\frac{d}{dr}(r \what{\psi_s}) \right|^2 \frac{1}{r} \, dr
+ \xi^2 \int_0^1 \left| \what{\psi_s} \right|^2 r \, dr\right)^{\frac12} \\
& \leq C (\Phi |\xi| )^{-\frac53} \int_0^1 |\what{\BF^*}|^2 r \, dr.
\ea
\end{equation}

The existence of the solution $\what{\psi_s}$ follows from the similar arguments as in Section \ref{sec-ex}. Hence the proof of Lemma \ref{propslip} is completed.
\end{proof}

Now we give the proof of Proposition \ref{Bpropcase3}.

\begin{proof}[Proof of Proposition \ref{Bpropcase3}]
Let $\what{\psi_s}$ denote the solution to \eqref{slip}. Note that $\what{\psi_s}$ satisfies the same equation as $\hat{\psi}$, but with a slip boundary condition. We will recover the no slip boundary condition by the boundary layer analysis.
Define
\be \nonumber
\mcA = i \xi \bar{U}(r) - \mathcal{L} + \xi^2 , \ \ \ \ \mcH= \mathcal{L} - \xi^2,
\ee
and
\be \nonumber
\widetilde{\mcA} = i \frac{\xi \Phi}{\pi} 4 (1 - r) - \frac{d^2}{dr^2} + \xi^2,\ \ \ \ \widetilde{\mcH}= \frac{d^2}{dr^2} - \xi^2.
\ee
$\widetilde{\mathcal{A}}$ and $\widetilde{\mathcal{H}}$ can be regarded as the leading parts of $\mcA$ and $\mcH$, respectively, for the boundary layer function.

We look for a boundary layer $\what{\psi_{BL}}$, which is a solution to
\be \nonumber
\widetilde{\mcA}\widetilde{\mcH} \what{\psi_{BL}} = 0 .
\ee
First, consider the problem
\be \nonumber
\widetilde{\mcA} \phi = 0.
\ee
As discussed in \cite{M}, the operator $\widetilde{\mcA}$ can be written as  the Airy operator with  complex coefficients. Let Airy function $Ai(z)$ denote the solution to
\be \nonumber
\frac{d^2 Ai}{dz^2} -  z Ai =0, \ \ \ \mbox{in}\ \mathbb{C}.
\ee
Define
\be \label{5-73}
\widetilde{G}_{\xi, \Phi} ( \rho) = \left\{  \ba  & Ai \left( C_{+} (\rho + \frac{\pi |\beta | \xi}{4 i \Phi }) \right), \ \ \ \mbox{when}\ \xi>0 , \\
 & Ai \left(  C_{-}  (\rho + \frac{\pi |\beta | \xi}{4 i \Phi })   \right),\ \ \ \ \mbox{when}\ \xi < 0 , \ea
\right.
\ee
where $ |\beta | = \left( \frac{ 4 \xi \Phi}{\pi}  \right)^{\frac13}$,  $C_{+} = e^{ i \frac{\pi}{6}}$ and $C_{-} = e^{-i \frac{\pi}{6}}$.
Define
\be \label{5-74}
\widetilde{\widetilde{G}}_{\xi, \Phi}(r) = \widetilde{G}_{\xi, \Phi} (|\beta| ( 1  - r) ).
\ee
It is straightforward to check that
\be \nonumber
\widetilde{\mcA} \widetilde{\widetilde{G}}_{\xi, \Phi} = 0.
\ee

Without loss of generality, we assume that $\xi > 0$ from now on.
Next, define
\be \label{5-75}
G_{\xi, \Phi} (\rho) = \int_{\rho}^{+\infty} e^{- \frac{|\xi|}{|\beta|} ( \rho - \tau) } \int_{\tau}^{+ \infty}
e^{- \frac{|\xi|}{|\beta|} (\sigma - \tau)} \widetilde{G}_{\xi, \Phi} (\sigma) \, d\sigma d\tau ,
\ee
which satisfies
\be \label{5-76}
\frac{d^2 G_{\xi, \Phi} }{d \rho^2} - \frac{|\xi|^2 }{|\beta|^2 } G_{\xi, \Phi} = \tilde{G}(\rho).
\ee
Define
\be \label{5-77}
C_{0, \xi, \Phi} = \left\{  \ba & \frac{1}{G_{\xi, \Phi} (0) }, \ \ \ \mbox{if}\ |G_{\xi, \Phi} (0) | \geq 1,   \\ &
1,\ \ \ \ \ \ \ \ \ \ \ \ \mbox{otherwise}.   \ea  \right.
\ee
Set
\be \label{5-78}
\what{\psi_{BL}} (r) : = C_{0, \xi, \Phi} G_{\xi, \Phi} (|\beta| ( 1  - r) ).
\ee
It can be verified that
\be \label{5-79}
\widetilde{\mcA}\widetilde{\mcH} \what{\psi_{BL}} = 0, \ \ \ 0 < r <1 ,
\ee
and $|\what{\psi_{BL}}(1)| \leq 1. $

Let us insert a lemma which gives the estimates of $C_{0, \xi, \Phi}$ and $\frac{d \what{\psi_{BL}} }{dr}|_{r=1}$.
\begin{lemma}\label{bound} $(1)$\ It holds that
\be \nonumber
\tilde{C_0} : = \inf \left\{  | C_{0, \xi, \Phi} | :\ \Phi\geq 1, \ \frac{1}{ \Phi} \leq |\xi| \leq  \sqrt{\Phi}    \right\} >0 ,
\ee
and the function $G_{\xi, \Phi}$ defined in \eqref{5-75} satisfies for $\varsigma$ large enough
\be \label{5-80}
\sup_{\Phi \geq 1 } \sup_{ \frac{1}{ \Phi} \leq |\xi| \leq  \sqrt{\Phi} } \sup_{\rho \geq \varsigma}  e^{\rho}
\left| \frac{d^k G_{\xi, \Phi}}{d \rho^k} (\rho) \right| < \infty, \ \ \ k =0, 1, 2, 3.
\ee

$(2)$\  There is a constant $\epsilon \in (0, 1)$ such that
\be \nonumber
\Sigma_{\epsilon} : =  \{  \mu \in \mathbb{C}\, | \, arg \mu = -\frac{\pi}{6} ,\  0\leq |\mu| \leq \epsilon \},
\ee
then
\be \nonumber
K_{\epsilon} : = \inf_{\mu \in \Sigma_{\epsilon}} \left| \int_0^{+ \infty} e^{- \mu s } Ai (s + \mu^2) \, ds    \right| \geq \frac16.
\ee

\end{lemma}
The proof of Lemma \ref{bound} is exactly the same as that of  \cite[Lemma 3.7]{M}, so we omit the details here.

Now we are ready to construct the remainder term $\what{\psi_e}$, which is the solution to the following problem
\be \label{5-81}
\left\{  \ba
&  i \xi \bar{U}(r) ( \mathcal{L} -\xi^2) \what{\psi_e} - (\mathcal{L} - \xi^2)^2 \what{\psi_e} = \widetilde{\mcA}\widetilde{\mcH} ((1 - \chi) \what{\psi_{BL}} ) + (\widetilde{\mcA}\widetilde{\mcH} - \mcA\mcH) ( \chi \what{\psi_{BL}} ) ,\\
&  \what{\psi_e} (0) = \what{\psi_e} (1) = \mathcal{L} \what{\psi_e} (0) = \mathcal{L} \what{\psi_e} (1) = 0,
\ea
\right.
\ee
where $\chi$ is a smooth cut-off function, satisfying
\be \nonumber
\chi(r) = \left\{ \ba & 1,\ \ \ \ \ \ r\geq \frac12,\\ & 0,\ \ \ \ \ \ r\leq \frac14    \ea  \right.
\ee
The $L^2_r$-norm of the right hand of \eqref{5-81} can be estimated as follows.
\be \nonumber
\widetilde{\mcA}\widetilde{\mcH} ( (1  -\chi ) \what{\psi_{BL}} )
= \left( i \frac{\xi \Phi }{\pi } 4 (1 - r) - \frac{d^2}{dr^2}  + \xi^2   \right) \left( \frac{d^2}{dr^2} - \xi^2   \right) ( (1 - \chi ) \what{\psi_{BL} })
\ee
According to Lemma \ref{bound}, when $\frac12 |\beta| \geq \varsigma$, one has
\be \label{5-82-1} \ba
& \int_0^1 \left| i \frac{\xi \Phi}{\pi} 4(1 -r)  \frac{d^2}{dr^2} ( (1 - \chi ) \what{\psi_{BL} })          \right|^2  r \, dr  \\
\leq & C \Phi^2 \xi^2 \int_0^{\frac12}  \left( \left| \frac{d^2}{dr^2} \what{\psi_{BL}}         \right|^2  + \left| \frac{d}{dr}  \what{\psi_{BL}} \right|^2 + | \widehat{\psi_{BL}} |^2 \right) \, dr \\
\leq & C \Phi^2 \xi^2 \int_{\frac{ |\beta| }{2}}^{|\beta|} \left(  |\beta|^3 \left|  \frac{d^2 G_{\xi, \Phi} }{d\rho^2}  \right|^2 + |\beta| \left| \frac{d G_{\xi, \Phi} }{d\rho}   \right|^2 + |\beta|^{-1} | G_{\xi, \Phi}  |^2 \right) \, d\rho \\
\leq & C \Phi^2 \xi^2 \int_{\frac{ |\beta| }{2}}^{|\beta|} \left( |\beta|^3  + |\beta| + |\beta|^{-1} \right) e^{-2\rho} \, d\rho \\
\leq & C e^{ - |\beta| } ( |\beta|^{10} + |\beta|^8 + |\beta|^6)  \leq C .
\ea
\ee
Similarly, one has
\be \label{5-82-2}
\int_0^1 \left|  \widetilde{\mcA}\widetilde{\mcH} ( (1  -\chi ) \what{\psi_{BL}} ) \right|^2 r \, dr
\leq C .
\ee

On the other hand, the estimates
\be \label{5-83} \ba
&( \widetilde{\mcA}\widetilde{\mcH} - \mcA\mcH ) (\chi \what{\psi_{BL}} )   =
(\widetilde{\mcA} - \mcA) \mcH ( \chi \what{\psi_{BL}} )  + \widetilde{\mcA} ( \widetilde{\mcH} - \mcH) ( \chi \what{\psi_{BL}} ) \\
 = &\left( i \frac{\xi \Phi}{\pi}2(1 - r)^2 + \frac1r \frac{d}{dr} - \frac{1}{r^2}  \right)
\left( \frac{d^2}{dr^2} + \frac1r \frac{d}{dr} - \frac{1}{r^2} - \xi^2 \right) ( \chi \what{\psi_{BL}} ) \\
&\ \  \ + \left(  i \frac{\xi \Phi}{\pi } 4 (1 - r) - \frac{d^2}{dr^2} + \xi^2     \right) \left( - \frac1r \frac{d}{dr} +
\frac{1}{r^2}  \right) ( \chi \what{\psi_{BL}} ).
\ea
\ee
and
\be \label{5-83-1}
\ba
& \int_0^1  \left| i \frac{ \xi \Phi }{\pi } 2(1 - r)^2 \frac{d^2}{dr^2}  ( \chi \what{\psi_{BL}} )        \right|^2 r \, dr \\
\leq & C \Phi^2 \xi^2 \int_0^1 (1 -r)^4  \left[ \left|  \chi^{\prime \prime} (r ) \what{\psi_{BL}} \right|^2  +
\left| \chi^{\prime}(r) \frac{d}{dr} \what{\psi_{BL}}   \right|^2 + \left| \chi(r) \frac{d^2}{dr^2}\what{\psi_{BL}}  \right|^2 \right] \, dr \\
\leq & C \Phi^2 \xi^2 \int_{\frac12 |\beta| }^{\frac34 |\beta| } |\beta|^{-4} \rho^4 |G_{\xi, \Phi}(\rho) |^2 |\beta|^{-1} \, d \rho
+ C \Phi^2 \xi^2 \int_{\frac12 |\beta| }^{\frac34 |\beta| } |\beta|^{-4} \rho^4 |\beta|^2 |G^{\prime}_{\xi, \Phi} (\rho) |^2 |\beta|^{-1} \, d\rho\\
&\ \ \ + C \Phi^2 \xi^2 \int_0^{\frac34 |\beta| } |\beta|^{-4} \rho^4 |\beta|^4 |G_{\xi, \Phi}^{\prime \prime} (\rho)|^2 |\beta|^{-1} \, d\rho \\
\leq & C( |\beta|+ |\beta|^3)   \int_{\frac12 |\beta| }^{\frac34 |\beta| } \rho^4 e^{-2 \rho} \, d\rho + C |\beta|^5
\left(  \int_0^{\varsigma} \rho^4 \, d\rho + \int_{\varsigma}^{\frac34 |\beta|} \rho^4 e^{-2\rho} \, d\rho \right)  \\
\leq & C (1+ |\beta|^5)
\ea
\ee
hold. Similarly, one has
\be \label{5-83-2}
\int_0^1 \left|  ( \widetilde{\mcA}\widetilde{\mcH} - \mcA\mcH ) (\chi \what{\psi_{BL}} ) \right|^2 r \, dr \leq C |\beta|^5.
\ee
Hence, it follows from Lemma \ref{propslip} that one has
\be \label{5-84}
\int_0^1 |\what{\psi_e}|^2 r \, dr
\leq C |\beta|^5 ( \Phi |\xi|)^{-2} \leq C ( \Phi |\xi| )^{-\frac13},
\ee
\be \label{5-85}
\int_0^1 | \mathcal{L} \what{\psi_e}  |^2 r\, dr + \xi^2 \int_0^1 \left| \frac{d}{dr} ( r \what{\psi_e}) \right|^2 \frac1r \, dr + \xi^4 \int_0^1 |\what{\psi_e}|^2 r\, dr
\leq C |\beta|^5 (\Phi |\xi| )^{-\frac67} \leq C ( \Phi |\xi| )^{\frac{17}{21}}
\ee
\be \label{5-86}\ba
& \int_0^1 \left| \frac{d}{dr} ( r \mathcal{L} \what{\psi_e} ) \right|^2 \frac1r \, dr + \xi^2 \int_0^1 |\mathcal{L} \what{\psi_e} |^2 r \, dr + \xi^4 \int_0^1 \left| \frac{d}{dr} ( r \what{\psi_e}) \right|^2 \frac1r \, dr + \xi^6 \int_0^1 |\what{\psi_e}|^2 r\, dr \\
\leq & C (\Phi |\xi| )^{\frac53}.
\ea
\ee

Finally, suppose the stream function  $\hat{\psi}$  is of the form \eqref{case3-1} with the constants $a$ and $b$ to be determined, where $I_1(z)$ is the modified Bessel function of the first kind satisfying \eqref{eqBessel1}. It is straightforward to show that
\[
\mcH I_1(|\xi|r)=0.
\]
Let us determine $a$ and $b$ now. On the boundary $r = 1$, it holds that
\be \nonumber
\hat{\psi}(1) = \frac{d}{dr} (  r \hat{\psi} ) (1) = \frac{d \hat{\psi} }{d r}(1)  = 0.
\ee
Hence we have
\be \label{5-87} \left\{ \ba
& a I_1 (|\xi|) + b \what{\psi_{BL}} (1) = 0, \\
& a |\xi| I_1^{\prime} (|\xi| )  + b \frac{d}{dr} \what{\psi_{BL}} (1) + b \frac{d}{dr} \what{\psi_e} (1) = - \frac{d}{dr} \what{\psi_s} (1) .    \ea
\right.
\ee
Solving the linear system \eqref{5-87} yields
\be \label{5-88}
b= \frac{\frac{d}{dr} \what{\psi_s} (1) I_1 (|\xi|) }{ \what{\psi_{BL}}(1)|\xi| I_1^{\prime} (|\xi|) -
\left( \frac{d}{dr} \what{\psi_{BL}} (1) + \frac{d}{dr} \what{\psi_e} (1)  \right) I_1 (|\xi|) } \ \ \text{and}\ \ a = - \frac{ b \what{\psi_{BL}}(1)}{I_1 (|\xi|) }.
\ee

The quantity $\frac{d  }{dr}\what{\psi_{BL}} (1) $ is computed as follows, when $\xi>0$,
\be \label{5-89}
 \frac{d }{dr } \what{\psi_{BL}} (1) = C_{0, \xi, \Phi} |\beta| \left[ C_{-} \int_0^{+\infty} e^{-\lambda s} Ai (s + \lambda^2) \, ds + \frac{|\xi|}{|\beta|} G_{\xi, \Phi} (0)  \right],
\ee
where $C_{-} = e^{-  i \frac{\pi}{6}}$ and $\lambda = \frac{|\xi|}{|\beta|} C_{-}$.
Note that
\be \nonumber
|\lambda | = \frac{|\xi|}{|\beta|}  = \left(\frac{\pi}{4} \right)^{\frac13} |\xi|^{\frac23} \Phi^{-\frac13} \leq \left(\frac{\pi}{4} \right)^{\frac13} (\epsilon_1)^{\frac23}.
\ee
When $\epsilon_1$ satisfies
\be \nonumber
\left(\frac{\pi}{4}\right)^{\frac13} \epsilon_1^{\frac23} \leq \min \left\{  \epsilon, \ \frac{1}{12} \tilde{C_0}  \right \} ,
\ee
 according to Lemma \ref{bound}, one has
\be \label{5-90}
\left| \frac{d}{dr} \what{\psi_{BL}} (1) \right| \geq \tilde{C}_0 |\beta| \left(  \frac16 - \frac{|\xi|}{|\beta|} \frac{1}{\tilde{C}_0}\right) \geq \frac{1}{12} \tilde{C}_0 |\beta| = \kappa |\beta|.
\ee
When $\xi <0$, similar computations verify that \eqref{5-90} also holds.

On the other hand, it follows from Lemma \ref{lemmaA2} that one has
\be \label{5-91} \ba
 &\left| \frac{d}{dr} \what{\psi_e}(1)  \right|^2  =  \left| \frac{d}{dr} ( r \what{\psi_e} ) (1)  \right|^2  \\
\leq &  C  \left( \int_0^1 |\mathcal{L} \what{\psi_e} |^2  r \, dr \right)^{\frac34} \left( \int_0^1 \left| \what{\psi_e}\right|^2 r \, dr  \right)^{\frac14}
\leq  C(\Phi |\xi| )^{\frac{11}{21}}.
\ea
\ee
Moreover, using Lemma \ref{lemBessel} gives
\be \nonumber
\frac{|\xi| I_1^{\prime}(|\xi|) }{I_1(|\xi|) }  \leq 1 + |\xi|.
\ee
Hence one has
\be \nonumber \ba
 \left| \left[ \frac{d}{dr}\what{\psi_{BL}} (1)  + \frac{d}{dr} \what{\psi_e}(1)  \right] - \xi \frac{I_1^{\prime}(|\xi|) }{I_1 (|\xi|) } \what{\psi_{BL}} (1)  \right|  \geq & \kappa |\beta| - C |\beta|^{\frac{11}{14}} - (1 + |\xi|) \\
\geq & \kappa |\beta| - C |\beta|^{\frac{11}{14}} - 1 - \left(\frac{\pi}{4}\right)^{\frac13} \epsilon_1^{\frac23} |\beta|.
\ea
\ee
When $\left(\frac{\pi}{4}\right)^{\frac13} \epsilon_1^{\frac23} \leq \frac{\kappa}{4} $ and  $|\beta|$ is large enough,
\be \label{5-92}
\left| \left[ \frac{d}{dr}\what{\psi_{BL}} (1)  + \frac{d}{dr} \what{\psi_e}(1)  \right] - |\xi| \frac{I_1^{\prime}(|\xi|) }{I_1 (|\xi|) } \what{\psi_{BL}} (1)  \right|  \geq \frac{\kappa}{2} |\beta|.
\ee

Therefore, taking \eqref{5-92} into \eqref{5-88}, by Lemma \ref{propslip}, one has
\be \label{5-93}
 \ba
|b|    \leq & C |\beta|^{-1} \left| \frac{d}{dr} \what{\psi_s} (1)  \right|
\leq C |\beta|^{-1} (\Phi |\xi|)^{-\frac12}  \left( \int_0^1 |\what{\BF^*}|^2 r \, dr \right)^{\frac12} \\
\leq  & C ( \Phi |\xi|)^{-\frac{5}{6} } \left( \int_0^1 |\what{\BF^*}|^2 r \, dr  \right)^{\frac12},
\ea \ee
and
\be \label{5-94}
|a| \leq C ( \Phi |\xi|)^{-\frac{5}{6} } |I_1 (|\xi| )|^{-1} \left( \int_0^1 |\what{\BF^*}|^2 r \, dr  \right)^{\frac12}.
\ee
Hence, according to Lemma \ref{bound}, one has
\be \label{B-95-1} \ba
& b^2 \int_0^1 \left| \chi \what{\psi_{BL}}  \right|^2 r \, dr
\leq Cb^2 \int_{\frac14}^1 \left|   \chi \what{\psi_{BL}} \right|^2 \, dr \\
\leq & Cb^2 \int_{0}^{\frac34 |\beta| } |G_{\xi, \Phi} ( \rho ) |^2 |\beta|^{-1} \, d\rho \\
\leq & Cb^2 \left[  \int_{0}^{\varsigma } |G_{\xi, \Phi} ( \rho ) |^2 |\beta|^{-1} \, d\rho  + \int_{\varsigma}^{\frac34 |\beta| } |G_{\xi, \Phi} ( \rho ) |^2 |\beta|^{-1} \, d\rho          \right] \\
\leq & Cb^2 \left(  |\beta|^{-1} + \int_{\varsigma}^{\frac34 \beta} e^{-2\rho} |\beta|^{-1} \, d\rho     \right) \leq C (\Phi |\xi|)^{-2}\int_0^1 |\what{\BF^*}|^2 r \, dr ,
\ea
\ee
where $\varsigma$ is the one appeared in \eqref{5-80} in  Lemma \ref{bound}.
Similarly, one has
\be \label{B-95-2}
\ba
& b^2 \int_0^1 \left|\mathcal{L} (\chi \what{\psi_{BL}} ) \right|^2 r \, dr \\
\leq &  Cb^2 \int_0^1 \left( \left| \frac{d^2}{dr^2} ( \chi \what{\psi_{BL}}) \right|^2 + \frac{1}{r^2} \left| \frac{d}{dr}( \chi \what{\psi_{BL}})   \right|^2 + \frac{1}{r^4} \left| \chi \what{\psi{BL}} \right|^2  \right)^2 r \, dr\\
\leq & C b^2 \int_{\frac14}^{\frac12} |\chi^{\prime \prime}|^2 |\what{\psi_{BL}}|^2  \, dr
+ Cb^2 \int_{\frac14}^{\frac12} |\chi^{\prime}|^2 \left|  \frac{d}{dr} \what{\psi_{BL}}  \right|^2 \, dr
+ Cb^2 \int_{\frac14}^{1} \chi^2 |\what{\psi_{BL}}|^2 \, dr \\
& \ + C b^2 \int_{\frac14}^{\frac12} |\chi^{\prime}|^2 |\what{\psi_{BL}}|^2 \, dr
+ C b^2 \int_{\frac14}^1 \chi ^2 \left| \frac{d}{dr} \what{\psi_{BL}}  \right|^2 \, dr + C b^2 \int_{\frac14}^1 \chi^2 |\what{\psi_{BL}}|^2 \, dr \\
\leq & Cb^2 \int_{\frac14}^{1} \left( |\what{\psi_{BL}}|^2 + \left| \frac{d}{dr} \what{\psi_{BL}}  \right|^2 + \left| \frac{d^2}{dr^2} \what{\psi_{BL}}  \right|^2 \right) \, dr \\
\leq & Cb^2 ( |\beta|^{-1} + |\beta| + |\beta|^3) \leq C (\Phi |\xi|)^{-\frac23} \int_0^1 |\what{\BF^*}|^2 r \, dr ,
\ea
\ee
\be \label{B-95-3}
b^2  \xi^2 \int_0^1 \left| \frac{d}{dr} ( r \chi \what{\psi_{BL}} )\right|^2 \frac1r \, dr + b^2 \xi^4 \int_0^1 \left| \chi \what{\psi_{BL}} \right|^2 r \, dr
\leq C (\Phi |\xi|)^{-\frac23} \int_0^1 | \what{\BF^*}|^2 r \, dr ,
\ee
and
\be \label{B-95-4}\ba
& b^2 \int_0^1 \left|  \frac{d}{dr} ( r \mathcal{L} (\chi \what{\psi_{BL}} ) \right|^2 \frac1r \, dr + b^2 \xi^2 \int_0^1 \left|\mathcal{L} (\chi \what{\psi_{BL}} ) \right|^2 r \, dr +   b^2  \xi^4 \int_0^1 \left| \frac{d}{dr} ( r \chi \what{\psi_{BL}} )\right|^2 \frac1r \, dr  \\ & \ \ \ \  + b^2 \xi^6 \int_0^1 \left| \chi \what{\psi_{BL}} \right|^2 r \, dr \leq C \int_0^1 |\what{\BF^*}|^2 r \, dr .
\ea
\ee

Meanwhile, according to Lemma \ref{AlemBessel2}, it holds that
\be \label{B-96-1}
\ba
a^2 \int_0^1 | I_1(|\xi|r)|^2 \, dr
\leq &  C \min\left\{ 1,  |\xi|^{-1} \right\} (\Phi |\xi|)^{-\frac53}    \int_0^1 |\what{\BF^*}|^2 r \, dr \\
\leq &  C (\Phi |\xi|)^{-\frac53}   \int_0^1 |\what{\BF^*}|^2 r \, dr
\ea
\ee
and
\be \label{B-96-2}
\ba
&|a|^2\left( \xi^2 \int_0^1 | I_1(|\xi|r)|^2 \, dr +
\int_0^1 \left| \frac{d}{dr} (r I_1 (\xi r) ) \right|^2 \frac1r \, dr \right)\\
\leq & C (\min\{1, |\xi|^{-1}\} \xi^2 + \max\{1, |\xi|\} )  (\Phi |\xi|)^{- \frac{5}{3}} \int_0^1|\what{\BF^*}|^2 r \, dr \\
\leq & C  (\Phi |\xi|)^{- \frac{4}{3}} \int_0^1|\what{\BF^*}|^2 r \, dr.
\ea
\ee
Furthermore, one has
\be \label{B-96-3}
\ba
& a^2 \left(\int_0^1 |\mathcal{L} I_1(|\xi|r)|^2 r \, dr + \xi^4 \int_0^1 |I_1(|\xi|r)|^2 r\, dr \right)\\
\leq & C \min\{1, |\xi|^{-1}\} |\xi|^4 (\Phi |\xi| )^{-\frac{5}{3} } \int_0^1 |\hat{f}|^2 r \, dr \leq  C (\Phi |\xi| )^{-\frac{2}{3} } \int_0^1 |\what{\BF^*}|^2 r \, dr
\ea
\ee
and
\be \label{B-96-4} \ba
& a^2 \int_0^1 \left| \frac{d}{dr} ( r \mathcal{L} I_1(|\xi|r))  \right|^2 \frac1r \, dr + a^2 \xi^2 \int_0^1 |\mathcal{L} I_1(|\xi|r)|^2 r \, dr \\
&\quad + a^2 \xi^4 \int_0^1 \left| \frac{d}{dr} (r I_1 (|\xi| r)) \right|^2 \frac1r \, dr   + a^2  \xi^6 \int_0^1 |I_1 (|\xi| r)|^2 r \, dr\\
 \leq &  C \left( \min\{1, |\xi|^{-1}\} |\xi|^6 + \max\{1, |\xi|\}|\xi|^4 \right) (\Phi |\xi| )^{-\frac{5}{3} } \int_0^1 |\what{\BF^*}|^2 r \, dr \\
 \leq &  C   \int_0^1 |\what{\BF^*}|^2 r \, dr.
\ea
\ee
 Combining the estimates in Lemma \ref{propslip},  \eqref{B-95-1}--\eqref{B-96-4}, one proves \eqref{case3-15}.     This finishes the proof of Proposition \ref{Bpropcase3}.
\end{proof}

Let
\be \nonumber
\chi_3 (\xi ) = \left\{ \ba & 1 , \ \ \  \frac{1}{\epsilon_1 \Phi}  \leq |\xi| \leq \epsilon_1 \sqrt{\Phi} , \\
& 0,\ \ \ \ \text{otherwise},  \ea  \right.
\ee
and $\psi_{med}$ be the function  such that
$
\what{\psi_{med}}  =  \chi_3 (\xi) \hat{\psi}.
$
Define
\be \nonumber
v^r_{med} = \partial_z \psi_{med},\ \ \ \ v^z_{med} = - \frac{\partial_r ( r \psi_{med} )}{r}, \ \ \ \ \text{and}\ \ \ \ \Bv^*_{med} = v^r_{med}\Be_r + v^z_{med}\Be_z.
\ee
And similarly, we define $F^r_{med}, F^z_{med}, \BF^*_{med}$, $\Bo^\theta_{med}$.
\begin{pro} \label{Bpropcase3-1}
The solution $\Bv^*$  satisfies
\be \label{B-99}
\|\Bv^*_{med} \|_{H^2(\Omega)} \leq C  \|\BF^*_{med} \|_{L^2(\Omega)},
\ee
where $C$ is a uniform constant independent of $\Phi$ and $\BF$.
\end{pro}

\begin{proof}
It follows from the definition of $\Bv^*_{med}$ and Proposition \ref{Bpropcase3} that
\be \label{5-102}
\|\Bv^*_{med} \|_{L^2(\Omega)}^2 \leq C \int_{ \frac{1}{\epsilon_1 \Phi} \leq |\xi| \leq \epsilon_1 \sqrt{\Phi}   } \int_0^1 \left( \xi^2 |\hat{\psi}|^2 r +
\left| \frac{\partial}{\partial r}( r \hat{\psi} ) \right|^2 \frac1r   \right) \, dr
\leq C \|\BF^*_{med} \|_{L^2(\Omega)}^2 .
\ee

On the other hand, the straightforward computations give
\be \label{5-103}
\Delta \Bv^*_{med} = -{\rm curl}~\Bo^\theta_{med} = \partial_z \omega^\theta_{med} \Be_r - \left( \partial_r \omega^\theta_{med} + \frac{\omega^\theta_{med}}{r} \right)\Be_z.
\ee
Thus according to Proposition \ref{Bpropcase3}, one has
\be \label{5-104}
\|\partial_z \omega^\theta_{med} \|_{L^2(\Omega)}^2 \leq
C \int_{\frac{1}{\epsilon_1 \Phi} \leq |\xi| \leq \epsilon_1 \sqrt{\Phi} } \int_0^1 \xi^2 | ( \mathcal{L} - \xi^2) \hat{\psi}|^2 r \, dr d\xi
\leq C\|\BF^*_{med} \|_{L^2(\Omega)}^2 ,
\ee
and
\be \label{5-105} \ba
\left\| \partial_r \omega^\theta_{med} + \frac{\omega^\theta_{med}}{r}   \right\|_{L^2(\Omega)}^2
& \leq C \int_{\frac{1}{\epsilon_1 \Phi} \leq |\xi| \leq \epsilon_1 \sqrt{\Phi} } \int_0^1 \left| \frac{\partial}{\partial r} [ r ( \mathcal{L} -\xi^2) \hat{\psi} ]   \right|^2
\frac1r \, dr d\xi  \\
& \leq C\|\BF^*_{med} \|_{L^2(\Omega)}^2 .
\ea
\ee

Applying the regularity theory for the elliptic equation \eqref{5-103} with homogeneous Dirichlet boundary condition for $\Bv^*_{med}$ yields
\eqref{B-99}. This finishes the proof of Proposition \ref{Bpropcase3-1}.
\end{proof}

Combining the existence result in Proposition \ref{back}  and the estimates in Propositions \ref{smallflux} \ref{Bpropcase1-1}, \ref{Bpropcase2-1}, \ref{Bpropcase3-1}, we finish the proof of Theorem \ref{thm1}.


\section{Nonlinear Structural Stability}\label{secnonlinear}
In this section, we use a fixed point theorem to prove the existence of solutions for the nonlinear problem \eqref{SNS}-\eqref{fluxBC}, which gives the uniform nonlinear structural stability of Hagen-Poiseuille flows.

 Let $\Bv$ denote the perturbed  velocity. Then $\Bv$ satisfies the  system
 \be  \label{perturb1}
\left\{ \ba
&\bBU \cdot \nabla \Bv + \Bv \cdot \nabla \bBU + (\Bv \cdot \nabla )\Bv - \Delta \Bv + \nabla P = \BF, \ \ \ \mbox{in}\ \Omega, \\
& {\rm div}~\Bv = 0,\ \ \ \mbox{in}\ \Omega, \\
\ea
\right.
\ee
supplemented with the following conditions
\be \label{perturb2} \Bv = 0\ \ \ \mbox{on}\ \partial \Omega,
\ \ \  \int_{\Sigma} \Bv \cdot \Bn \, dS = 0.
\ee

Now we apply Lemma \ref{nonlinear} to show the existence of solutions for the problem \eqref{perturb1}--\eqref{perturb2}. Set
$Y = H^{\frac53} (\Omega)$. For any given $\BF \in L^2(\Omega)$, as proved in Theorem \ref{thm1}, there exists  a strong solution $\Bv\in H^2(\Omega)$ to the linear problem \eqref{2-0-1} and  \eqref{BC}. We denote this solution by $\mathcal{T} \BF$.
According to Theorem \ref{thm1}, one has
\be \label{perturb3}
\|\mathcal{T} \BF\|_{Y} \leq C \| \BF\|_{L^2(\Omega)}.
\ee
For any $\Bv_1, \Bv_2 \in Y$, the bilinear form
\be
\mathcal{B} ( \Bv_1, \Bv_2) = - \mathcal{T} ( (\Bv_1 \cdot \nabla) \Bv_2)),
\ee
is well-defined and $\mathcal{B} (\Bv_1, \Bv_2) \in Y$, satisfying
\be
\|\mathcal{B}( \Bv_1, \Bv_2) \|_{Y} \leq C\|\Bv_1 \cdot \nabla) \Bv_2 \|_{L^2(\Omega)} \leq C \|\Bv_1\|_{L^{12}(\Omega)} \|\nabla \Bv_2\|_{L^{\frac{12}{5}}(\Omega)}
\leq C \|\Bv_1\|_{Y} \|\Bv_2\|_{Y},
\ee
where the last inequality comes from the Sobolev embedding for functions in three dimensional domains.
Hence, it follows from Lemma \ref{nonlinear} that if $\epsilon_0$ is suitably small and
\be \nonumber
\|\BF\|_{L^2(\Omega)} \leq \epsilon_0,
\ee
 then
the problem
\be
\Bv = \mathcal{T} \BF + \mathcal{B}(\Bv, \Bv)
\ee
has a unique solution $\Bv$ satisfying $\|\Bv\|_{Y} \leq C \epsilon_0$.

Furthermore, in fact the solution $\Bv$ satisfies that
\be
\|\Bv\|_{H^2(\Omega)}
\leq C (1 + \Phi^{\frac14} ) \|\BF\|_{L^2(\Omega)} +
C (1 + \Phi^{\frac14} ) \|\Bv\|_{Y}^2
\leq  C (1 + \Phi^{\frac14} ) \|\BF\|_{L^2(\Omega)},
\ee
and $\Bv$ is a strong solution to the problem \eqref{perturb1}--\eqref{perturb2}. Hence, the proof of Theorem \ref{mainthm} is completed.




\appendix
\section{Some elementary lemmas}
In this appendix, we collect some basic lemmas which play important roles in the paper and might be useful elsewhere.
We first give some Poincar\'e type inequalities.
\begin{lemma}\label{lemma1}
For a  function $g\in C^2([0,1])$
it holds that
\be \label{2-1-11}
\int_0^1 |g |^2  r \, dr \leq  \int_0^1 \left|   \frac{d}{dr} (r g )     \right|^2 \frac1r \, dr.
\ee
If, in addition, $g(0)=g(1)=0$, then one has
\be
  \int_0^1 \left|   \frac{d}{dr} (r g )     \right|^2 \frac1r \, dr
  \leq \left( \int_0^1 | \mathcal{L}g|^2  r \, dr  \right)^{\frac12} \left( \int_0^1 |g|^2 r\, dr  \right)^{\frac12}
\leq  \int_0^1 | \mathcal{L}g|^2  r \, dr.
\ee
\end{lemma}
\begin{proof}
For every $r\in [0, 1]$, one has
\be \nonumber
|r g (r) | = \left| \int_0^r  \frac{d}{ds} (s g(s) ) \, ds         \right|
\leq   r  \left( \int_0^1   \left|  \frac{d}{dr} (r g)                \right|^2 \frac1r \, dr
\right)^{\frac12} .
\ee
Therefore,
\be \la{2-1-14}
\sup_{r\in [0, 1]} |g (r)|^2  \leq \int_0^1 \left|  \frac{d}{dr} ( r g)    \right|^2 \frac1r \, dr ,
\ee
which implies \eqref{2-1-11}.

If $g(0)=g(1)=0$, then
integration by parts and using the homogeneous boundary conditions for $g$ give
\be \la{2-1-12} \ba
\int_0^1 \left| \frac{d}{dr } ( rg )     \right|^2  \frac1r \, dr
&= \int_0^1 \mathcal{L} g  \overline{g } r \, dr
 \leq \left( \int_0^1 | \mathcal{L} g |^2  r \, dr \right)^{\frac12}
 \left( \int_0^1 |g |^2  r \, dr   \right)^{\frac12}.
 \ea \ee
Hence, it follows from \eqref{2-1-11}  that one has
\be \nonumber
\int_0^1 |g|^2 r \, dr \leq \int_0^1 \left| \frac{d}{dr} (r g)    \right|^2   \frac{1}{r} \, dr \leq \int_0^1 | \mathcal{L} g |^2 r \, dr.
\ee
This finishes the proof of the lemma.
\end{proof}

In the following lemma, we give the pointwise estimate for the functions evaluated on the boundary.
\begin{lemma}\label{lemmaA2}
For a function $g\in C^3([0,1])$, one has
\begin{equation}\label{estLinfty}
\begin{aligned}
\left|\frac{d}{dr}(rg)(1)\right|\leq  & 2  \left( \int_0^{1}\left|\frac{d}{dr}(rg)\right|^2 \frac{1}{r}\, dr \right)^{\frac14} \left( \int_0^{1}\left|\mathcal{L}g\right|^2r \, dr \right)^{\frac14} + 4 \left( \int_0^1 \left| \frac{d}{dr}(r g) \right|^2 \frac1r \, dr \right)^{\frac12},
\end{aligned}
\end{equation}
and
\begin{equation}\label{3-3-1-16}
| \mathcal{L} g (1) |  \leq 2 \left( \int_{\frac12}^1 |\mathcal{L } g |^2  r \, dr \right)^{\frac12}
+ 2 \left(   \int_{\frac12}^1 \left| \frac{d}{dr} ( r \mathcal{L} g)   \right|^2  \frac{1}{r} \, dr \right)^{\frac14}  \left( \int_{\frac12}^1 \left| \mathcal{L} g \right|^2 r \, dr       \right)^{\frac14}.
\end{equation}
Furthermore, if $g\in C^4([0,1])$, with $\mathcal{L} g (0) = 0$ and $\frac{d}{dr}( r \mathcal{L}g) \in C[0, 1]$, then one has
\begin{equation}\label{3-3-1-20}
\int_0^1 \left| \frac{d}{dr} ( r \mathcal{L} g )   \right|^2 \frac{1}{r} \, dr  \leq C \left(\int_0^1 |\mathcal{L} g|^2 r \, dr \right)^{\frac12} \left( \int_0^1 |\mathcal{L}^2 g |^2 r \, dr   \right)^{\frac12} + C \int_0^1 |\mathcal{L}g|^2 r\, dr ,
\end{equation}
\begin{equation}\label{3-3-1-20-1}
\ba
 \left| \frac{d}{dr} (r \mathcal{L}g)(1)  \right|^2
 \leq & 4 \int_{\frac12}^1 \left|  \frac{d}{dr} ( r \mathcal{L} g)   \right|^2  \frac{1}{r}\, dr\\
  &\quad    + 8 \left(  \int_{\frac12}^1 | \mathcal{L}^2 g |^2  r\, dr   \right)^{\frac12}  \left(
\int_{\frac12}^1 \left| \frac{d}{dr} ( r \mathcal{L} g) \right|^2 \frac{1}{r} \, dr  \right)^{\frac12},
\ea
\end{equation}
and
\be \label{estlinfty}
|\mathcal{L} g (1)  |  + \left| \frac{d}{dr} ( r \mathcal{L} g) (1)  \right| \leq C \left( \int_0^1| \mathcal{L} g|^2 r \, dr +  \int_0^1 | \mathcal{L}^2 g |^2 r \, dr   \right)^{\frac12}.
\ee

If, in addition, $g(0)=g(1)=0$, then one has
\begin{equation}\label{estLinfty1}
\left|\frac{d}{dr}(rg)(1)\right|\leq  2\sqrt{3} \left( \int_0^{1}\left|g\right|^2 r \, dr \right)^{\frac18} \left( \int_0^{1}\left|\mathcal{L}g\right|^2 r \, dr \right)^{\frac38}.
\end{equation}
\end{lemma}

\begin{proof}
For every $r\in [\frac{1}{2}, 1] $, one has
\begin{equation} \label{A2-5}
\begin{aligned}
\left(\left|\frac{1}{r}\frac{d}{dr}(rg)\right|\Big|_{r=1}\right)^2 = & \left|\frac{1}{r}\frac{d}{dr}(rg)(r)\right|^2 + \int_r^1 \frac{d}{ds}\left|\frac{1}{s}\frac{d}{ds}(sg)\right|^2 \, ds\\
  \leq & \left|\frac{1}{r}\frac{d}{dr}(rg)(r)\right|^2 + 2 \int_r^1 \left|\frac{1}{s}\frac{d}{ds}(sg)\right| |\mathcal{L}g| \,  ds.
  \end{aligned}
  \end{equation}
  Integrating \eqref{A2-5} over $(\frac{1}{2}, 1)$ yields
  \begin{equation} \label{A2-6}
  \begin{aligned}
  \left(\left|\frac{1}{r}\frac{d}{dr}(rg)\right|_{r=1}\right)^2 \leq &
    2 \int_{\frac{1}{2}}^{1}\left|\frac{d}{dr}(rg)\right|^2 \frac{1}{r^2} \, dr +   4 \left( \int_{\frac12}^{1}\left|\mathcal{L}g\right|^2r \, dr  \right)^{\frac12}
   \left(\int_{\frac12}^1  \left|\frac{d}{dr}(rg)\right|^2 \frac{1}{r^3} \,  dr\right)^{\frac12}\\
   \leq &  4 \int_{\frac{1}{2}}^{1} \left|\frac{d}{dr}(rg)\right|^2 \frac{1}{r}\, dr +  8 \left( \int_{\frac12}^{1}\left|\mathcal{L}g\right|^2r \, dr  \right)^{\frac12}
   \left(\int_{\frac12}^1  \left|\frac{d}{dr}(rg)\right|^2 \frac{1}{r} \, dr\right)^{\frac12}.
\end{aligned}
\end{equation}

In particular, if $g(0)= g(1) = 0$, according to Lemma \ref{lemma1}, one has
\be \label{A2-7} \ba
\left| \frac{d}{dr}(r g) (1) \right|^2  & \leq 12 \left( \int_0^1 |\mathcal{L} g|^2 r \, dr   \right)^{\frac12} \left( \int_0^1 \left| \frac{d}{dr} ( rg) \right|^2 \frac1r \, dr  \right)^{\frac12} \\
& \leq 12 \left( \int_0^1 |\mathcal{L} g|^2 r \, dr  \right)^{\frac34} \left( \int_0^1 | g|^2 r \, dr   \right)^{\frac14},
\ea
\ee
which gives \eqref{estLinfty1}.

Similarly, for every $s\in [\frac12, 1]$, one has
\be \la{3-3-1-15}
|\mathcal{L} g (1)|^2
= | s \mathcal{L}g (s) |^2  + \int_s^1 \frac{d}{dr} ( r \mathcal{L}g )  r \mathcal{L} \overline{g} \, dr
+ \int_s^1 r \mathcal{L} g  \frac{d}{dr} ( r \mathcal{L} \overline{g} ) \, dr .
\ee
Integrating \eqref{3-3-1-15}  over $[\frac12, 1]$ with respect to $s$ and using H\"{o}lder inequality yield \eqref{3-3-1-16}.

Furthermore, if $g\in C^4([0,1])$, with $\mathcal{L} g (0) = 0$, then the straightforward computations give
\begin{equation}\label{3-3-1-12}
\int_0^1  \left| \frac{d}{dr} ( r \mathcal{L} g ) \right|^2\frac{1}{r}\, dr
= - \int_0^1 \mathcal{L}^2 g \mathcal{L}\overline{g} r \, dr
+ \frac{d}{dr}(r \mathcal{L} g )(1) \mathcal{L} \overline{g} (1).
\end{equation}
The similar computations yield
\begin{equation}\label{3-3-1-17}
\ba
& \left| \frac{d}{dr} (r \mathcal{L}g)(1)  \right|^2
=   \left| \left[ \frac{1}{r} \frac{d}{dr} ( r \mathcal{L} g)   \right](1) \right|^2 \\
 \leq &  4 \left[ \int_{\frac12}^1  \left| \frac{d}{dr} \left( \frac{1}{r} \frac{d}{dr} (r \mathcal{L}
g)  \right) \right|^2   \frac{1}{r} \, dr   \right]^{\frac12}  \left[ \int_{\frac12}^1
\left|  \frac{1}{r} \frac{d}{dr} (r \mathcal{L} g)  \right|^2  r \, dr \right]^{\frac12}\\
&\ \ \ \  +  4 \int_{\frac12}^1 \left| \frac{1}{r} \frac{d}{dr} ( r \mathcal{L} g )     \right|^2  r \, dr \\
 \leq &    8 \left(  \int_{\frac12}^1 | \mathcal{L}^2 g |^2  r\, dr   \right)^{\frac12}  \left(
\int_{\frac12}^1 \left| \frac{d}{dr} ( r \mathcal{L} g) \right|^2 \frac{1}{r} \, dr  \right)^{\frac12}
+ 4 \int_{\frac12}^1 \left|  \frac{d}{dr} ( r \mathcal{L} g)   \right|^2  \frac{1}{r}\, dr.
\ea
\end{equation}
This gives \eqref{3-3-1-20-1}. If one
denotes
\begin{equation}\nonumber
N_2 = \int_0^1 |\mathcal{L} g|^2 r \, dr ,\,\,\,\,  N_3 =\int_0^1 \left|\frac{d}{dr}
(r \mathcal{L} g) \right|^2 \frac{1}{r}\, dr,\,\,\,\,\text{and}\,\,\,\, N_4 = \int_0^1 |\mathcal{L}^2 g |^2  r \, dr,
\end{equation}
then the inequality \eqref{3-3-1-12}, together with \eqref{3-3-1-16} and \eqref{3-3-1-17}, gives
\begin{equation}\nonumber
N_3 \leq N_2^{\frac{1}{2}}  N_4^{\frac{1}{2}} + \left(2 N_2^{\frac{1}{2}} + 4 N_3^{\frac14} N_2^{\frac14} \right)
\left( 4 N_3^{\frac{1}{2}} + 4 N_3^{\frac14} N_4^{\frac14} \right).
\end{equation}
Therefore, one has
\begin{equation}\label{3-3-1-19}
N_3 \leq 2N_2^{\frac{1}{2}} N_4^{\frac{1}{2}} + C N_2 + C N_2^{\frac23} N_4^{\frac13} \leq C N_2 + C N_2^{\frac12} N_4^{\frac14} \leq  C (N_2 + N_4),
\end{equation}
which is exactly \eqref{3-3-1-20}. The estimate \eqref{estlinfty} follows from \eqref{3-3-1-20} and \eqref{3-3-1-17}.
This finishes the proof of the lemma.
\end{proof}


The following lemma is a variant of Hardy-Littlewood-Polya type inequality.
\begin{lemma}\label{lemmaHLP}
Let $g\in C^1([0,1])$ satisfy $g(0)=0$, one has
\begin{equation}\label{ineqHLP}
\int_0^1|g(r)|^2 dr \leq \frac{1}{2} \int_0^1 |g^{\prime}(r)|^2 (1-r^2) \, dr,
\end{equation}
and
\be \label{HLP-2}
\int_0^1 |g|^2 r \, dr \leq C \int_0^1 \left| \frac{d(r g) }{dr}   \right|^2 \frac{1-r^2}{r} \, dr .
\ee
\end{lemma}
\begin{proof}
It follows from \cite[p. 165]{HLP} that for any function $g\in C^1([-1,1])$, one has
\begin{equation}\label{ineqHLP1}
2\int_{-1}^1 |g(t)|^2dt -\left(\int_{-1}^1 g(t)dt\right)^2 \leq \int_{-1}^{1} (1 - t^2) |g^{\prime}(t)|^2 dt.
\end{equation}
Let $\tilde{g}$ be the odd extension of $g$ on $[-1,1]$, i.e.,
\begin{equation}
\tilde{g} (t) =\left\{
\begin{aligned}
& g(t), \quad \, \ \  \ \ \ \,\, 0\leq t \leq 1 ,\\
& - g(-t), \quad \,\, -1 \leq t < 0 .
\end{aligned}
\right.
\end{equation}
Since $g(0)= 0$ and $g\in C^1[0, 1]$, $\tilde{g} \in C^1[-1, 1]$.
Using \eqref{ineqHLP1} for $\tilde{g}$ gives
\begin{equation}
\begin{aligned}
& 4\int_{0}^1 |g(t)|^2 =2\int_{-1}^1 |\tilde{g}(t)|^2dt -\left(\int_{-1}^1 \tilde{g}(t)dt\right)^2\\
\leq &  \int_{-1}^1 (1-t^2) |\tilde{g}'(t)|^2dt =2\int_0^1 (1-t^2)|{g}'(t)|^2dt
\end{aligned}
\end{equation}
This yields the inequality \eqref{ineqHLP}.

Applying \eqref{ineqHLP} for the function $r g$ gives
\be \label{HLP-5}
\int_0^1 |r g|^2 \, dr \leq \frac12  \int_0^1  (1 - r^2) \left| \frac{d ( r g )}{dr}   \right|^2
\, dr,
\ee
which implies that
\be \label{HLP-6}
\int_{\frac12}^1 |g|^2 r \, dr \leq  \int_0^1 \left| \frac{d}{dr}( r g ) \right|^2 \frac{1-r^2}{r} \, dr.
\ee
On the other hand, one has
\be \nonumber
|r g(r)| = \left| \int_0^r \frac{d}{ds}(s g) \, ds  \right| \leq
\left( \int_0^r \left|  \frac{d}{ds}( s g) \right|^2 \frac{1}{s} \, ds  \right)^{\frac12} r,
\ee
which implies
\be \label{HLP-7}
\sup_{r\in [0, \frac12]}  |g(r)|^2 \leq \frac43 \int_0^{\frac12} \left| \frac{d}{dr}( r g)   \right|^2 \frac{1-r^2}{r}\, dr.
\ee
Hence, one has
\be \label{HLP-8}
\int_0^{\frac12} |g(r)|^2 r \,dr \leq \frac23  \int_0^1 \left| \frac{d}{dr}( r g)  \right|^2 \frac{1-r^2}{r} \, dr .
\ee
The estimate \eqref{HLP-2} follows from \eqref{HLP-6} and \eqref{HLP-8}. The proof of Lemma \ref{lemmaHLP} is completed.
\end{proof}

We collect some basic properties of the modified Bessel functions of the first kind in the following lemma.
\begin{lemma}\label{lemBessel}
Let $I_1(z)$ be the modified Bessel function of the first kind, i.e. it is the solution of the problem
\eqref{eqBessel1}.
Assume that $0 < x <y$, it holds that
\be \label{Bessel1}
e^{x - y} \frac{x}{y} < \frac{I_1(x)}{I_1 (y)} < e^{x - y} \left(\frac{y}{x}\right)^{1/2}.
\ee
Furthermore, for every $x>0$, it holds that
\be\label{Bessel1-5}
\frac{x}{2}\leq  I_1(x)  \leq \frac{x}{2} \cosh x
\ee
and
\be \label{Bessel2}
 0\leq I_1^{\prime}(x)  \leq I_1(x)+\frac{I_1(x)}{x}.
\ee
\end{lemma}

\begin{proof}
The first inequality in \eqref{Bessel1} was proved in \cite{Bordelon}, while the second inequality was proved in \cite[(2.6)]{Baricz}. The inequality \eqref{Bessel1-5} can be found in  \cite[(1.1)]{Ifantis}. Moreover, it follows from \cite[p. 79]{Watson} that
\be \nonumber
I_1^{\prime}(x) = I_2(x)+ x^{-1} I_1(x).
\ee
This, together with the estimate (\cite[(2.16)]{Baricz})
\be \nonumber
I_2(x) < I_1(x) \tanh x,
\ee
gives  \eqref{Bessel2}. Hence the proof of  Lemma \ref{lemBessel} is completed.
\end{proof}

The integrals of the modified Bessel functions of the first kind are estimated in the following lemma.
\begin{lemma}\label{AlemBessel2}
It holds that
\be \label{A-96}
\int_0^1 | I_1(|\xi| r)|^2 r\, dr
\leq  C \min\{1, |\xi|^{-1}\} (I_1(|\xi|))^2
\ee
and
\be \label{A-97}
 \int_0^1 \left| \frac{d}{dr} \left(r I_1 (|\xi| r) \right) \right|^2 \frac1r \, dr
\leq  C \max\{1, |\xi|\} (I_1(|\xi|))^2.
\ee
Furthermore, one has
\be\label{A-98}
\int_0^1 | \mathcal{L}I_1(|\xi| r)|^2 r\, dr
\leq  C  \min\{1, |\xi|^{-1}\} \xi^4 (I_1(|\xi|))^2
\ee
and
\be\label{A-99}
\int_0^1 \left| \frac{d}{dr} \left( r \mathcal{L} I_1(|\xi|r)\right)  \right|^2 \frac1r \, dr
\leq  C  \max\{1, |\xi|\} \xi^4 (I_1(|\xi|))^2.
\ee
\end{lemma}
\begin{proof}
The straightforward computations give
\be \label{A-90-1}
 \int_0^1 |I_1 (|\xi| r)|^2 r \, dr  = |\xi|^{-2} \int_0^{|\xi|} |I_1 (\eta) |^2 \eta \, d\eta.
\ee
For $|\xi| \in (0, 1]$, since $I_1 (\eta) $ is increasing, one has
\be \label{A-90-2} \ba
\int_0^{|\xi|} |I_1 (\eta)|^2 \eta \, d \eta
\leq  C  (I_1(|\xi|))^{2} \int_0^{|\xi|} \left| \frac{I_1(\eta)}{I_1(|\xi|)} \right|^2 \eta \, d\eta
\leq C  \xi^{2} (I_1(|\xi|^2))^2.
\ea \ee
When $|\xi| > 1$, since $I_1(\eta)$ is an increasing function, one has
\be\label{Besseltrick}  \ba
&\int_0^{|\xi|} | I_1 (\eta) |^2 \eta \, d\eta
=  \int_0^{\frac{|\xi|}{2}} \left| I_1 (\eta ) \right|^2 \eta \, d\eta +  \int_{\frac{|\xi|}{2}}^{|\xi|} \left| I_1 (\eta ) \right|^2 \eta \, d\eta
\leq   2 \int_{\frac{|\xi|}{2}}^{|\xi|} \left| I_1 (\eta ) \right|^2 \eta \, d\eta.
\ea
\ee
It follows from \eqref{Bessel1} in Lemma \ref{lemBessel} that one has
\be \label{A-90-3}
\ba
&\int_0^{|\xi|} | I_1 (\eta) |^2 \eta \, d\eta
\leq  2  (I_{1}(|\xi|))^2 \int_{\frac{|\xi|}{2}}^{|\xi|}  \left| \frac{I_1 (\eta )}{I_1(|\xi|) } \right|^2 \eta \, d\eta \\
\leq & C(I_{1}(|\xi|))^2 \int_{\frac{|\xi|}{2}}^{|\xi|} e^{2(\eta - |\xi|)} \eta \, d\eta
\leq  C |\xi| (I_{1}(|\xi|))^2.
\ea
\ee
The estimate \eqref{A-96} follows from \eqref{A-90-1}-\eqref{A-90-2} and \eqref{A-90-3}.

Next, the straightforward computations give
\be \label{A-90-4}
\ba
 \int_0^1 \left| \frac{d}{dr} ( r I_1 (|\xi| r) )   \right|^2 \frac1r \, dr
\leq &  2 \int_0^1 |I_1 (|\xi| r) |^2  \frac1r \, dr + 2  \xi^2 \int_0^1 | I_1^{\prime} (|\xi| r)|^2 r \, dr \\
= & 2  \int_0^{|\xi|} |I_1 (\eta)|^2 \frac{1}{\eta} \, d \eta + 2  \int_0^{|\xi|} | I_1^{\prime} (\eta) |^2 \eta \, d\eta .
\ea
\ee
For $|\xi| \in (0, 1]$, it follows from \eqref{Bessel1-5} in Lemma \ref{lemBessel} that one has
\be \label{A-90-5} \ba
& \int_0^{|\xi|} |I_1 (\eta)|^2 \frac{1}{\eta} \, d\eta
\leq  \int_0^{|\xi|} \left|  \frac{\eta \cosh \eta }{2  } \right|^2 \frac{1}{\eta} \, d\eta
\leq  C |\xi|^2\leq C(I_1(|\xi|))^2.
\ea
\ee
Since $I_1(\eta)$ is smooth on $[0, +\infty)$, $I_1^{\prime}(\eta)$ is bounded on $[0, 1]$. Hence,
\be \label{5-90-6}
\ba
& \int_0^{|\xi|} |I_1^{\prime} (\eta) |\eta \, d\eta
\leq  C |\xi|^2 \leq C (I_1(|\xi|))^2.
\ea
\ee
When $|\xi| >1$, it follows from Lemma \ref{lemBessel}  that
\be \label{A-90-7}
\ba
&\int_0^{|\xi|} | I_1 (\eta)|^2 \frac{1}{\eta} \, d \eta
\leq  \int_0^1 |I_1(\eta)|^2 \frac{1}{\eta} \, d\eta +  \int_1^{|\xi|} |I_1(\eta)|^2 \frac{1}{\eta} \, d\eta \\
\leq & C  \int_0^1 \left| \frac{\eta \cosh \eta }{2} \right|^2 \frac{1}{\eta} \, d\eta
+ (I_1(|\xi|))^2  \int_{1}^{|\xi|} \left| \frac{ I_1(\eta)} {I_1(|\xi|)}\right|^2 \frac{1}{\eta} \, d\eta  \\
\leq & C  + (I_1(|\xi|))^2  \int_{1}^{|\xi|} e^{2(\eta-|\xi|)}\frac{|\xi|}{\eta} \frac{1}{\eta} \, d\eta  \\
\leq & C + (I_1(|\xi|))^2 \left( e^{-2|\xi|}|\xi| \int_{1}^{\frac{|\xi|}{2}} \frac{e^{2\eta}}{\eta^2} \, d\eta  +\frac{4e^{-2|\xi|}}{|\xi|}  \int_{\frac{|\xi|}{2}}^{|\xi|} e^{2\eta} \, d\eta \right)\\
\leq & C
+ (I_1(|\xi|))^2 \left( |\xi|e^{-|\xi|} + \frac{1}{|\xi|}\right)  \\
\leq& \frac{C}{|\xi|} (I_1(|\xi|))^2,
\ea
\ee
where we used the estimate \eqref{Bessel1-5} for the last inequality.
Similarly, due to the boundedness of $I_1^{\prime}(\eta)$ on $[0, 1]$ and \eqref{Bessel2}, one has
\be \label{A-90-8} \ba
&  \int_0^{\xi} |I_1^{\prime} (\eta) |^2 \eta \, d\eta
\leq  \int_0^1  |I_1^{\prime} (\eta) |^2 \eta \, d\eta + \int_1^{|\xi|} |I_1^{\prime} (\eta) |^2\eta \, d\eta\\
\leq&    C  + 2\left(
\int_1^{|\xi|} |I_1(\eta)|^2 \frac{1}{\eta} \, d\eta + \int_1^{|\xi|} |I_1(\eta)|^2 {\eta} \, d\eta\right) \\
\leq & C + 2 \int_1^{|\xi|} |I_1(\eta)|^2 \eta \, d\eta  .
\ea \ee
It follows from  \eqref{A-90-3} that
\be \label{A-90-8-5} \ba
&  \int_0^{\xi} |I_1^{\prime} (\eta) |^2 \eta \, d\eta
\leq C  |\xi|  (I_1(|\xi|))^2\quad \mbox{for}\,\, |\xi|\geq 1.
\ea \ee
Combining the estimate \eqref{A-90-4}-\eqref{A-90-8-5} yields \eqref{A-97}.

Since $( \mathcal{L} -\xi^2) I_1(|\xi|r) = 0$, one has
\be \label{A-90-9} \ba
 \int_0^1 |\mathcal{L}  I_1(|\xi|r) |^2 r \, dr
=  |\xi|^4 \int_0^1 | I_1(|\xi|r)|^2 r \, dr
\ea
\ee
and
\be \label{A-90-10} \ba
 \int_0^1\frac{1}{r} \left|\frac{d}{dr}(r\mathcal{L}  I_1(|\xi|r)) \right|^2  \, dr
=  |\xi|^4\int_0^1\frac{1}{r} \left|\frac{d}{dr}(r  I_1(|\xi|r)) \right|^2  \, dr.
\ea \ee
Thus the estimates \eqref{A-98}-\eqref{A-99} follow from \eqref{A-90-9}-\eqref{A-90-10} and \eqref{A-96}-\eqref{A-97}.
\end{proof}

The following lemma is about two weighted interpolation inequalities, which are quite similar to \cite[(3.28)]{M}.
\begin{lemma}\label{weightinequality} Let $g \in C^2[0, 1]$, then one has
\begin{equation} \label{weight1} \ba
\int_0^1 |g|^2r \, dr  \leq & C \left(\int_0^1 (1-r^2)|g|^2 r \, dr\right)^{\frac23} \left(\int_0^1  \left|\frac{d}{dr}(rg)\right|^2 \frac{1}{r} \, dr\right)^{\frac13} \\
&\ \ \ \  + C \int_0^1 (1-r^2)|g|^2 r\, dr ,
\ea
\end{equation}
and
\be \label{weight2} \ba
\int_0^1 \left| \frac{d}{dr}(rg) \right|^2 \frac1r \, dr &  \leq C \left( \int_0^1 \frac{1-r^2}{r} \left| \frac{d}{dr} ( rg)  \right|^2 \, dr \right)^{\frac23} \left( \int_0^1 |\mathcal{L} g|^2 r \, dr  \right)^{\frac13} \\
&\ \ \ \ \ \ \ \ \ \ \ + C \int_0^1 \frac{1-r^2}{r} \left| \frac{d}{dr} ( rg)  \right|^2 \, dr .
\ea
\ee
\end{lemma}
\begin{proof}
First, one has
\begin{equation}
\begin{aligned}
\int_0^1 |g|^2r \, dr \leq & 2 \int_0^{\frac12} (1-r^2)|g|^2 r \, dr+  \int_{\frac12}^{1-\delta} \frac{1-r^2}{\delta} |g|^2 r \, dr  + \int_{1-\delta}^1 |g|^2 r \, dr\\
\leq & 2 \int_0^{\frac12} (1-r^2)|g|^2 r \, dr+  \frac{2}{\delta} \int_{\frac12}^{1-\delta} {(1-r^2)}|g|^2 r \, dr  + 2\delta \sup_{r\in [0,1]} |rg(r)|^2 .
\end{aligned}
\end{equation}
If one chooses
\begin{equation}
\delta =  \frac{ \displaystyle \left(\int_0^1 (1-r^2)|g|^2 r \, dr\right)^{\frac12}}{ \displaystyle \sup_{r\in [0,1]} |rg(r)|+\left(\int_0^1 (1-r^2)|g|^2 r \, dr\right)^{\frac12}},
\end{equation}
then one has
\begin{equation}
\int_0^1 |g|^2rdr
\leq  4 \int_0^{1} (1-r^2)|g|^2 r\, dr+ 4 \left(\int_{0}^{1} {(1-r^2)}|g|^2 r\, dr \right)^{\frac12} \sup_{r\in [0,1]} |rg(r)|.
\end{equation}
Note that
\begin{equation}
\begin{aligned}
|r g(r)|^2  = &\int_0^r \frac{d}{ds} \left| (sg(s))\right|^2 \,  ds  \\
\leq & 2 \left(\int_0^r  \left|\frac{d}{ds}(sg(s))\right|^2 \frac{1}{s} \, ds\right)^{\frac12}  \left(\int_0^r \left|sg(s)\right|^2 s \, ds\right)^{\frac12 }\\
\leq & 2  \left(\int_0^1 \left|\frac{d}{ds}(sg(s))\right|^2 \frac{1}{s}  \, ds\right)^{\frac12 }  \left(\int_0^1 \left|sg(s)\right|^2 s\, ds\right)^{\frac12}.
\end{aligned}
\end{equation}
Hence, one has
\begin{equation}
\begin{aligned}
\int_0^1 |g|^2r \, dr
\leq &  4 \int_0^{1} (1-r^2)|g|^2 r\, dr\\
&\quad + 8 \left(\int_{0}^{1} {(1-r^2)}|g|^2 r \, dr \right)^{\frac12} \left(\int_0^1  \left|\frac{d}{dr}(rg)\right|^2 \frac{1}{r}\, dr \right)^{\frac14}  \left(\int_0^1 \left|rg \right|^2 r \, dr\right)^{\frac14}\\
\leq & 4 \int_0^{1} (1-r^2)|g|^2 r\, dr + C  \left(\int_{0}^{1} {(1-r^2)}|g|^2 r\, dr \right)^{\frac23} \left(\int_0^1 \left|\frac{d}{dr}(rg)\right|^2 \frac{1}{r}  \, dr\right)^{\frac13}\\
&\quad  +  \frac{1}{2}\int_0^1 \left|g \right|^2 r \, dr.
\end{aligned}
\end{equation}
Thus we get
 \eqref{weight1}.

The proof of \eqref{weight2} is similar to that of \eqref{weight1}. First, one has
\be \label{weight15}
\int_0^{\frac12} \left| \frac{d}{dr} (rg) \right|^2 \frac1r \, dr \leq \frac43 \int_0^1 \frac{1 - r^2}{r} \left| \frac{d}{dr} (r g)  \right|^2 \, dr.
\ee
Second, the straightforward computations give
\be \label{weight16} \ba
\int_{\frac12}^1 \left| \frac{d}{dr} ( r g) \right|^2 \frac1r \, dr
& \leq \frac{1}{\delta} \int_{\frac12}^{1 -\delta} \frac{1-r^2}{r} \left| \frac{d}{dr} (r g) \right|^2 \, dr +  \int_{1- \delta}^1 \left| \frac1r \frac{d}{dr}(r g) \right|^2 \, dr \\
& \leq \frac{1}{\delta} \int_{\frac12}^{1- \delta} \frac{1-r^2}{r} \left| \frac{d}{dr} (r g) \right|^2 \, dr +  \delta \sup_{r \in [\frac12, 1]} \left| \frac{d}{dr} ( rg ) \frac1r  \right|^2 .
\ea
\ee
Let us choose
\be \nonumber
\delta = \frac{\displaystyle \left( \int_{\frac12}^1 \frac{1- r^2 }{r} \left| \frac{d}{dr} (r g ) \right|^2 \, dr \right)^{\frac12}    }{ \displaystyle \sup_{r \in [\frac12, 1]} \left| \frac{d}{dr} ( rg ) \frac1r  \right| + \left( \int_{\frac12}^1 \frac{1- r^2 }{r} \left| \frac{d}{dr} (r g ) \right|^2 \, dr \right)^{\frac12}}.
\ee
Therefore, it follows from \eqref{weight16} that one has
\be \label{weight17} \ba
\int_{\frac12}^1 \left| \frac{d}{dr} (r g)  \right|^2  \frac1r \, dr
\leq & 2 \int_{\frac12}^1  \frac{1-r^2}{r}  \left| \frac{d}{dr} (r g)  \right|^2  \frac1r \, dr \\
&\ \ \ \ \ + 2 \left( \int_{\frac12}^1 \frac{1- r^2 }{r} \left| \frac{d}{dr} (r g ) \right|^2 \, dr \right)^{\frac12} \sup_{r \in [\frac12, 1]} \left| \frac{d}{dr} ( rg ) \frac1r  \right|.
\ea
\ee
For every $r, s \in [\frac12, 1]$, one has
\be \label{weight18}
\left| \frac1r \frac{d}{dr}( r g)  \right|^2
\leq \left| \frac1s \frac{d}{ds}( s g)  \right|^2 + 2 \int_s^r | \mathcal{L} g|  \left| \frac1t \frac{d}{dt}(t g) \right| \, dt .
\ee
Integrating \eqref{weight18} with respect to $s$ over $[\frac12, 1]$ yields
\be \label{weight19}
\left| \frac1r \frac{d}{dr}( r g)  \right|^2
\leq 4 \int_{\frac12}^1 \left| \frac{d}{ds} (s g)  \right|^2 \frac1s \, ds + 4 \int_{\frac12}^1 |\mathcal{L}g | \left| \frac{d}{ds} (sg)  \right| \, ds.
\ee
Hence one has
\be \label{weight20}
\ba
& \int_{\frac12}^1 \left| \frac{d}{dr} (r g ) \right|^2 \frac1r \, dr  \\
\leq & 2 \int_{\frac12}^1 \frac{1- r^2}{r} \left| \frac{d}{dr} ( rg) \right|^2 \, dr
+ 8 \left( \int_{\frac12}^1 \frac{1- r^2}{r} \left| \frac{d}{dr} ( rg) \right|^2 \, dr   \right)^{\frac12} \left( \int_{\frac12}^1 \left| \frac{d}{dr} (r g ) \right|^2 \frac1r \, dr \right)^{\frac12} \\
& \ \ + 8 \left( \int_{\frac12}^1 \frac{1- r^2}{r} \left| \frac{d}{dr} ( rg) \right|^2 \, dr   \right)^{\frac12}
\left(  \int_{\frac12}^1 |\mathcal{L} g |^2 r \, dr \right)^{\frac14} \left( \int_{\frac12}^1 \left| \frac{d}{dr} (r g ) \right|^2 \frac1r \, dr  \right)^{\frac14} \\
\ea
\ee
By Young's inequality, one proves that
\be \label{weight21} \ba
\int_{\frac12}^1 \left| \frac{d}{dr} (r g ) \right|^2 \frac1r \, dr & \leq
C \left( \int_{\frac12}^1 \frac{1- r^2}{r} \left| \frac{d}{dr} ( rg) \right|^2 \, dr   \right)^{\frac23} \left(  \int_{\frac12}^1 |\mathcal{L} g |^2 r \, dr \right)^{\frac13} \\
&\ \ \ \ \ + C  \int_{\frac12}^1 \frac{1- r^2}{r} \left| \frac{d}{dr} ( rg) \right|^2 \, dr.
\ea
\ee
The inequalities \eqref{weight15} and \eqref{weight21} give the inequality \eqref{weight2}. This finishes the proof of Lemma \ref{weightinequality}.
\end{proof}

The following elementary fixed point theorem is the basic tool to prove the nonlinear structural stability.
\bl \la{nonlinear} Let $Y$ be a Banach space with the norm $\|\cdot\|_{Y}$, and $\mathcal{B}:\  Y\times Y\rightarrow Y$ be a bilinear map. If for all $\zeta_1, \zeta_2 \in Y$, one has
\be \nonumber
\|\mathcal{B}(\zeta_1, \zeta_2) \|_{Y} \leq \eta \|\zeta_1\|_{Y}  \|\zeta_2\|_{Y},
\ee
then for all $\zeta^* \in Y$ satisfying $4 \eta \| \zeta^* \|_{Y} < 1$, the equation
\be \nonumber
\zeta = \zeta^*  + \mathcal{B}(\zeta, \zeta)
\ee
has a unique solution $\zeta \in Y$ satisfying
\be \nonumber
\|\zeta\|_{Y} \leq 2\|\zeta^*\|_{Y}.
\ee
\el

\begin{proof}
Let $\zeta_0 = \zeta^*$ and $\zeta_n = \zeta^* + \mathcal{B}(\zeta_{n-1}, \zeta_{n-1})$ for $n \in \mathbb{N}$. First, we  prove that
for every $n \in \mathbb{N}$,
\be \la{nonlinear1}
\| \zeta_n\|_{Y} \leq 2 \|\zeta^*\|_{Y}.
\ee
Obviously, the assertion \eqref{nonlinear1} holds for $n=0$.
Assume that the assertion \eqref{nonlinear1} holds for $n = k$, then
\be \nonumber
\|\zeta_{k+1}\|_{Y} \leq \|\zeta^*\|_{Y}  + \eta \|\zeta_k \|_{Y}^2 \leq \|\zeta^*\|_{Y} + 4\eta \|\zeta^*\|_{Y}^2
\leq 2 \|\zeta^*\|_{Y}.
\ee
Hence,  the assertion \eqref{nonlinear1} holds for every $n \in \mathbb{N}$ by induction.

Next, we  prove that $\{ \zeta_n \}$ converges in $Y$.
\be \la{nonlinear2} \ba
\| \zeta_{n+1} - \zeta_n\|_{Y} & = \|\mathcal{B}(\zeta_n , \zeta_n) - \mathcal{B}(\zeta_{n -1}, \zeta_{n-1} )\|_{Y }\\
& \leq \eta \|\zeta_n - \zeta_{n-1}\|_{Y}  \left( \|\zeta_n\|_{Y} + \|\zeta_{n-1}\|_{Y} \right)\\
& \leq 4 \eta  \|\zeta^*\|_{Y}  \|\zeta_n - \zeta_{n-1} \|_{Y}.
\ea
\ee
Since $4\eta  \|\zeta^*\|_{Y} < 1, $ \eqref{nonlinear2} implies that $\{\zeta_n\}$ converges in $Y$.

Let $\zeta= \lim_{n \rightarrow + \infty} \zeta_n$. It is easy to check that
\be \nonumber
\zeta = \zeta^* + \mathcal{B}(\zeta, \zeta)\ \ \ \mbox{and}\ \ \|\zeta\|_{Y} \leq 2 \|\zeta^*\|_{Y}.
\ee
If there exists another solution $\tilde{\zeta}$ satisfying $\|\tilde{\zeta} \|_{Y} \leq 2 \|\zeta^*\|_{Y}$, then
\be \nonumber
\|\zeta - \tilde{\zeta}\|_{Y} \leq 4\eta \|\zeta^*\|_{Y} \|\zeta - \tilde{\zeta}\|_{Y},
\ee
which implies that $\zeta = \tilde{\zeta}$. Hence, the proof of Lemma \ref{nonlinear} is completed.
\end{proof}

\section{Analysis on vorticity}\label{secappendix}

In this appendix, we give the detailed proof for Lemmas \ref{lemma3-4-3} and  \ref{lemma3-4-4}.

\begin{proof}[Proof of Lemma \ref{lemma3-4-3}]
Direct calculations show that
in the domain $\Omega_{r_0}$,
\be \nonumber
\partial_x \Bo^\theta = \left(-\left( \partial_r \omega^\theta - \frac{\omega^\theta}{r} \right) \cos \theta \sin \theta, \
\partial_r \omega^\theta \cos^2 \theta + \frac{\omega^\theta}{r} \sin^2 \theta , \  0                 \right),
\ee
\be \nonumber
\partial_y \Bo^\theta = \left( -\partial_r \omega^\theta \sin^2 \theta - \frac{\omega^\theta}{r} \cos^2 \theta,
\ \left(  \partial_r \omega^\theta - \frac{\omega^\theta}{r} \right) \cos \theta \sin \theta  , \ 0                \right),
\ee
\be \nonumber
\partial_z \Bo^\theta = \partial_z \omega^\theta \Be_\theta,
\ee
\be \nonumber
 \partial_x^2 \Bo^\theta = \left( \ba & -\partial_r^2 \omega^\theta \cos^2 \theta \sin \theta - \left( \frac{\partial_r \omega^\theta}{r} - \frac{\omega^\theta}{r^2} \right) (\sin^3 \theta - 2\cos^2 \theta \sin \theta)  \\
&\partial_r^2 \omega^\theta \cos^3 \theta + 3\left(
\frac{\partial_r \omega^\theta}{r}- \frac{\omega^\theta}{r^2}  \right) \sin^2 \theta \cos \theta \\
&  0                      \ea \right)^t,
\ee
\be \nonumber
 \partial_y^2 \Bo^\theta = \left( \ba & -\partial_r^2 \omega^\theta  \sin^3 \theta   - 3 \left( \frac{\partial_r \omega^\theta}{r} - \frac{\omega^\theta}{r^2} \right) \cos^2 \theta \sin \theta  \\
&\partial_r^2 \omega^\theta \cos \theta \sin^2 \theta + \left(
\frac{\partial_r \omega^\theta}{r}- \frac{\omega^\theta}{r^2}  \right) ( \cos^3 \theta - 2 \sin^2 \theta \cos \theta)  \\
&  0                      \ea \right)^t,
\ee
\be \nonumber
 \partial_z^2 \Bo^\theta = ( - \partial_z^2 \omega^\theta \sin\theta, \ \partial_z^2 \omega^\theta \cos \theta, \ 0 ),
\ee
\be \nonumber
 \partial_x \partial_y \Bo^\theta= \left( \ba & -\partial_r^2 \omega^\theta  \sin^2 \theta \cos \theta    -  \left( \frac{\partial_r \omega}{r} - \frac{\omega^\theta}{r^2} \right) (\cos^3 \theta - 2 \sin^2 \theta \cos \theta )  \\
&\partial_r^2 \omega^\theta \cos^2 \theta \sin \theta + \left(
\frac{\partial_r \omega^\theta}{r}- \frac{\omega^\theta}{r^2}  \right) ( \sin^3 \theta - 2 \cos^2 \theta \sin \theta)  \\
&  0                      \ea \right)^t,
\ee
\be \nonumber
\begin{aligned}
 \partial_x \partial_z \Bo^\theta = \left( \ba & - \left( \partial_r \partial_z \omega^\theta - \frac{\partial_z \omega^\theta}{r} \right) \sin \theta \cos \theta  \\
& \partial_r \partial_z \omega^\theta \cos^2 \theta + \frac{\partial_z \omega^\theta}{r} \sin^2 \theta   \\
&  0                      \ea \right)^t,
\end{aligned}\quad
\begin{aligned}
 \partial_y \partial_z \Bo^\theta = \left( \ba &  - \partial_r \partial_z \omega^\theta \sin^2 \theta - \frac{\partial_z \omega^\theta}{r} \cos^2 \theta  \\
& \left(   \partial_r \partial_z \omega^\theta - \frac{\partial_z \omega^\theta}{r}  \right) \cos \theta \sin \theta   \\
&  0                      \ea \right)^t.
\end{aligned}
\ee
It follows from Lemma \ref{lemma3-4-1} that
\be \nonumber \ba
&\,\, \int_{- \infty}^{+ \infty} \int_{r_0}^1  \left| \frac{\partial_r \omega^\theta}{r} \right|^2  r \, dr dz
 \leq  C \int_{-\infty}^{+ \infty} \int_{r_0}^1 \left( \left| \frac{\partial}{\partial r} \mathcal{L} \hat{\psi} \right|^2 + \xi^4 \left|    \frac{\partial \hat{\psi} }{\partial r} \right|^2 \right)  \frac1r  \,  dr d\xi
 \leq \,\, \frac{C}{r_0^2} \|\psi\|_{H_r^4(D)}^2.
\ea \ee
Similarly, it holds that
\be \nonumber \ba
&\,\, \int_{-\infty}^{+ \infty} \int_{r_0}^1 \left| \frac{\omega^\theta}{r^2} \right|^2  r \, dr dz
 \leq  C \int_{-\infty}^{+\infty} \int_{r_0}^1 \left(   | \mathcal{L} \hat{\psi} |^2 + \xi^4 |\hat{\psi} |^2   \right)  \frac{1}{r^3} \, dr d\xi \\
 \leq &\,\, \frac{C}{r_0^4} \int_{-\infty}^{+ \infty} \int_{r_0}^1 \left( | \mathcal{L} \hat{\psi} |^2 + \xi^4 |\hat{\psi}|^2      \right)  r \, dr d\xi
 \leq \,\,  \frac{C}{r_0^4} \| \psi \|_{H_r^4(D)}^2
\ea \ee
and
\be \nonumber \ba
&\,\, \int_{-\infty}^{+ \infty} \int_{r_0}^1 | \partial_r \partial_z \omega^\theta |^2 r \, dr dz
 \leq  C \int_{-\infty}^{+ \infty} \int_{r_0}^1  \left( \xi^2 \left|\frac{\partial}{\partial r} \mathcal{L} \hat{\psi}   \right|^2 + \xi^6 \left| \frac{\partial}{\partial r} \hat{\psi}  \right|^2         \right)  r \, dr d\xi  \\
 \leq &\,\, C \int_{-\infty}^{+ \infty} \int_{r_0}^1  \left(  \xi^2 \left| \frac{\partial}{\partial r} (r \mathcal{L} \hat{\psi} ) \right|^2
+ \xi^2 | \mathcal{L} \hat {\psi}|^2 + \xi^6 \left| \frac{\partial}{\partial r} ( r\hat{\psi} )   \right|^2 + \xi^6 |\hat{\psi}|^2 \right)  \frac{1}{r} \, dr d\xi \\
 \leq&\,\, \frac{C}{r_0^2} \|\psi\|_{H_r^4(D)}^2.
\ea \ee
Furthermore, one has
\be \nonumber \ba
&\,\, \int_{-\infty}^{+ \infty} \int_{r_0}^1 \left| \frac{\partial_z \omega^\theta}{r}     \right|^2  r \, dr dz
 \leq  C \int_{-\infty}^{+ \infty} \int_{r_0}^1 \left(
\xi^2 | \mathcal{L} \hat{\psi}|^2 + \xi^6 |\hat{\psi} |      \right)  \frac1r \, dr d\xi\\
 \leq &\,\, \frac{C}{r_0^2} \int_{-\infty}^{+ \infty} \int_{r_0}^1 \left( \xi^2 | \mathcal{L} \hat{\psi} |^2 + \xi^6 |\hat{\psi}|^2 \right)  r \, dr d\xi
 \leq \,\, \frac{C}{r_0^2} \|\psi\|_{H_r^4(D)}^2
\ea \ee
and
\be \nonumber \ba
&\,\, \int_{- \infty}^{+ \infty} \int_{r_0}^1 \left| \partial_r^2 \omega^\theta \right|^2  r\, dr dz
 =  \int_{-\infty}^{+\infty} \int_{r_0}^1 \left| \mathcal{L} \omega^\theta - \frac{\partial_r \omega^\theta}{r} + \frac{\omega^\theta}{r^2} \right|^2  r \, drdz \\
\leq &\,\, C \int_{-\infty}^{+\infty} \int_{r_0}^1 \left( | \mathcal{L}^2 \psi|^2 + |\partial_z^2 \mathcal{L} \psi|^2 \right) r \, dr dz
+ \frac{C}{r_0^4} \|\psi\|_{H_r^4(D)}^2
 \leq \,\, \frac{C}{r_0^4} \|\psi \|_{H_r^4(D)}^2.
\ea \ee
Thus we get  the following estimate,
\be \la{3-4-31}
\|\Bo^\theta \|_{H^2( \Omega_{r_0}  )}
\leq C(r_0) \|\psi \|_{H_r^4(D)},
\ee
Hence the proof of  Lemma \ref{lemma3-4-3} is finished.
\end{proof}

Now we give the proof of Lemma \ref{lemma3-4-4}.
\begin{proof}[Proof of Lemma \ref{lemma3-4-4}] With the $H^2$-estimate for $\Bo^\theta$ over the domain $\Omega_r$, it suffices to get the interior
$H^2$-estimate for $\Bo^\theta$. First, we claim that the equation
\be \la{3-4-40}
\Delta \Bo^\theta = (\mathcal{L} + \partial_z^2)^2 \psi \Be_\theta.
\ee
actually holds  on the whole domain $\Omega$.

Indeed, let $\Bp \in C_0^\infty(\Omega; \mathbb{R}^3)$ be  a smooth vector-valued function defined on $\Omega$, and suppose that $\supp(\Bp) \subseteq B_1(0) \times [-Z, Z]$. One has
\be \la{3-4-41} \ba
&\,\, \int_{- \infty}^{ + \infty} \int_{B_1(0)\setminus B_r(0) } \Bo^\theta \cdot \Delta \Bp \, dx dy dz  \\
 = &\,\, - \int_{-\infty}^{+ \infty} \int_{B_1(0) \setminus B_r(0) } \nabla \Bo^\theta : \nabla \Bp \, dx dy dz
- \int_{-\infty}^{+\infty} \int_{\partial B_r (0) } \Bo^\theta \cdot \frac{\partial \Bp }{\partial \Bn } \, dSdz \\
= &\,\, \int_{-\infty}^{+\infty} \int_{B_1(0) \setminus B_r(0) } \Delta \Bo^\theta \cdot \Bp \, dx dy dz
+ \int_{-\infty}^{+\infty} \int_{\partial B_r(0)} \frac{\partial \Bo^\theta } {\partial \Bn } \cdot \Bp \, dSdz\\
&\ \ \ \
- \int_{-\infty}^{+\infty} \int_{\partial B_r(0)} \Bo^\theta \cdot \frac{\partial \Bp }{\partial \Bn} \, dSdz,
\ea \ee
where $\Bn$ is the unit outer normal.
Since $\Bo^\theta\in L^2(\Omega)$ and
\be \nonumber
\int_{-\infty}^{+ \infty} \int_0^1  \left| ( \mathcal{L} + \partial_z^2)^2 \psi    \right|^2  r \, dr dz \leq C \|\psi\|_{H_r^4(D) }^2,
\ee
it follows that
\be \la{3-4-42}
\lim_{r\rightarrow 0+ }   \int_{- \infty}^{ + \infty} \int_{B_1(0)\setminus B_r(0) } \Bo^\theta \cdot \Delta \Bp \, dx dy dz
= \int_{- \infty}^{ + \infty} \int_{B_1(0)} \Bo^\theta \cdot \Delta \Bp \, dx dy dz
\ee
and
\be \la{3-4-43}
\lim_{r \rightarrow 0+ } \int_{- \infty}^{ + \infty} \int_{B_1(0)\setminus B_r(0) }  \Delta \Bo^\theta \cdot  \Bp \, dx dy dz
= \int_{- \infty}^{ + \infty} \int_{B_1(0)} ( \mathcal{L} + \partial_z^2)^2 \psi \Be_\theta \cdot  \Bp \, dx dy dz.
\ee

Note that on $\partial B_r(0)$,
\be \nonumber
\frac{\partial \Bo^\theta }{\partial \Bn} = \frac{\partial \omega^\theta}{\partial r} \Be_\theta.
\ee
Thus
\be \nonumber  \ba
& \,\,\left| \int_{-\infty}^{+\infty} \int_{\partial B_r(0)} \frac{\partial \Bo^\theta } {\partial \Bn } \cdot \Bp \, dSdz \right|
 =  \left| \int_{-\infty}^{+\infty} \int_{\partial B_r(0)} \frac{\partial \omega^\theta}{\partial r} \Be_\theta \cdot \Bp \, dSdz \right| \\
= &\,\, \left| \int_{-\infty}^{+\infty} \int_{\partial B_r(0) } \frac{\partial ( \mathcal{L} \psi) }{\partial r} \Be_\theta \cdot \Bp \, dSdz + \int_{-\infty}^{+ \infty} \int_{\partial B_r(0) } \frac{\partial \psi}{\partial r} \Be_\theta \cdot \frac{\partial^2}{\partial z^2} \Bp \, dSdz                \right| \\
 \leq &\,\, C \int_{-Z }^{+Z} \left|  \frac{\partial}{\partial r}( \mathcal{L} \psi )   \right|(r,z)   \int_{\partial B_r(0)} |\Bp |\, dS dz
+ C \int_{- Z }^{+ Z} \left| \frac{\partial \psi}{\partial r}      \right|(r,z)  \int_{\partial B_r(0)} |\partial_z^2 \Bp|\, dSdz \\
 \leq &\,\, C \int_{-Z}^{+ Z} \left( \sup_{\Omega} | \Bp |
\left| r \frac{\partial }{\partial r} (  \mathcal{L} \psi) ( r, z)  \right| + \sup_{\Omega} |\partial_z^2 \Bp | \cdot \left| r \frac{\partial \psi}{\partial r} ( r, z)  \right| \right) \, dz .
\ea \ee
It follows from the proof of Lemma \ref{lemma3-3-1} that for every fixed $z\in \mathbb{R}$,
\be \nonumber
\left| r \frac{\partial }{\partial r} (  \mathcal{L} \psi)(r, z) \right|
\leq C \left( r^{\frac34} + r |\ln r|^{\frac12} \right) \|\psi( \cdot , z)\|_{X_4}
\ee
and
\be \nonumber
\left| r \frac{\partial \psi}{\partial r}(r, z)   \right| \leq \left|  \frac{\partial (r \psi)}{\partial r} (r, z)   \right| + |\psi(r, z)| \leq C \left( r |\ln r|^{\frac12} + r^{\frac34} \right)  \|\psi(\cdot, z) \|_{X_4}.
\ee
Hence,
\be \la{3-4-48}
\lim_{r\rightarrow 0+} \left| \int_{-\infty}^{+\infty} \int_{\partial B_r(0)} \frac{\partial \Bo^\theta } {\partial \Bn } \cdot \Bp \, dSdz \right|   = 0 .
\ee

On the other hand,
\be \la{3-4-52}
\left| \int_{-\infty}^{+\infty} \int_{\partial B_r(0)}
\Bo^\theta \cdot \frac{\partial \Bp}{\partial \Bn} \, dSdz    \right| \leq C \int_{-Z}^{+ Z}
\left( \sup_{\Omega} |\nabla \Bp | r |\mathcal{L}\psi (r, z)|
+ \sup_{\Omega} |\nabla^3 \Bp |  r |\psi(r, z) |   \right)\, dz.
\ee
It follows from the proof of Lemma \ref{lemma3-3-1} that for every fixed $z \in \mathbb{R}$, one has
\be \nonumber
|\mathcal{L}\psi(r ,z) | +
|\psi(r, z)| \leq C r^{\frac34} \|\psi( \cdot , z)\|_{X_4}.
\ee
Hence,
\be \la{3-4-55}
\lim_{r \rightarrow 0+ } \int_{-\infty}^{+ \infty} \int_{\partial B_r(0)} \Bo^\theta
\cdot \frac{\partial \Bp}{\partial \Bn } \, dSdz = 0.
\ee

Collecting \eqref{3-4-41}--\eqref{3-4-55} together gives
\be \nonumber
 \int_{- \infty}^{ + \infty} \int_{B_1(0) } \Bo^\theta \cdot \Delta \Bp \, dx dy dz  =   \int_{-\infty}^{+\infty} \int_{B_1(0) } ( \mathcal{L} + \partial_z^2)^2 \psi   \Be_\theta \cdot \Bp \, dx dy dz.
\ee
This implies that the equation \eqref{3-4-40} holds in $\Omega$ as claimed.
 According to the regularity theory for elliptic equations (\cite{GT}),  for every $0<r_1 < 1$, one has
\be \nonumber
\|\Bo^\theta\|_{H^2(B_{r_1} (0) \times \mathbb{R})} \leq C (r_1) \|( \mathcal{L} + \partial_z^2)^2 \psi \Be_
\theta\|_{L^2(\Omega)} + C(r_1) \|\Bo^\theta\|_{L^2(\Omega)}
\leq C (r_1) \|\psi \|_{H_r^4(D)}.
\ee
This, together with Lemma \ref{lemma3-4-3}, gives \eqref{globalomega} so that the proof of Lemma \ref{lemma3-4-4} is completed.
\end{proof}


{\bf Acknowledgement.}
The research of Wang was partially supported by NSFC grant 11671289. The research of  Xie was partially supported by  NSFC grants 11971307, 11631008, and 11422105,  and Young Changjiang Scholar of Ministry of Education in China. Xie would like to thank the hospitality and support of The Institute of Mathematical Sciences, The Chinese University of Hong Kong where part of the work was done.  The authors would like to thank Professor Maekawa for helpful discussions.

\end{document}